%% file: 3DNMS202107Submitted.tex
\documentclass[11pt,a4paper,english,reqno]{amsart}
\usepackage{amsfonts,amsthm,epsfig,mathrsfs,amssymb}
\usepackage{cite}
\usepackage{epstopdf}
\usepackage[T1]{fontenc}
\usepackage[utf8]{inputenc}
\usepackage{color}
\usepackage{array}
\usepackage{bm}
\usepackage{amstext}
\usepackage{graphicx}
\usepackage[export]{adjustbox}
\usepackage{setspace}
\usepackage{esint}
\usepackage{bbm}
\usepackage{appendix}
\usepackage{enumitem}
\usepackage[hidelinks]{hyperref}
\usepackage[noabbrev,capitalise]{cleveref}
\makeatletter

\numberwithin{equation}{section}
\numberwithin{figure}{section}
\theoremstyle{plain}
\newtheorem{thm}{\protect\theoremname}[section]
  \theoremstyle{plain}
\newtheorem{lem}[thm]{\protect\lemmaname}
\newtheorem{prop}{Proposition}[section]
\theoremstyle{plain}
\newtheorem{rmk}{Remark}[section]
\theoremstyle{plain}
\usepackage{epsfig}
\usepackage{mathrsfs}
\newcommand{\omgrt}{\widetilde{\Omega}_T}
\newcommand{\abs}[1]{\left\vert#1\right\vert}
\makeatletter

\newcommand{\Rmnum}[1]{\expandafter\@slowromancap\romannumeral#1@}
\makeatother
\theoremstyle{remark}

\numberwithin{equation}{section}
\allowdisplaybreaks

\newcommand{\defs}{:=}

\DeclareMathOperator{\tr}{Tr}

\newcommand{\md}{\mathrm{D}}
\DeclareSymbolFont{lettersA}{U}{pxmia}{m}{it}
\DeclareMathSymbol{\piup}{\mathord}{lettersA}{"19}

\input{Denotations_Math_Fonts.tex}

\makeatother
\makeatother
\usepackage{babel}
  \providecommand{\lemmaname}{Lemma}
\providecommand{\theoremname}{Theorem}
\title[Persistence of the steady planar normal shock]{Persistence of the steady planar normal shock structure in 3-D unsteady potential flows}
\begin{document}

\author{Beixiang Fang}
\author{Feimin Huang}
\author{Wei Xiang}
\author{Feng Xiao}

\address{Beixiang Fang: School of Mathematical Science, MOE-LSC, and SHL-MAC, Shanghai Jiao
Tong University, Shanghai 200240, China.}
\email{\textbf{bxfang@sjtu.edu.cn }}
\address{Feimin Huang: Academy of Mathematics and Systems Science, Chinese Academy of Science, Beijing 100190, China; School of Mathematical Sciences, University of Chinese Academy of Sciences, Beijing 100049, China.}
\email{\textbf{fhuang@amt.ac.cn}}
\address{Wei Xiang: Department of Mathematics, City University of Hong Kong, Hong Kong, China.}
\email{\texttt{weixiang@cityu.edu.hk}}
\address{Feng Xiao: NCMIS, Academy of Mathematics and Systems Science, Chinese Academy of Sciences, Beijing 100190, China; Institute of Applied Mathematics, Academy of Mathematics and Systems Science, Chinese Academy of Science, Beijing 100190, China.}
\email{\textbf{xiaofeng@amss.ac.cn}}
\keywords{Dynamic stability; Unsteady Perturbation; Planar Normal Shocks; Artificial Perturbation; Potential Flow Equations; Dihedral Singularity; Hyperbolic Equations}
\subjclass[2000]{35L65, 35L67, 35M10, 35B35, 76H05, 76N10}
\date{\today}
\maketitle
\begin{abstract}
This paper concerns the dynamic stability of the steady 3-D wave structure of a planar normal shock front intersecting perpendicularly to a planar solid wall for unsteady potential flows. The stability problem can be formulated as a free boundary problem of a quasi-linear hyperbolic equation of second order in a dihedral-space domain between the shock front and the solid wall. The key difficulty is brought by the edge singularity of the space domain, the intersection curve between the shock front and the solid wall. Different from the 2-D case, for which the singular part of the boundary is only a point, it is a curve for the 3-D case in this paper. This difference brings new difficulties to the mathematical analysis of the stability problem. A modified partial hodograph transformation is introduced such that the extension technique developed for the 2-D case can be employed to establish the well-posed theory for the initial-boundary value problem of the linearized hyperbolic equation of second order in  a dihedral-space domain. Moreover, the extension technique is improved in this paper such that loss of regularity in the a priori estimates on the shock front does not occur. Thus the classical nonlinear iteration scheme can be constructed to prove the existence of the solution to the stability problem, which shows the dynamic stability of the steady planar normal shock without applying the Nash-Moser iteration method. 
\end{abstract}
%\tableofcontents
\section{Introduction}
%In our previous work \cite{FXX}, it has been shown that the steady normal shock structure is dynamically stable under unsteady perturbation in a 2-D nozzle for the two dimensional potential flow equation. Then it is natural to study the stability of the steady normal shock structure, as shown in figure \ref{fig1}, for the three dimensional unsteady potential flow equation in a straight nozzle.
\subsection{Description of the problem}
This paper concerns the dynamic stability of the steady 3-D wave structure of a planar normal shock front intersecting perpendicularly to a planar solid wall (see Figure \ref{fig1}) for unsteady potential flows.  %Here ``dynamic stability'' means the stability under small perturbations depending on time.
As stated by Courant-Friedrichs in \cite[page 375]{CoF}, ``\textit{Whether or not a flow compatible with the boundary condition occurs depends moreover on its stability}'', it is important and necessary to study the stability of the normal shock structure, namely, whether or not the shock structure will basically maintain as the parameters of the flow fields are slightly perturbed. For steady flows, for which the parameters (density, velocity, pressure, etc.) do not depend on the time variable, there have been plenty of works on the existence and stability of transonic shocks, for instance, see  \cite{BaeXiang,CCXiang2020,CFang,CFeldman2003,CFeldman2004,ChenFang,ChenYuan,Fang,FangLiuYuan2013,FangX,FangXin,LiXuYin,LiXinYin,LiXinYin1,XinYin,YinZhou,Yuan} and the references cited therein. 
%Here ``steady'' means that the parameters of the flow (density, velocity, pressure etc.) do not depend on time.
 As pointed out by von Karman in the discussion chaired by von Neumann and recorded in \cite{Neumann}, a steady motion ``\textit{can occur only as a limiting case}'' of a physical process. Therefore, it is necessary to investigate the unsteady motions associated with the steady planar normal shocks and study their dynamic stability under unsteady perturbations. It has been established the stability of normal shocks, which are far away from physical boundaries, in \cite{Majda1983E,Majda1983S} by Majda for Euler flows, and in \cite{MT1987} by Majda and Thomann for potential flows. See also, for instance, \cite{BenzoniSerre2007,Metivier} and references therein for further studies. However, in practice, shocks often appear together with physical boundaries such as solid walls, wedges, wings, etc.. Therefore, it is important and necessary to further study the stability of shocks involving physical boundaries. In this paper, we are going to study the dynamic stability of the steady 3-D wave structure of a planar normal shock front intersecting perpendicularly to a planar solid wall (see Figure \ref{fig1}), namely, whether the structure will maintain, at least in a short time, under unsteady perturbations of the flow parameters.
%the stability of steady structures under unsteady perturbations absolutely deserves investigation.
In this paper the flows are governed by the unsteady potential flow equations, which read
\begin{equation}\label{potential system}
    \left\{\!\!
    \begin{array}{l}
    \partial_t\rho+\nabla\cdot(\rho \nabla\Phi)=0,\\
    \partial_t\Phi+\frac{1}{2}|\nabla\Phi|^2+\imath(\rho)=B_0,
    \end{array}
    \right.
\end{equation}
where $\nabla\defs (\partial_{x_1},\partial_{x_2},\partial_{x_3})^\top$ is the gradient operator with respect to the space variables $\mathbf{x}\defs(x_1,x_2,x_3)\in \mathbb{R}^3$ and $t>0$ is the time variable. $\imath(\rho)\defs\frac{\rho^{\gamma-1}-1}{\gamma-1} $ is the specific enthalpy, $\Phi$ the velocity potential, $\rho$ the density, $B_0$ the Bernoulli constant, and $\gamma >1$ the adiabatic exponent. The importance of the potential flow equations is first observed by Jacques Hadamard in \cite{Hadamard} for the unsteady Euler equations with weak shocks. Since then, the potential flow equations have been studied by mathematicians steadily, for instance, see Bers \cite{Bers}, Courant-Friedrichs \cite{CoF}, Majda-Thomann \cite{MT1987} and Morawetz \cite{Morawetz}.

By the second equation of $\eqref{potential system}$, one can express the density $\rho$ as a function with respect to $\md\Phi\defs(\partial_t\Phi,\nabla\Phi)$, $B_0$ and $\gamma$, i.e.,
\begin{align}
\rho=\mathfrak{h} (\mathrm{D} \Phi;B_0,\gamma)\defs\left((\gamma-1)(B_0-\partial_t\Phi-\frac12|\nabla\Phi|^2)+1\right)^{\frac{1}{\gamma-1}}.\label{1.2}
\end{align}

Replacing $\rho$ in the first equation of $\eqref{potential system}$ by $\mathfrak{h}(\mathrm{D} \Phi;B_0,\gamma)$, one deduces that $\Phi$ satisfies a hyperbolic equation of second order:
\begin{equation}\label{eq:potential-equations}
\partial_{tt}\Phi+2\sum_{i=1}^3\partial_{x_i}\Phi\partial_{tx_i}\Phi-\sum_{i,j=1}^3(\delta_{ij}c^2-\partial_{x_i}\Phi\partial_{x_j}\Phi)\partial_{x_ix_j}\Phi=0,
\end{equation}
where $c=\sqrt{\rho^{\gamma-1}}$ is the sonic speed and
\begin{equation*}
    \delta_{ij}=\left\{\!\!
    \begin{array}{l}
    1 \quad\mbox{if\ } i=j,\\
    0 \quad\mbox{if\ }i\neq j.\\
    \end{array}
    \right.
\end{equation*}
Let $\Gamma_{\scriptstyle{\textrm{shock}}}:=\{(t,\mathbf{x}):\, x_1=\mathcal{X}(t,x_2,x_3)\}$ be a smooth shock front in the flow field. Then on $\Gamma_{\scriptstyle{\textrm{shock}}}$, the velocity potential $\Phi$ has to satisfy the following Rankine-Hugoniot conditions:
\begin{equation}\label{R-H conditions}
[\Phi]=0\qquad\mbox{and}\qquad \partial_{t}\mathcal{X}[\rho]-[\rho\partial_{x_1}\Phi]+\partial_{x_2}\mathcal{X}[\rho\partial_{x_2}\Phi]+\partial_{x_3}\mathcal{X}[\rho\partial_{x_3}\Phi]=0,
\end{equation}
where the square bracket $[m]$ stands for the jump of the quantity $m$ across the shock front $\Gamma_{\scriptstyle{\textrm{shock}}}$; that is, assuming
\begin{equation}\label{defn R_pm}
\mathcal{R}_{\pm}\defs\{(t,\mathbf{x})\in\mathbb{R}^+\times\mathbb{R}^3 :x_1\gtrless\mathcal{X}(t,x_2,x_3)\}
\end{equation}
and
\[
\mathbf{n}_s\defs-\frac{(\partial_t\mathcal{X},-1,\partial_{x_2}\mathcal{X},\partial_{x_3}\mathcal{X})}{\sqrt{1+|\partial_t\mathcal{X}|^2+|\partial_{x_2}\mathcal{X}|^2+|\partial_{x_3}\mathcal{X}|^2}},
\]
for every $(t,\mathbf{x})\in \Gamma_{\scriptstyle{\textrm{shock}}}$, there exists $\tilde{\alpha}>0$ such that $(t,\mathbf{x})\pm\tau\mathbf{n}_s\in \mathcal{R}_{\pm}$ for any $\tau\in (0,\tilde{\alpha})$, define
\begin{align}
[m](t,\mathbf{x})\defs
\lim_{\substack{\\(\tilde{t},\tilde{\mathbf{x}})\rightarrow (t,\mathbf{x})\\ (\tilde{t},\tilde{\mathbf{x}})\in\mathcal{R}_{+}}}m(\tilde{t},\tilde{\mathbf{x}})-\lim_{\substack{\\(\tilde{t},\tilde{\mathbf{x}})\rightarrow (t,\mathbf{x})\\ (\tilde{t},\tilde{\mathbf{x}})\in\mathcal{R}_{-}}}m(\tilde{t},\tilde{\mathbf{x}}).
\end{align}
It is easy to verify that the Ranking-Hugoniot conditions are equivalent to the following free boundary conditions for $\Phi$:
\begin{equation}
[\Phi]=0 \quad\mbox{and}\quad
[\rho][\partial_t\Phi]+[\partial_{x_1}\Phi][\rho\partial_{x_1}\Phi]+[\partial_{x_2}\Phi][\rho\partial_{x_2}\Phi]+[\partial_{x_3}\Phi][\rho\partial_{x_3}\Phi]=0. \label{eq:potential_bdry_RH-2}
\end{equation}
\vskip 0.2cm
\textbf{The Steady Planar Normal Shock Structure.}
\vskip 0.2cm
A steady planar normal shock solution(see Figure \ref{fig1}) to the potential flow equations \eqref{potential system}, satisfying the Rankine-Hugoniot conditions \eqref{R-H conditions} on the planar shock front, can be easily constructed, which is the reference state in this paper.
\begin{figure}[ht]
	\begin{center}		{\includegraphics[scale=0.6]{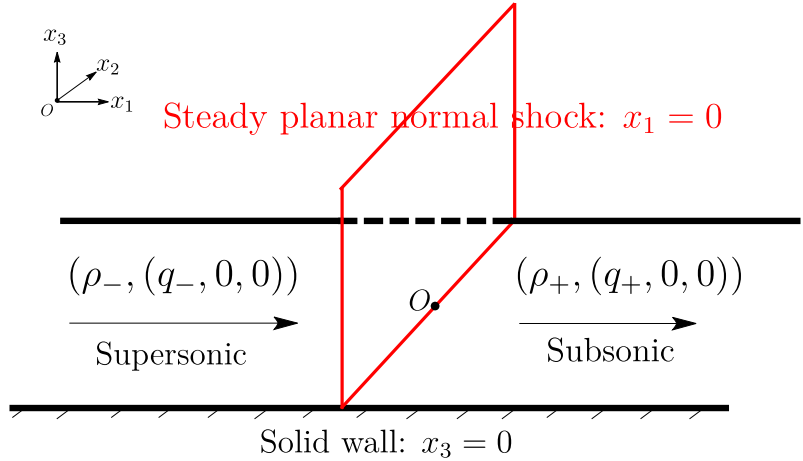}}\caption{The steady planar normal shock structure.}\label{fig1}
	\end{center}
\end{figure}

In Figure \ref{fig1}, the red rectangle stands for a steady planar normal shock front $\{x_1=0\}$ intersecting the solid wall $\{x_3=0\}$ at the edge $\{x_1=0\}\cap\{x_3=0\}$. Constants $\rho_{\pm}$ represent the density of the fluid behind and ahead of the steady planar normal shock, respectively, and $(q_{\pm},0,0)$ are the constant velocities of the flow fields behind and ahead of the steady planar normal shock, respectively.

Now we give a mathematical definition to this steady planar normal shock structure. Denote by $\overline{\Gamma}_0\defs\{\mathbf{x}\in\mathbb{R}^3:\ x_3= 0\} $ the flat solid wall and let $\overline{\Gamma}_{\scriptstyle{\textrm{shock}}}\defs\{\mathbf{x}\in \mathbb{R}^3:\ x_1=\overline{\mathcal{X}}(t,x_2,x_3)\equiv 0\}$ be the position of the steady planar normal shock. The flow field is divided by the normal shock front $\overline{\Gamma}_{\scriptstyle{\textrm{shock}}}$ into two parts $\overline{\mcd}_{-}$ and $\overline{\mcd}_{+}$, which are the regions ahead of and behind the steady shock front $\overline{\Gamma}_{\scriptstyle{\textrm{shock}}}$, respectively, i.e.,
 $$\overline{\mcd}_{\pm}\defs\{\mbx\in\mathbb{R}^3: x_1\gtrless \overline{\mathcal{X}}(t,x_2,x_3),\ x_2\in \mathbb{R},\ x_3>0\}.$$
The constant densities and velocities of the fluid in $\overline{\mcd}_{\pm}$ are given by $(\rho_{\pm},(q_{\pm},0,0))$, respectively. Then $\rho_{\pm}$ are determined by $q_{\pm}$ via $\eqref{1.2}$, \emph{i.e.},
\begin{align}
	\rho_{\pm}=\mathfrak{h}((0,q_{\pm},0,0);B_0,\gamma)=\left((\gamma-1)\left(B_0-\frac{1}{2}q_{\pm}^2\right)+1\right)^{\frac{1}{\gamma -1}}.
\end{align}
 Let $\overline{\Phi}(t,\mathbf{x})$ be defined as
\begin{equation}\label{2.7}
\overline{\Phi}(t,\mathbf{x})=
\begin{cases}
\overline{\Phi}_-(t,\mbx)\defs q_-\cdot x_1\quad\mbox{for }(t,\mbx)\in \mathbb{R}_+\times\overline{\mcd}_-,\\
\overline{\Phi}_+(t,\mbx)\defs q_+\cdot x_1\quad\mbox{for }(t,\mbx)\in \mathbb{R}_+\times\overline{\mcd}_+.
\end{cases}
\end{equation}

Then it is easy to see that $\overline{\Phi}(t,\mathbf{x})$ satisfies \eqref{eq:potential-equations} in the two regions $\overline{\mcd}_-$ and $\overline{\mcd}_+$. Moreover, it satisfies
\vspace{-0.1cm}
\begin{equation}
\nabla\overline{\Phi}(t,\mathbf{x})=
\begin{cases}
(q_-,0,0)\quad\mbox{for }(t,\mbx)\in \mathbb{R}_+\times\overline{\mcd}_-,\\
(q_+,0,0)\quad\mbox{for }(t,\mbx)\in \mathbb{R}_+\times\overline{\mcd}_+.
\end{cases}
\end{equation}
Thus $\overline{\Phi}(t,\mathbf{x})$ is a velocity potential of the flow field above the solid wall $\overline{\Gamma}_0$. Due to the Rankine-Hugoniot conditions \eqref{R-H conditions} (or equivalently \eqref{eq:potential_bdry_RH-2}) and the entropy condition, constants $(\rho_-,\rho_+,q_-,q_+)$ must satisfy
\begin{equation}\label{2.8}
\rho_-<\rho_+,\quad \rho_-q_-=\rho_+q_+, \quad\mbox{and}\quad \frac{q^2_-}{2}+\imath(\rho_-)=\frac{q_+^2}{2}+\imath(\rho_+).
\end{equation}
%which are known as the entropy condition, conservation of mass and Bernoulli's law, respectively.
The steady planar normal shock $\overline{\Gamma}_{\scriptstyle{\textrm{shock}}}$ is a transonic shock: ahead of the shock front $\overline{\Gamma}_{\scriptstyle{\textrm{shock}}}$, the uniform coming flow $(\rho_-,(q_-,0,0))$ is supersonic and behind the shock front $\overline{\Gamma}_{\scriptstyle{\textrm{shock}}}$, the flow $(\rho_+,(q_+,0,0))$ is subsonic, \emph{i.e.},
\begin{equation}\label{2.9}
q_-^2>c_-^2=\rho_-^{\gamma-1}\quad \mbox{and}\quad q_+^2<c_+^2=\rho_+^{\gamma -1}.
\end{equation}
Then the triplet $(\overline{\Phi}(t,\mathbf{x}),\overline{\Gamma}_{\scriptstyle{\textrm{shock}}},\overline{\Gamma}_0)$ is called the steady planar normal shock structure, which will be the reference state investigated in this paper. The steady planar normal shock structure can be observed in many situations. For example, if a normal shock appears in a nozzle with flat boundary (for instance the nozzle with rectangular cross-section), then this kind of normal shock coincides locally with the steady planar normal shock structure in Figure \ref{fig1}.

\subsection{Mathematical formulation}
The theme of this paper is to study the dynamic stability of the steady planar normal shock structure $(\overline{\Phi}(t,\mathbf{x}),\overline{\Gamma}_{\scriptstyle{\textrm{shock}}},\overline{\Gamma}_0)$, in the framework of unsteady potential flow equation \eqref{eq:potential-equations}. We want to know whether or not the steady planar normal shock structure persists, at least for a short time, when the uniform supersonic coming flow $(\rho_-,(q_-,0,0))$ is perturbed a little unsteadily and the flat solid wall $\overline{\Gamma}_0$ becomes slightly curved.
% In figure \ref{fig1}, the fluid is assumed to have constant density and velocity and the nozzle is supposed to be bounded by two flat nozzle walls $x_3=0$ and $x_3=1$. The dashed vertical line in the nozzle stands for the steady normal shock. $\rho_+$ and $\rho_-$ represents the constant density of the fluid ahead of and behind the normal shock respectively, while $(q_+,0,0)$ and $(q_-,0,0)$ stands for the constant velocity of the fluid ahead of and behind the normal shock respectively. The detailed definition of the steady normal shock structure is referred to \eqref{2.7}-\eqref{extra condition} in section \ref{sec:2}.
%In this paper, we will show that the steady planar normal shock structure $(\overline{\Phi},\overline{\chi},\overline{\mathcal{W}})$ is dynamically stable for three dimensional unsteady potential flows.
Let $\mcw(x_1,x_2)$ be a smooth function. %satisfying $\partial_{x_2}\mcw(x_1,x_2)|_{x_2=0}=0$. %(This assumption is reasonable, since the singularity of the edge is much stronger than a single point singularity in \cite{FXX}, as we will explain in detail later).
We denote by $\Gamma_0:=\{(t,\mathbf{x}):  x_3=\mcw(x_1,x_2)\}$
an impermeable solid boundary of the flow field. Then the whole flow field is
 $$\mathcal{D}\defs\{\mbx\in\mathbb{R}^3: x_3>\mcw(x_1,x_2)\}.$$ 
$\Phi$ satisfies the slip boundary condition 
$\nabla \Phi\cdot \mathbf{n}=0$ on $\Gamma_0$, where $\mathbf{n}$ is the unit exterior normal vector of $\Gamma_0$, \emph{i.e.},
\begin{equation}
\label{eq:potential_bdry_wedge}
-\partial_{x_1}\Phi\partial_{x_1}\mcw-\partial_{x_2}\Phi\partial_{x_2}\mcw+\partial_{x_3}\Phi=0\quad\mbox{on }\quad\Gamma_0.
\end{equation}
Moreover, let the initial states of the fluid be also slightly perturbed such that the initial conditions for $\Phi$ are given as:
\begin{equation}\label{IBC}
\Phi(0,\mathbf{x})=\Phi_0(\mathbf{x})\qquad\mbox{and}\qquad\partial_t\Phi(0,\mathbf{x})=\Phi_1(\mathbf{x}),
\end{equation}
where for $i=0,1$,
\begin{equation}\label{defn of IBC}
\Phi_i(\mathbf{x})\defs
\begin{cases}
\Phi_i^+(\mathbf{x})\quad&\mbox{for }\mbx\in\mathcal{R}^0_+\defs\{x_1>\mathcal{X}(0,x_2,x_3)\}\cap \mcd,\\
\Phi_i^-(\mathbf{x})\quad&\mbox{for }\mbx\in\mathcal{R}^0_-\defs\{x_1<\mathcal{X}(0,x_2,x_3)\}\cap\mcd.
\end{cases}
\end{equation}
Here the initial position $\mathcal{X}(0,x_2,x_3)$ of the perturbed shock front $\Gamma_{\scriptstyle{\textrm{shock}}}$ is a small perturbation of the reference shock front $\overline{\Gamma}_{\scriptstyle{\textrm{shock}}}$.
\vskip 0.2cm
Now the dynamic stability problem (see Figure \ref{fig2}) can be precisely reformulated as following problem:
\vskip 0.2cm
\textbf{Problem 1:} Suppose $\Gamma_{0}$ is a small perturbation of $\overline{\Gamma}_0$, \emph{i.e.}, $\mathcal{W}$ is close to zero and the initial data $(\Phi_0,\Phi_1)$ are small perturbations of $\overline{\Phi}(0,\mathbf{x})$, \emph{i.e.}, $\Phi_0$ is close to $\overline{\mathbf\Phi}(0,{x})$ and $\Phi_1(\mathbf{x})$ is close to zero. One looks for a unique local piece-wise smooth solution $(\Phi(t,\mathbf{x}),\mathcal{X}(t,x_2,x_3))$ to equation \eqref{eq:potential-equations} in the flow field $\mathcal{D}=\{\mbx\in\mathbb{R}^3: x_3>\mcw(x_1,x_3)\}$ such that:
\begin{enumerate}
\item The shock front is given by 
	\[\Gamma_{\scriptstyle{\textrm{shock}}}\defs\{(t,\mbx)\in\mathbb{R}^+\times{\mathbb{R}^3}:x_1=\mathcal{X}(t,x_2,x_3)\},\] which divides the flow field into  $\mathcal{D}_{+}\defs\mathcal{D}\cap\mathcal{R}_{+}$ and $\mathcal{D}_{-}\defs\mathcal{D}\cap\mathcal{R}_{-}$, where $\mathcal{R}_{\pm}$ are defined in \eqref{defn R_pm}.
	
\item $\Phi(t,\mbx)$ is smooth up to either sides of $\Gamma_{\scriptstyle{\textrm{shock}}}$ such that
	 \begin{equation*}
	\Phi(t,\mathbf{x})=
	\begin{cases}
	\Phi^+(t,\mbx) \mbox{\ for\ } (t,\mbx)\in \mathcal{D}_+, \\
\Phi^-(t,\mbx) \mbox{\ for\ } (t,\mbx)\in \mathcal{D}_-,
	\end{cases}
	\end{equation*}
	and $\Phi^{\pm}(t,\mbx)$ satisfy equation \eqref{eq:potential-equations} in $\mathcal{D}_{\pm}$, respectively.
	
\item $\Phi^{\pm}(t,\mbx)$ satisfy the slip boundary condition \eqref{eq:potential_bdry_wedge}, respectively, i.e,
\[-\partial_{x_1}\Phi^\pm\partial_{x_1}\mcw-\partial_{x_2}\Phi^\pm\partial_{x_2}\mcw+\partial_{x_3}\Phi^\pm=0{\quad}\mbox{for\ }(t,\mbx)\in\mathbb{R}^+\times\Gamma_0.\]

\item $\Phi^{\pm}(t,\mbx)$ satisfy the initial conditions \eqref{IBC}-\eqref{defn of IBC}, respectively, i.e.,
\[\Phi^\pm(t,\mbx)|_{t=0}=\Phi^\pm_0(\mbx)\quad\mbox{\ for\ }\mbx\in\mathcal{R}^0_{\pm}\cap\mathcal{D},\]
and 
\[\partial_t\Phi^\pm(t,\mbx)|_{t=0}=\Phi^\pm_1(\mbx)\!\!\!\quad\mbox{\ for\ }\mbx\in\mathcal{R}^0_{\pm}\cap\mathcal{D},\]
where $\mathcal{R}^0_{\pm}$ are the ones defined in \eqref{defn of IBC}.
\item $(\Phi^+(t,\mbx),\Phi^-(t,\mbx),\mathcal{X}(t,x_2,x_3))$ satisfy the Rankine-Hugoniot conditions in \eqref{R-H conditions}.

\item $(\Phi(t,\mbx),\mathcal{X}(t,x_2,x_3))$ is close to the steady normal shock solution $(\overline{\Phi},\overline{\mathcal{X}})$, \emph{i.e.},
$\Phi^{\pm}(t,\mbx)$ is close to
$\overline{\Phi}_{\pm}(t,\mbx)$ in $\mcd_{\pm}$,  respectively, and $\mathcal{X}(t,x_2,x_3)$ is close to $\overline{\mathcal{X}}(t,x_2,x_3)$.
\end{enumerate}
% $\mathcal{X}(t,x_2,x_3)$ is close to the steady normal shock front $\overline{\mathcal{X}}(t,x_2,x_3)$ and the perturbed velocity potential $\Phi(t,\mathbf{x})$ is close to the background potential function $\overline{\Phi}(t,\mbx)$ in the flow field $\mathcal{D}_+\cup\mathcal{D}_-$, i.e., close to $\overline{\Phi}_{+}(t,\mbx)$ in  $\mathcal{D}_{+}$ and $\overline{\Phi}_-(t,\mbx)$ in $\mathcal{D}_{-}$; and $\Phi(t,\mbx)$ satisfies the Rankine-Hugoniot conditions $\eqref{eq:potential_bdry_RH-2}$ on $\Gamma_{\scriptstyle{\textrm{shock}}}$, the slip boundary condition $\eqref{eq:potential_bdry_wedge}$ on $\Gamma_{0}$, as well as the initial conditions $\eqref{IBC}$ and $\eqref{defn of IBC}$ on $\{t=0\}$?

\begin{figure}[ht]
	\begin{center}{\includegraphics[scale=0.7]{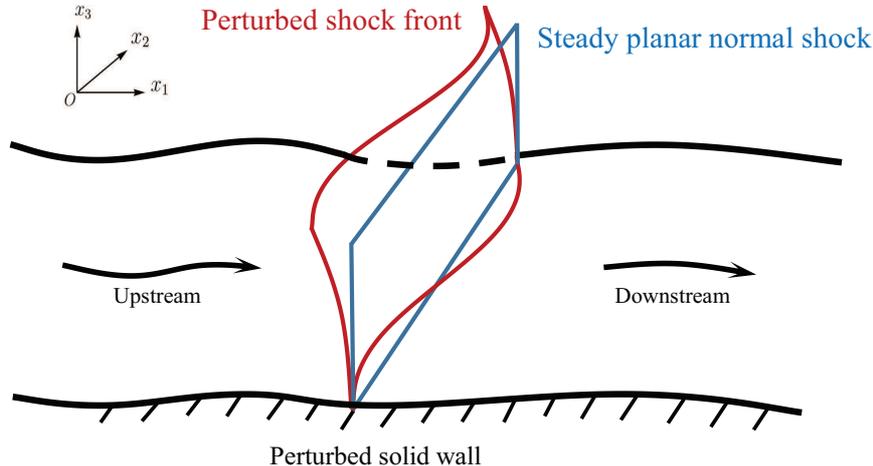}}
		\caption{Persistence of the steady planar normal shock under perturbation.}\label{fig2}
	\end{center}
\end{figure}
\begin{rmk}
Thanks to the  property of the finite speed of propagation of hyperbolic equations and the well-established mathematical theory for initial boundary value problems for hyperbolic equations with smooth boundaries (for instance, see \cite{BenzoniSerre2007}), one can assume that, without loss of generality, the perturbation only occurs near the intersection curve, where the shock front intersects the solid wall $x_3=\mathcal{W}(x_1,x_2)$. Therefore, this paper only solves the stability problem near the edge of the dihedral-space domain, and in a short time.
\end{rmk}
%The theme of current paper is to consider following problem:

%\textit{Without any symmetry assumptions on the flow, is the steady normal shock structure still dynamically stable for the three dimensional unsteady potential flow equation in a straight nozzle with arbitrary cross-section? }

%In this paper, we will give a positive answer to this question under the assumption that the flow is irrotational and compressible. The motion of the compressible flow is governed by
%the three dimensional unsteady isentropic Euler system, which reads
%\begin{equation}\label{Euler system}
%    \left\{\!\!
%    \begin{array}{l}
%    \partial_t\rho+\nabla\cdot(\rho \mathbf{u})=0,\\
%    \partial_t(\rho\mathbf{u})+\nabla\cdot(\rho\mathbf{u}\otimes\mathbf{u})+\nabla \mathbf{P}=0.
%    \end{array}
%    \right.
%\end{equation}
%Here $\nabla\defs (\partial_{x_1},\partial_{x_2},\partial_{x_3})$ is the gradient in $\mathbf{x}=(x_1,x_2,x_3)\in\mathbb{R}^3$, $\mathbf{u}=(u_1,u_2,u_3)$ the fluid velocity, $\rho$ the density, $\mathbf{P}$ the pressure. The pressure $\mathbf{P}$ is given by the equation of state $\mathbf{P}=\frac{1}{\gamma}\rho^\gamma$, where $\gamma>1$ is the adiabatic exponent.
%For irrotational flow, it is easy to know that the fluid velocity can be expressed as the gradient of some potential function $\Phi=\Phi(t,\mathbf{x})$, i.e., $\mathbf{u}=\nabla\Phi(t,\mathbf{x})$ (This is why it is called potential flow). Then system $\eqref{Euler system}$ becomes
%Before going on, we would like to make some comments on this problem and introduce some related studies.

The initial boundary value problem \eqref{eq:potential-equations}, \eqref{eq:potential_bdry_RH-2}, and \eqref{eq:potential_bdry_wedge}-\eqref{defn of IBC} is a free boundary problem in a dihedral-space domain between two surfaces, %formed by two hyperplanes,
the shock front $\Gamma_{\scriptstyle{\textrm{shock}}}$ and the perturbed solid wall $\Gamma_0$. %One of the main difficulties for the mathematical analysis of \eqref{eq:potential-equations}, \eqref{eq:potential_bdry_wedge} and \eqref{eq:potential_bdry_RH-2}-\eqref{defn of IBC} is that the shock front $\Gamma_{\scriptstyle{\textrm{shock}}}$ is unknown before the initial boundary value problem is solved and it should be determined as a part of the solution. Another difficulty comes from the singularity of the boundary of the space domain: it is not smooth along the edge of the dihedral space. %the well-posedness of initial boundary value problem of hyperbolic equations in dihedral-space domains.
The key difficulty in the mathematical analysis of the problem comes from the singularity of the boundary of the space domain, which is not smooth along the edge of the dihedral-space domain, especially as it couples with other difficulties such as nonlinearity, free boundaries, etc.
In fact, Osher has given examples in \cite{Osher,Osher1} showing that hyperbolic equations in cornered space domain may be ill-posed.
On the other hand, for the well-posedness problem of hyperbolic equations in space-domains with non-smooth boundaries, there are also positive results, for instance, see \cite{GS,Godin,GodinMAMS,Yuan}. In particular, under certain symmetry assumptions, Gazzola-Secchi \cite{GS} studied the inflow-outflow problem in a bounded cylinder. Then Yuan \cite{Yuan} studied the dynamic stability of normal shock in a duct with flat boundaries in two space dimensions. In both works, the symmetry assumptions play an essential role in the analysis, under which the extension techniques can be employed such that the non-smooth domain is reduced into a smooth domain. %In \cite{Yuan}, the nozzle is kept flat and the symmetry assumptions that even or odd traces of given data vanish on the flat nozzle walls are needed. Actually, these symmetry assumptions together with the assumption that the nozzle is flat are crucial to his work. Since only when all these assumptions hold, can the author extend the solution from the nozzle to the whole domain periodically.
Such symmetry assumptions fail to be valid in the problem \eqref{eq:potential-equations}, \eqref{eq:potential_bdry_RH-2}, and \eqref{eq:potential_bdry_wedge}-\eqref{defn of IBC} studied in this paper, since the solid wall $\Gamma_0$ is a curved surface. Hence the methods developed in \cite{GS,Yuan} are not applicable.
Nevertheless, the assumption that $\Gamma_0$ is a slightly perturbed surface from a flat one implies that there may hold some symmetry properties under certain transformation. Recently in \cite{FXX}, the authors develop an extension technique successfully to deal with the difficulty in a 2-D cornered-space domain. However, the technique cannot be directly applied to the problem in this paper because the singular set of the boundary is no longer a single point, but a curve, which is the edge. Therefore, new methods should be developed and more careful analysis are needed to establish the well-posedness of the solutions in the dihedral-space domain.

Motivated by the extension techniques developed in \cite{FXX} for 2-D case, we shall look for an appropriate transformation, under which it is possible to extend the linearized initial-boundary value problem in the dihedral-space domain into an initial-boundary value problem in the half-space domain. To make it, a modified partial hodograph transformation (see \eqref{partial hodograph trans} for details), different from the transformation employed for the 2-D case, is introduced. Then the problem in the dihedral-space domain will be extended into a problem in a half-space domain, and the unique existence of a  $H^2_\eta$-solution (a weighted Sobolev space) can be established by employing the classical theory for initial-boundary value problems of hyperbolic equations (see \cite{BenzoniSerre2007}, for instance). Similar to the 2-D case, the $H^2_\eta$ regularity is not sufficient to close the nonlinear iteration. Therefore, \emph{a priori} estimates for higher order derivatives are required, which should be established directly in the dihedral-space domain, since the extended coefficients are of low regularity.  Moreover, as the space dimension increases, the analysis needed for the \emph{a priori} estimates for higher order derivatives is more complicated than the 2-D case and it should be dealt with more carefully. Finally, it is worth mentioning that a transformation (see section \ref{sec:4}) is introduced to reformulate the nonlinear problem, which helps to improve the extension argument develop in \cite{FXX}, such that the loss-of-regularity for the \emph{a priori} estimates on the shock-front will not occur. Hence instead of the Nash-Moser iteration scheme employed in \cite{FXX}, a classical nonlinear iteration scheme is sufficient to prove the existence of the solutions to the nonlinear problem.
%Moreover, rarefaction waves and contact
% discontinuity are also studied a lot, for instance one can refer \cite{CHWX,ChengDX,HKWX} for the steady contact discontinuity in a nozzle and \cite{Alinhac,QuXiang,WY} for the rarefaction wave.

Up to now, much great progress has been made in the study of weak solutions of multidimensional unsteady compressible Euler equations. For instance, see \cite{Colmbel2002,Colmbel2004,LaiXiangZhou,LiWittYin1,Majda1983E,Majda1983S,MT1987,Metivier,Metivierbook} for the study of shock waves, \cite{Alinhac1989,AlinhacCPDE,ChenChenJHDE} for rarefaction waves, \cite{ChenSecchiWang2019,ChenSecchiWang2020,CouSecchi2004,CouSecchi} for contact discontinuities, \cite{BaeChenFeldmanInvenM,BaeChenFeldmanMAMS,ChenFeldmanAnnals,ChenFeldmanBOOK,CFeldmanX,Chen5,EL,ZhengLiARMA,ZhengLiCMP} for self-similar solutions, and \cite{Al,Bardos,ChiKreml,ChiDeLellisCPAM,DeSz2009,DeSz2010,Klingenberg2018,LuoXieXinAdv} for the non-uniqueness of weak solutions. %for multidimensional Euler equatios.

The remainder of the paper is organized as follows. In section \ref{sec:2}, a modified partial hodograph transformation is introduced to fix the free boundary and flat the curved solid wall. Then the dynamic stability problem is reformulated as the well-posedness problem of an initial boundary value problem for a nonlinear hyperbolic equation of second order, in a dihedral-space domain with fixed boundaries. Finally, the main theorem, theorem \ref{wellposedness of the NLP}, is presented at the end of this section. In section \ref{sec:3}, we obtain the well-posedness of a general initial boundary value problem for a linear hyperbolic equation of second order in the dihedral-space domain. In section $\ref{sec:4}$, the nonlinear problem $\eqref{NLP}$ is reformulated. In section $\ref{sec:5}$, an iteration scheme is introduced to solve the reformulated nonlinear problem. Then one proves the main theorem by showing that the iteration scheme provides a sequence of functions which converges to the desired solution, and hence prove the dynamic stability of the steady planar normal shock structure.

\section{Partial hodograph transformation and main result}\label{sec:2}

In this section, we introduce a modified partial hodograph transformation, which is used to fix the free boundary $\Gamma_{\scriptstyle{\textrm{shock}}}$ and straighten the perturbed solid wall $\Gamma_0$. With the aid of this transformation, the previous initial boundary value problem  \eqref{eq:potential-equations}, \eqref{eq:potential_bdry_RH-2}, and \eqref{eq:potential_bdry_wedge}-\eqref{defn of IBC} is mapped to an initial boundary value problem in a dihedral-space domain with fixed boundaries in the new coordinate system. Then Problem 1 is converted to Problem 2 and solving Problem 1 is equivalent to solve Problem 2. Finally, at the end of this section, we present our main result.
%\subsection{Formulation of the initial boundary value problem}

\subsection{Partial hodograph transformation}
Let $ \Phi^{-} $ be the potential for the flow field ahead of the shock-front and $ \Phi$ the one behind the shock-front. Extend $\Phi^{-}$ by solving the equation \eqref{eq:potential-equations} with the boundary condition \eqref{eq:potential_bdry_wedge} into the domain ahead of the shock-front, which is at least $C^1$ across the shock-front. More precisely, first we extend $\Phi_0^-(\mathbf{x})$ and $\Phi_1^-(\mathbf{x})$ smoothly into the whole domain $\mathbb{R}^3$. Then solve the initial boundary value problem \eqref{eq:potential-equations}, \eqref{eq:potential_bdry_wedge}, and \eqref{IBC}, where $\Phi_i(\mathbf{x})$ in \eqref{IBC} is replaced by $\Phi_i^-(\mathbf{x})$. Obviously, such solution exists locally (this is reasonable, one can see \cite{ChenSX} for the case of compressible Euler equations, which includes the case of potential flows) and is a solution of Problem 1 when $x_1<\mathcal{X}(t,x_2,x_3)$. Denote by $\Phi^-(t,\mathbf{x})$ this smooth solution and define
\begin{equation}\label{eq:potential_hodograph}
	\phi(t,\mathbf{x})\defs \Phi^{-}(t,\mathbf{x}) - \Phi(t,\mathbf{x}).
\end{equation}
Then the potential equation \eqref{eq:potential-equations} for $ \Phi$ is reformulated as a second order equation for $ \phi $:
\begin{equation}
\label{eq:potential_eq_for_phi}
\sum_{i,j=0}^{3}a_{ij}(\mathrm{D}\phi;\mathrm{D}\Phi^{-})\partial_{x_ix_j}\phi = \sum_{i,j=0}^{3}a_{ij}(\mathrm{D}\phi;\mathrm{D}\Phi^{-})\partial_{x_i x_j}\Phi^{-},
%\phi_{tt}+2\sum_{i=1}^2\Phi_{x_i}\Phi_t+\sum_{i,j=1}^2(\delta_{ij}c^2-\Phi_{x_i}\Phi_{x_j})\Phi_{x_ix_j}=0,
\end{equation}
where
\begin{equation}\label{eq:potential_eq_coefficients1}
	a_{00} = 1,\mbox{\ }a_{0j} =a_{j0}\defs\partial_{x_j}\Phi^{-} - \partial_{x_j}\phi=\partial_{x_j}\Phi,
\end{equation}
and
\begin{equation}\label{eq:potential_eq_coefficients2}
	a_{ij} =a_{ji}\defs -c^2\delta_{ij}+(\partial_{x_i}\Phi^{-} - \partial_{x_i}\phi)(\partial_{x_j}\Phi^{-} - \partial_{x_j}\phi)=-c^2\delta_{ij}+\partial_{x_i}\Phi\partial_{x_j}\Phi
\end{equation}
for $i,j=1,2,3$.

We introduce the following partial hodograph transformation:
\begin{equation}\label{partial hodograph trans}
    \mfp:\left\{
    \begin{array}{ll}
    \!\!y_0=t\\
    \!\!y_1=\phi(t,\mathbf{x})\\
    \!\!y_2=x_2+p(\mathbf{x})\\
    \!\!y_3=x_3-\mcw(x_1,x_2)
\end{array}
\right.
\end{equation}
%\[
%\mfp(t,x_{1},x_{2},x_3)=(y_0,\mathbf{y})=(y_{0},y_{1},y_{2},y_3)^\top\defs(t,\phi(t,\mathbf{x}),x_2+\lambda(\mathbf{x}), x_3-\mcw(x_1,x_2))^\top,
%\]
where
\begin{equation}
    p(\mbx)=\frac{\partial_{x_2}\mcw}{1+|\partial_{x_1}\mcw|^2+|\partial_{x_2}\mcw|^2}(x_3-\mcw(x_1,x_2)).
\end{equation}
Here $p(\mbx)$ is introduced to balance the perturbation on the $x_2$-direction.
\begin{rmk}
In \cite{FXX}, $p(\mathbf{x})$ does not appear in the partial hodograph transformation. While in this paper, $p(\mathbf{x})$ plays an essential role, as it is used to match the perturbations on the $x_2$-direction and $x_3$-direction. As one will see from the proof of lemma \ref{lem 5.1}, the appearance of $p(\mathbf{x})$ guarantees the vanishing property of $\tilde{a}_{23}$ and $\tilde{a}_{32}$ on $\{y_3=0\}$, which is necessary to the application of the extension technique and crucial to the solvability of the linearized problem in the dihedral-space domain.
\end{rmk}
\vspace{-0.3cm}
The inverse of $\mfp$ is
\begin{equation}
    \mfp^{-1}:{\ }t=y_0,{\ }x_1=u(y_0,\mathbf{y}),{\ }
    x_2=x_2(y_0,\mathbf{y}),{\ }
    x_3=y_3+\mcw(u(y_0,\mathbf{y}),x_2(y_0,\mathbf{y})),
\end{equation}
where $(y_0,\mathbf{y}):=(y_0,y_1,y_2,y_3)$ are the time-spatial variables in the new coordinate and $u(y_0,\mby)$ is the new unknown function. Taking the partial derivatives to the equation
$
y_{1}=\phi\circ\mfp^{-1}(y_0,\mathbf{y})
$
with respect to $y_{j}\mbox{\ }(j=0,1,2,3)$, we obtain a linear system with respect to $\mathrm{D}_{t,\mathbf{x}}\phi\defs(\partial_t\phi,\nabla\phi)$. By solving this system, one can express $\mathrm{D}_{t,\mathbf{x}}\phi$ in terms of $\md u\defs(\partial_{y_0}u,\partial_{y_1}u,\partial_{y_2}u,\partial_{y_3}u)$,
%\begin{equation}
%\begin{cases}
%\phi_t&=-\frac{\partial_{y_0}u}{\partial_{y_1}u}\quad\phi_{x_1}=-\frac{\partial_{x_1}p\partial_{y_2}u-\partial_{y_3}u\partial_{x_1}\mcw-1}{\partial_{y_1}u}\\
%\phi_{x_2}&=\frac{\partial_{x_3}p\partial_{y_2}u\partial_{x_2}\mcw+\partial_{y_3}u\partial_{x_2}\mcw-\partial_{y_2}u}{\partial_{y_1}u}
%\end{cases}
%\end{equation}
\begin{equation}\label{Dphi in terms of Du}
    \left\{\!\!
    \begin{array}{l}
     \partial_t\phi=-\dfrac{\partial_{y_0}u}{\partial_{y_1}u},\vspace{1ex}
     \\
    \partial_{x_1}\phi=-\dfrac{\partial_{x_1}p\partial_{y_2}u-\partial_{x_1}\mcw\partial_{y_3}u-1}{\partial_{y_1}u},\vspace{1ex}
     \\
    \partial_{x_2}\phi=\dfrac{\partial_{x_3}p\partial_{x_2}\mcw\partial_{y_2}u+\partial_{x_2}\mcw\partial_{y_3}u-\partial_{y_2}u}{\partial_{y_1}u},\vspace{1ex}
     \\
    \partial_{x_3}\phi=-\dfrac{\partial_{x_3}p\partial_{y_2}u+\partial_{y_3}u}{\partial_{y_1}u}.
\end{array}
    \right.
\end{equation}

%\begin{align*}
%    \phi_t&=-\frac{\partial_{y_0}u}{\partial_{y_1}u},\\
%    \phi_{x_1}&=-\frac{\partial_{x_1}p\partial_{y_2}u-\partial_{y_3}u\partial_{x_1}\mcw-1}{\partial_{y_1}u},\\
%    \phi_{x_2}&=\frac{\partial_{x_3}p\partial_{y_2}u\partial_{x_2}\mcw+\partial_{y_3}u\partial_{x_2}\mcw-\partial_{y_2}u}{\partial_{y_1}u},\\
%    \phi_{x_3}&=-\frac{\partial_{x_3}p\partial_{y_2}u+\partial_{y_3}u}{\partial_{y_1}u}.
%\end{align*}
%We remark that actually we can deduce $x_2=x_2(u(y_0,\mathbf{y}),y_2,y_3)$ by the definition of $\mfp$ and the implicit function theorem. Then $\frac{\partial x_2}{\partial y_j}$ can be obtained by differentiating $y_2=x_2+p(\mathbf{x})$ with respect to $y_i$ for $i=0,1,2,3$.
%So
%\[
%\begin{split}
%	\mathrm{D}_{t,\mathbf{x}}\phi
%		 =\frac{1}{\partial_{y_{1}}u}\mathbf{v},
%\end{split}
%\]
%where $\mathbf{v}=\left(-\partial_{y_{0}}u, -\partial_{x_1}p\partial_{y_2}u+\partial_{x_1}\mcw\partial_{y_3}u+1,\partial_{x_3}p\partial_{y_2}u\partial_{x_2}\mcw+\partial_{x_2}\mcw\partial_{y_3}u-\partial_{y_2}u, -\partial_{x_3}p\partial_{y_2}u-\partial_{y_3}u\right)$.\\
The Jacobi matrix of $\mfp$ is
\[
\frac{\partial(y_0,\mathbf{y})}{\partial(t,\mathbf{x})} =
\begin{bmatrix}
1 & 0 & 0 & 0\\
\partial_t\phi & \partial_{x_1}\phi & \partial_{x_2}\phi &\partial_{x_3}\phi\\
 0&\partial_{x_1}p&1+\partial_{x_2}p&\partial_{x_3}p\\
0 & -\partial_{x_1}\mcw & -\partial_{x_2}\mcw&1   \\
\end{bmatrix}
\defs\frac{1}{\partial_{y_1}u}\mathbf{J}^{\top},
\]
where
\[
\mathbf{J}\defs
\begin{bmatrix}
\partial_{y_1}u&-\partial_{y_0}u&0&0  \\
 0 &\partial_{x_1}\mcw\partial_{y_3}u-\partial_{x_1}p\partial_{y_2}u+1& \partial_{x_1}p\partial_{y_1}u&-\partial_{x_1}\mcw\partial_{y_1}u  &\\
0 & (\partial_{x_3}p\partial_{x_2}\mcw-1)\partial_{y_2}u+\partial_{y_3}u\partial_{x_2}\mcw & (\partial_{x_2}p+1)\partial_{y_1}u&  -\partial_{x_2}\mcw\partial_{y_1}u\\
0&-\partial_{x_3}p\partial_{y_2}u-\partial_{y_3}u&\partial_{x_3}p\partial_{y_1}u&\partial_{y_1}u
\end{bmatrix}.
\]

\subsection{Formulation in new coordinate}
In the remaining part of this paper, time $t$ may be denoted by $y_0$ and vice versa.
After a direct computation, we also obtain
\begin{eqnarray*}
	\frac{\partial\left(\mathrm{D}\phi\right)}{\partial\left(\mathrm{D}u\right)} =-\frac{1}{\left(\partial_{y_1}u\right)^{2}}\mathbf{J}.
\end{eqnarray*}

Denote by $\mathrm{D}^{2}\phi$ the Hessian matrix of $\phi$, i.e.,
\[
\mathrm{D}^{2}\phi = \frac{\partial(\mathrm{D}\phi)}{\partial(t,\mathbf{x})}.
\]
With the help of \eqref{Dphi in terms of Du}, by simple calculation, one has
\begin{align*}
	\mathrm{D}^{2}\phi
	=& \frac{\partial(\mathrm{D}\phi)}{\partial(\mathrm{D}u)}
	\begin{bmatrix}
		\partial_{y_{i}y_{j}}u
	\end{bmatrix}_{4\times 4}
	\frac{\partial(y_0,\mathbf{y})}{\partial(t,\mathbf{x})}+\frac{(-\partial_{x_1x_1}p\partial_{y_2}u+\partial_{x_1x_1}\mcw\partial_{y_3}u)\mathbf{I}_{11}}{\partial_{y_1}u}\\
	&+\frac{(-\partial_{x_1x_2}p\partial_{y_2}u+\partial_{x_1x_2}\mcw\partial_{y_3}u)\mathbf{I}_{12}}{\partial_{y_1}u}-\frac{\partial_{x_1x_3}p\partial_{y_2}u\mathbf{I}_{13}}{\partial_{y_1}u}\\
	&+\frac{1}{\partial_{y_1}u}((\partial_{x_1x_3}p\partial_{y_2}u\partial_{x_2}\mcw+\partial_{x_3}p\partial_{y_2}u\partial_{x_1x_2}\mcw+\partial_{x_1x_2}\mcw\partial_{y_3}u)\mathbf{I}_{21})\\
	&+\frac{1}{\partial_{y_1}u}((\partial_{x_2x_3}p\partial_{y_2}u\partial_{x_2}\mcw+\partial_{x_3}p\partial_{y_2}u\partial_{x_2x_2}\mcw+\partial_{x_2x_2}\mcw\partial_{y_3}u)\mathbf{I}_{22})\\
	&+\frac{\partial_{y_2}u}{\partial_{y_1}u}(-\partial_{x_1x_3}p\mathbf{I}_{31}-\partial_{x_2x_3}p\mathbf{I}_{32}-\partial_{x_3x_3}p\mathbf{I}_{33}+\partial_{x_3x_3}p\partial_{x_2}\mcw\mathbf{I}_{23}),\\
\end{align*}
where $\mathbf{I}_{ij}\defs e_i^\top e_j\in\mathbb{R}^{4\times 4}$ with $\{e_i\}_{i=0}^3$ being the canonical basis of $\mathbb{R}^4$.
Then we have
\begin{align}
\sum_{i,j=0}^{3}a_{ij}\partial_{x_ix_j}\phi
=\tr(\mathbf{A}^{\top}\mathrm{D}^{2}\phi)
=-\frac{1}{\big(\partial_{y_{1}}u\big)^{3}}\sum_{i,j=0}^{3}\tilde{a}_{ij}\partial_{y_iy_j}u+\sum_{i=1}^4 S_i,\label{equation of phi in new coordinate}
\end{align}
where $\mathbf{A}\defs\begin{bmatrix}a_{ij}\end{bmatrix}_{4\times 4}$ with $a_{ij}$ being defined in \eqref{eq:potential_eq_coefficients1}-\eqref{eq:potential_eq_coefficients2} and $\tr(\mathbf{M})$ means the trace of the square matrix $\mathbf{M}$. The coefficients $\tilde{a}_{ij}=\tilde{a}_{ij}(\partial_{x_1}\mcw,\partial_{x_2}\mcw,\mathrm{D}u;\mathrm{D}\Phi^{-})$ satisfy that
\begin{align}
    \begin{bmatrix}\tilde{a}_{ij}\end{bmatrix}_{4\times 4}\defs \mathbf{J}^{\top}\mathbf{A}\mathbf{J}=\tilde{\mathbf{A}}=\tilde{\mathbf{A}}^{\top},\nonumber
\end{align}
and
\begin{align}
    S_1&=a_{11}(-\partial_{x_1x_1}p\frac{\partial_{y_2}u}{\partial_{y_1}u}+\partial_{x_1x_1}\mcw\frac{\partial_{y_3}u}{\partial_{y_1}u}),\nonumber\\
    S_2&=\frac{1}{\partial_{y_1}u}(-a_{13}\partial_{x_1x_3}p\partial_{y_2}u+a_{21}(\partial_{x_1x_3}p\partial_{x_2}\mcw\partial_{y_2}u+\partial_{x_3}p\partial_{x_1x_2}\mcw\partial_{y_2}u+\partial_{x_1x_2}\mcw\partial_{y_3}u)),\nonumber\\
    S_3&=\frac{1}{\partial_{y_1}u}(a_{22}(\partial_{x_3x_3}p\partial_{x_2}\mcw\partial_{y_2}u+\partial_{x_3}p\partial_{x_2x_2}\mcw\partial_{y_2}u+\partial_{x_2x_2}\mcw\partial_{y_3}u)+a_{23}\partial_{x_3x_3}p\partial_{x_2}\mcw\partial_{y_2}u),\nonumber\\
    S_4&=-\frac{\partial_{y_2}u}{\partial_{y_1}u}(a_{31}\partial_{x_1x_3}p+a_{32}\partial_{x_2x_3}p+a_{33}\partial_{x_3x_3}p)\nonumber.
\end{align}
By simple calculation, especially, one has
\begin{align}
\tilde{a}_{03}=\tilde{a}_{30}=&(\partial_{y_1}u)^2\cdot d,\label{eq:a03}\\
\tilde{a}_{13}=\tilde{a}_{31}=&-(\partial_{y_1}u)^2\cdot d+\partial_{y_1}u(\partial_{x_1}\mcw\partial_{y_3}u-\partial_{x_1}p\partial_{y_2}u+1)\cdot d+c^2\partial_{x_1}\mcw)\nonumber\\
&-(\partial_{x_3}p\partial_{y_2}u+\partial_{y_3}u)\partial_{y_1}u(\partial_{x_3}\Phi\cdot d-c^2),\label{eq:a13}\\
\tilde{a}_{23}=\tilde{a}_{32}=&\partial_{x_1}p(\partial_{y_1}u)^2(\partial_{x_1}\Phi\cdot d+c^2\partial_{x_1}\mcw)+(\partial_{x_2}p+1)(\partial_{y_1}u)^2(\partial_{x_2}\Phi\cdot d+c^2\partial_{x_2}\mcw)\nonumber\\
&+((\partial_{x_3}p\partial_{x_2}\mcw-1)\partial_{y_2}u+\partial_{y_3}u\partial_{x_2}\mcw)\partial_{y_1}u(\partial_{x_2}\Phi\cdot d+c^2\partial_{x_2}\mcw)\nonumber\\
&+\partial_{x_3}p(\partial_{y_1}u)^2(\partial_{x_3}\Phi\cdot d-c^2),\label{eq:a23}
\end{align}
where $d=\partial_{x_3}\Phi-\partial_{x_1}\Phi\partial_{x_1}\mcw-\partial_{x_2}\Phi\partial_{x_2}\mcw$.

For the other coefficients, because we do not need the properties of their trace on the boundary, they are listed in the appendix.
%It is easy to see that $\tilde{a}_{03}|_{\scriptscriptstyle{\Gw}}=0$. In fact we also have $\tilde{a}_{13}|_{\scriptscriptstyle{\Gw}}=\tilde{a}_{23}|_{\scriptscriptstyle{\Gw}}=0$.

From \eqref{eq:potential_eq_for_phi} and \eqref{equation of phi in new coordinate}, we deduce that $u$ satisfies following equation
\begin{equation}\label{equation satisfied by u}
    \sum_{i,j=0}^3\tilde{a}_{ij}\partial_{y_iy_j}u+\tilde{a}_2\partial_{y_2}u+\tilde{a}_{3}\partial_{y_3}u+a_{12}\partial_{x_1x_2}p(\partial_{y_1}u)^3=-(\partial_{y_1}u)^3\sum_{i,j=0}^3a_{ij}\partial_{x_ix_j}\Phi^-,
\end{equation}
where
\begin{align}
    \tilde{a}_2=&(\partial_{y_1}u)^2(a_{11}p_{x_1x_1}+a_{13}\partial_{x_1x_3}p-a_{21}(\partial_{x_2x_3}p\partial_{x_2}\mcw+\partial_{x_3}p\partial_{x_1x_2}\mcw))\nonumber\\
    &+(\partial_{y_1}u)^2(a_{12}\partial_{x_1x_2}p+a_{31}\partial_{x_1x_3}p+a_{32}\partial_{x_2x_3}p-a_{22}\partial_{x_3}p\partial_{x_2x_2}\mcw),\\
    \tilde{a}_3=&(\partial_{y_1}u)^2(-a_{11}\partial_{x_1x_1}\mcw-a_{12}\partial_{x_1x_2}\mcw-a_{21}\partial_{x_1x_2}\mcw-a_{22}\partial_{x_2x_2}\mcw+a_{12}\partial_{x_1x_2}p).\label{eq:2-31}
\end{align}
Assume 
\begin{equation}\label{2.16}
\partial_{x_2}\mcw(x_1,0)=0.
\end{equation} 
Then the partial hodograph transformation $\mathscr{P}$ mapps the axis $x_2=0$ in $(t,\mbx)$-coordinate to the axis $y_2=0$ in $(y_0,\mby)$-coordinate. Moreover
the perturbed solid wall $\Gamma_0$ and the shock front $\Gamma_{\scriptstyle{\textrm{shock}}}$ in $(t,\mathbf{x})$-coordinate are mapped to
\begin{equation}\label{wedge in new coordinate}
    \Gamma_w\defs\{y_0>0,y_1>0,y_3=0\}
\end{equation}
and
\begin{equation}\label{shock font in new coordinate}
    \Gamma_{s}\defs\{y_0>0,y_1=0,y_3>0\},
\end{equation}
respectively.
Substituting the expressions of $\md_{t,\mbx}\phi$ and $p(\mathbf{x})$ into \eqref{eq:potential_bdry_wedge}, we find that $u$ satisfies
\begin{equation}\label{Neumann bdry condition of u on horizontal boundary}
    \partial_{y_3}u=-\frac{\partial_{x_1}\mcw}{1+|\partial_{x_1}\mcw|^2+|\partial_{x_2}\mcw|^2}\quad\mbox{on}\quad\Gamma_w.
\end{equation}
Substituting \eqref{Dphi in terms of Du} into \eqref{eq:potential_bdry_RH-2}, we obtain the Rankine-Hugoniot condition in the new coordinate variables:
\begin{equation}\label{R-H condition in new coordinate}
    G(u,\rm{D}u;\rm{D}\Phi^-)=0\quad\mbox{on}\quad \Gamma_{s},
\end{equation}
where
\begin{align}
    G(u,\rm{D}u;\rm{D}\Phi^-)\defs[\rho][\Phi_t]+[\partial_{x_1}\Phi][\rho\partial_{x_1}\Phi]+[\partial_{x_2}\Phi][\rho\partial_{x_2}\Phi]+[\partial_{x_3}\Phi][\rho\partial_{x_3}\Phi],\label{eq:2-34}
\end{align}
where $\md\Phi$ should be replaced by $\md\Phi^--\md_{t,\mbx}\phi$ and $\md_{t,\mbx}\phi$ should be replaced by $\md u$ via $\eqref{Dphi in terms of Du}$.
For the initial conditions, we assume
\[
u(y_0,\mby)|_{y_0}\defs u_0(\mby)\quad\mbox{and}\quad \partial_{y_0}u(y_0,\mby)|_{y_0=0}\defs u_1(\mby),
\]
where $u_0$ and $u_1$ are some given functions.

%We can deduce the initial conditions for $u$ from the initial conditions of $\Phi$. Let $u_0(\mathbf{y})\defs u(y_0,\mathbf{y})|_{y_0=0}$. Then we have
%\begin{align}\label{tranform identiy at initial time}
%    y_1=\phi(0,u_0,x_2(u_0(\mathbf{y}),y_2,y_3),y_3+\mcw(u_0(\mathbf{y}),x_2)).
%\end{align}
%Since at the background solution $u_b$, we have
%\begin{equation*}
%    \partial_{x_1}\phi+\partial_{x_2}\phi\frac{\partial x_2}{\partial u}+\partial_{x_3}\phi(\partial_{x_1}\mcw+\frac{\partial x_2}{\partial u})=q_--q_+>0,
%\end{equation*}
%which is still positive when $u$ is sufficiently close to $u_b$.
%Hence one can deduce $u_0(\mathbf{y})$ from $\eqref{tranform identiy at initial time}$ by the implicit function theorem. The initial condition for $\partial_{y_0}u$ can be deduced by the same argument and we define $u_1(\mathbf{y})\defs \partial_{y_0}u(y_0,\mathbf{y})|_{y_0=0}$.
For notational simplicity, one defines $Lu$ by
\[
Lu\defs\sum_{i,j=0}^3\limits\tilde{a}_{ij}\partial_{y_iy_j}u+\tilde{a}_2\partial_{y_2}u+\tilde{a}_{3}\partial_{y_3}u+a_{12}\partial_{x_1x_2}p(\partial_{y_1}u)^3+(\partial_{y_1}u)^3\sum_{i,j=0}^3\limits a_{ij}\partial_{x_ix_j}\Phi^-,
\]
where the coefficients depend on $u(y_0,
\mby)$ and its first order derivatives, as well as $\mcw(x_1,x_2)$ and its derivatives up to third order.
Gathering \eqref{equation satisfied by u}, \eqref{Neumann bdry condition of u on horizontal boundary}-\eqref{R-H condition in new coordinate}, and the initial conditions of $u$, we get the initial boundary value problem concerned in this paper:
\begin{equation}\label{NLP}
\begin{cases}
Lu=0&\mbox{in}\quad \Omega_T,\\
G(\mcw'(u),\mathrm{D}u;\mathrm{D}\Phi^-)=0&\mbox{on}\quad\{y_1=0\},\\
G_1\defs(1+|\partial_{x_1}\mcw|^2+|\partial_{x_2}\mcw|^2)\partial_{y_3}u+\partial_{x_1}\mcw=0&\mbox{on}\quad\{y_3=0\},\\
u(y_0,\mathbf{y})=u_0(\mathbf{y}),{\ } u(y_0,\mathbf{y})=u_1(\mathbf{y})&\mbox{on}\quad\{y_0=0\}.
\end{cases}\tag{NLP}
\end{equation}
Here $\Omega_T\defs[0,T]\times\Omega$ and $\Omega\defs\mathbb{R}^+_{y_1}\times\mathbb{R}_{y_2}\times\mathbb{R}^+_{y_3}$, where $\mathbb{R}^+=(0,+\infty)$ and $\mathbb{R}$ is the set of real numbers. Here and after, denote this initial boundary value problem by \eqref{NLP}.

In the $(t,\mathbf{x})$-coordinate, the background state for $\phi$ is
 $$\bar{\phi}(t,\mathbf{x})\defs\overline{\Phi}_{-}(t,\mbx)-\overline{\Phi}_{+}(t,\mbx)=(q_--q_+)x_1.$$
Then the corresponding partial hodograph transformation is
\begin{equation}
y_0=t,\qquad y_1=\bar{\phi}(t,\mathbf{x}),\qquad y_2=x_2,\qquad y_3=x_3,
\end{equation}
and its inverse transformation is
\begin{equation}
t=y_0,\qquad x_1=u_b(\mathbf{y}),\qquad x_2=y_2,\qquad x_3=y_3.
\end{equation}
It is clear that
$$x_1=\frac{1}{q_--q_+}\bar{\phi}(t,\mbx)=\frac{1}{q_--q_+}y_1.$$
Hence we have
\begin{equation}\label{defn of ub}
u_b(\mathbf{y})=\frac{1}{q_--q_+}y_1.
\end{equation}
% Hence , which implies
%\begin{equation}
%    u=\frac{1}{q_--q_+}y_1\defs u_b(y_1).
%\end{equation}
At the background state, i.e., the state that $u=u_b$, $\mcw(x_1,x_2)\equiv 0$, $\nabla\Phi(t,\mathbf{x})\equiv (q_+,0,0)$ and  $\nabla\Phi^-(t,\mathbf{x})\equiv (q_-,0,0)$, one has
\begin{align}
    &\tilde{a}_{00}=\frac{1}{(q_--q_+)^2}>0,\quad \tilde{a}_{01}=\tilde{a}_{10}=\frac{q_+}{q_--q_+}>0,\quad\tilde{a}_{02}=\tilde{a}_{20}=0,\label{eq:2-40}\\
    &\tilde{a}_{11}=q_+^2-c_+^2<0,\quad{\ \ \!}\quad\tilde{a}_{03}=\tilde{a}_{30}=0,\qquad\qquad\qquad\tilde{a}_{13}=\tilde{a}_{31}=0,\\
    &\tilde{a}_{22}=\frac{-c_+^2}{(q_--q_+)^2}<0,\quad\tilde{a}_{21}=\tilde{a}_{12}=0,\qquad\qquad\qquad\tilde{a}_{23}=\tilde{a}_{32}=0,\\
    &\tilde{a}_{33}=\frac{-c_+^2}{(q_--q_+)^2}<0.\label{eq:2-43}
\end{align}
In $\mathbf{y}$-coordinates, the dynamic stability problem is rewritten as the following problem:
\vskip 0.2 cm
\textbf{Problem 2.} Suppose the initial data $(u_0,u_1)$ and $\mathcal{W}$ are small perturbations of the background state $u_b$ and zero, respectively and $\nabla\Phi^-$ is close to $(q_-,0,0)$. Can we show the local existence and uniqueness of smooth solutions to \eqref{NLP}, such that the unique solution is still close to $u_b$?

The remaining part of this paper is devoted to solving this problem. It is shown that one can indeed find a unique smooth solution to \eqref{NLP} near $u_b$, if the following condition:
\begin{equation}\label{extra condition}
q_-\rho_+-q_+\rho_--\rho_+>0
\end{equation}
holds for the constants $(\rho_-,q_-,\rho_+,q_+)$.

\begin{rmk}
It should be noted that, as one will see from the proof of lemma $\ref{lem 5.1}$, the condition \eqref{extra condition} is employed to guarantee that the steady normal shock solution satisfies the stability conditions, which are defined in \ref{h4} below in the beginning of section \ref{sec:3}.
However, the conditions \eqref{2.8} and \eqref{2.9} are not sufficient to yield \eqref{extra condition}. For example, for any $1<\lambda<\frac{1+\sqrt{5}}{2}$, choose $(q_-,\rho_-,q_+,\rho_+)$ as follows:
\begin{equation}
	q_-=\lambda,{\ } q_+=1,{\ }
	\rho_-=\left(\frac{(\gamma-1)(\lambda^2-1)}{2(\lambda^{\gamma-1}-1)}\right)^{\frac{1}{\gamma-1}},\mbox{\ and\ \ }\rho_+=\lambda\rho_-.
\end{equation}
Then it can be easily verified that \eqref{2.8} and \eqref{2.9} are valid, but \eqref{extra condition} fails:
\begin{equation}
q_-\rho_+-q_+\rho_--\rho_+=(\lambda^2-\lambda-1)\rho_-<0.
\end{equation}
\end{rmk}
%In what follows, wewill give a positive answer to this question, if $(\rho_-,\rho_+,q_-,q_+)$ satisfy
%\begin{equation}
%    q_-\rho_+-q_+\rho_--\rho_+>0.
%\end{equation}
%\newline
%\begin{rmk}
%	Thanks to the  property of the finite speed of propagation of hyperbolic equations and the well-established mathematical theory for initial boundary value problems for hyperbolic equations with smooth boundaries (for instance, see \cite{BenzoniSerre2007}), one can assume that, without loss of generality, the perturbation only occurs near the corner, where the lower nozzle wall $x_3=\mathcal{W}(x_1,x_2)$ and the shock front meets.  Therefore, this paper only solves the stability problem near the corner, and in a short time.
%\end{rmk}
\begin{rmk}
It is worth pointing out that, since the solid boundary is perturbed and no longer flat, the symmetry assumptions proposed in \cite{GS,Yuan} fail to be valid in this problem. Therefore, new ideas and methods must be developed to deal with the dihedral singularity, which is also completely different from the one caused by the corner singularity in \cite{FXX}. These are the main new ingredients of this paper.
\end{rmk}
Now, we are ready to state our main result as following theorem:
\begin{thm}\label{wellposedness of the NLP}
\setlength{\abovedisplayskip}{15pt}
For each integer $s_0\ge 3$, suppose the initial-boundary data of \eqref{NLP} satisfy the compatibility condition up to order $s_0+1$.  If conditions \eqref{2.8}, \eqref{2.9}, \eqref{2.16} and \eqref{extra condition} hold, then there exist three constants $\eta_0>1$, $T_0>0$ and $\tilde{\epsilon}>0$ such that if
 \vspace{-0.5cm}
\begin{align}
    &\|u_0-u_b\|_{H^{s_0+1}(\Omega)}+\|u_1\|_{H^{s_0}(\Omega)}+\|\mcw\|_{W^{s_0+2,\infty}(\mathbb{R}^2)}\nonumber\\
    &\qquad+\|e^{-\eta t}(\mathrm{D}\Phi^--(q_-,0,0))\|_{H^{s_0}([0,T]\times\{x_3>\mcw(x_1,x_2)\})}\le\epsilon
\end{align}
is satisfied for $0<T\le T_0$, $\eta\ge \eta_0$ and $\epsilon\le \tilde{\epsilon}$, where $\|\cdot\|_{H^k}$ stands for the standard Sobolev norm.
Then \eqref{NLP} admits a unique  solution $u\in H^{s_0+1}(\Omega_T)$ satisfying
\vspace{-0.5cm}
\begin{equation}
\|e^{-\eta t}(u-u_b)\|_{H^{s_0+1}(\Omega_T)}\le C\epsilon,
\end{equation}
% for some $\zeta\in(0,\frac12)$, 
where $C$ is a positive constant depending on $(q_-, q_+, \rho_-, \rho_+,T_0,\eta_0)$.
\end{thm}
\begin{rmk}\label{rem:2.1}
The compatibility conditions mentioned in Theorem \ref{wellposedness of the NLP} come from the requirement that the initial-boundary data of \eqref{NLP} should be consistent. More precisely, by initial conditions in $\eqref{NLP}$ and the first equation of $\eqref{NLP}$, we know that at $y_0=0$,
\[
\md^{\beta}u=\md^{\beta}u_0,\quad \partial_{y_0}\md^{\beta}u=\md^{\beta}u_1
\]
and
\[
\quad \partial_{y_0}^2\md^{\beta}u=\md^\beta(\frac{1}{\tilde{a}_{00}}(\tilde{f}-\sum_{(i,j)\neq(0,0)}^2\tilde{a}_{ij}\partial_{y_iy_j}u)),
\]
where $\md^{\beta}=\partial_{y_1}^{\beta_1}\partial_{y_2}^{\beta_2}\partial_{y_3}^{\beta_3}$ is the spatial derivatives and $\beta=(\beta_1,\beta_2,\beta_3)$ is the multi-index corresponds to spatial derivative and
\[
\tilde{f}=(\partial_{y_1}u)^3\sum_{i,j=0}^3a_{ij}\partial_{x_ix_j}\Phi^-+\tilde{a}_2\partial_{y_2}u+\tilde{a}_{3}\partial_{y_3}u+a_{12}\partial_{x_1x_2}p(\partial_{y_1}u)^3.
\]
Then by induction on $k$ (i.e., assume we have already known the expression of $\partial_{y_0}^{m+1}\md^{\beta}u$ at $y_0=0$ for all $m\leq k$.) and by taking derivative $\md^{\beta}\partial_{y_0}^k$ on equation $\eqref{NLP}_{1}$, we will have the expression of $\partial_{y_0}^{k+2}\md^{\beta}u$ at $y_0=0$. We omit the details for the shortness. Then we have the expression of $\md^{\alpha}u$ at $y_0=0$ for all $\alpha=(\alpha_0,\alpha_1,\alpha_2,\alpha_3)$. Let
\begin{equation}\label{3.29x}
u_{\alpha}:=\md^{\alpha}u\big|_{y_0=0}.
\end{equation}
On the other hand, we have two boundary conditions in $\eqref{NLP}$.
So for any $(k_0,k_1,k_2,k_3)\in\mathbb{N}^4$, we have
\[
\md^{(k_0,k_1,k_2,0)}G=0\quad\mbox{on }\{y_3=0\}\qquad
\mbox{and}
\qquad
\md^{(k_0,0,k_2,k_3)}G_1=0\quad\mbox{on }\{y_1=0\}.
\]
Let $U:=(u,\md u)$, then by the Fa\'{a} di Bruno's formula and the Leibniz rule, we know there exist $c_{l_1\cdots l_m l'_1\cdots l'_m l''_1\cdots l''_m}(U)$ and $c'_{l_1\cdots l_m l'_1\cdots l'_m l''_1\cdots l''_m}(U)$ such that
\begin{align}
\!\!\!\!\!\!\sum_{m=1}^{\max(k_0,k_1,k_2)}\!\!\!\!\!\!\!\sum_{\substack{l_1+\cdots+l_m=k_0\\ l_1'+\cdots+l_m'=k_1\\l_1''+\cdots+l_m''=k_2}}\!\!\!\!c_{l_1\cdots l_m l'_1\cdots l'_m l''_1\cdots l''_m}(U)\cdot(\md^{(l_1,l_1',l_1'',0)}U,\cdots,\md^{(l_m,l_m',l_m'',0)}U)=0{\ }\mbox{on }\{y_3=0\}\nonumber
\end{align}
and
\begin{align}
\!\!\!\!\!\!\!\!\sum_{m=1}^{\max(k_0,k_2,k_3)}\!\!\!\!\!\!\!\!\sum_{\substack{l_1+\cdots+l_m=k_0\\ l_1'+\cdots+l_m'=k_2\\l_1''+\cdots+l_m''=k_3}}\!\!\!\!c'_{l_1\cdots l_m l'_1\cdots l'_m l''_1\cdots l''_m}(U)\cdot(\md^{(l_1,0,l_1',l_1'')}U,\cdots,\md^{(l_m,0,l_m',l_m'')}U)=0{\ }\mbox{on }\{y_1=0\}.\nonumber
\end{align}
Here integers $l_m$, $l_m'$ and $l_m''$ can be zero. Let $y_0=0$ and plug \eqref{3.29x} into the two identities above for all integers $k_0+k_1+k_2\leq s_0+1$ and $k_0+k_2+k_3\leq s_0+1$. Then we can obtain the identities that the initial and boundary data must satisfy for all integers $k_0+k_1+k_2\leq s_0+1$ and $k_0+k_2+k_3\leq s_0+1$. These identities are called the compatibility conditions up to order $s_0+1$.
\end{rmk}

\section{Well-posedness of the linear problem}\label{sec:3}
In this section, we will establish the well-posedness theorem for an initial boundary value problem of a linear hyperbolic equation of second order in the dihedral-space domain. The linear theorem will be used to solve the \eqref{NLP} by introducing an iteration scheme in the next section.

In the following part of this section, we investigate the following initial boundary value problem
\begin{equation}\label{LP}
\begin{cases}
    L'(u)w=f\quad&\mbox{in}\quad\Omega_T,\\
    \mathcal{B}(u)w=g\quad&\mbox{on}\quad \Gamma_s,\\
    \partial_{y_3}w=0\quad&\mbox{on}\quad\Gamma_w,\\
    (w,\partial_{y_0}w)=(0,0)\quad&\mbox{on}\quad\Gamma_{in}\defs \{y_0=0\},
\end{cases}\tag{LP}
\end{equation}
where
\[
L'(u)\defs\sum_{i,j=0}^3r_{ij}\partial_{ij}+\sum_{i=0}^3r_i\partial_i+r,
\]
\[
\mathcal{B}(u)\defs\sum_{i=0}^3b_i\partial_{i}+b,
\]
$\Omega_T$ is the time-spatial domain defined below \eqref{NLP} in section \ref{sec:2}, $\Gamma_w$ and $\Gamma_s$ are defined by  \eqref{wedge in new coordinate} and \eqref{shock font in new coordinate} respectively.
We impose following hypothesis on the coefficients of the operators $L'(u)$ and $\mathcal{B}(u)$.
\begin{enumerate}[label=($\textbf{H}_\arabic*$)]
\item $L'(u)$ is a hyperbolic operator of second order.
$r_{ij}$, $r_i$ and $r$ $\mbox{\ are\ smooth\ functions of\ }$ $\md\Phi^-,{\ }\md u {\ }\mbox{and}{\ }\mcw(u,x_2(u,y_2,y_3))$. Moreover $r_{32}$, $r_{31}$, $r_{30}$ and $ r_2$ vanish on the flat boundary $\Gamma_w$. In particular, at the background solution $u_b$, which is given in \eqref{defn of ub}, $r_{10}=r_{01}>0$, $r_{12}=r_{21}=0$, $r_{02}=r_{20}=0$, $r_{33}=r_{22}<0$, $r_{30}=r_{03}=r_{31}=r_{13}=r_{32}=r_{23}=0$ and $r_{11}<0$.\label{h1}
\item $b_i$ and $b$ are smooth functions depend on $\md u$ and $\mcw'(u)$ and
$b_3|_{\Gamma_w}=0$. Furthermore, $b=b_2=b_3=0$ at the background solution $u_b$.\label{h2}
\item There exists an integer $n_0\ge 1$ and $\delta>0$ such that $$\sup\limits_{0\le y_0\le T}\sum_{|\alpha|\le n_0+3}\|\md^\alpha(u-u_b)\|_{L^2(\Omega)}<\delta.$$\label{h3}
\item At the background solution $u_b$, the
following stability conditions hold for some constant $\gamma_0>0$:
$$\qquad|b_1|\geq\gamma_0,\quad \frac{\tilde{a}_{11}b_0}{b_1}-r_{01}\geq\gamma_0,\quad\sum_{i,j=0}^3r^{ij}\left(\frac{r_{11}b_i}{b_1}-r_{i1}\right)\left(\frac{r_{11}b_j}{b_1}-r_{j1}\right)\geq\gamma_0.$$
Here $r^{ij}$ is the $(i,j)$-th entry of the matrix $\begin{bmatrix}r_{ij}\end{bmatrix}^{-1}_{4\times 4}$, the inverse matrix of $\begin{bmatrix}r_{ij}\end{bmatrix}_{4\times 4}$.\label{h4}
\end{enumerate}

Let us introduce some notations:
\begin{align*}
  &\widetilde{\Omega}\defs\mathbb{R}^+_{y_1}\times\mathbb{R}_{y_2}\times\mathbb{R}_{y_3},\quad\omega\defs\{0\}\times\mathbb{R}_{y_2}\times\mathbb{R}_{y_3},\\
  &\omega^\ell\defs\omega\cap\{y_3>0\},\quad\quad\omega^r\defs\mathbb{R}^+_{y_1}\times\mathbb{R}_{y_2}\times\{0\}.
\end{align*}
And $\omgrt\defs[0,T]\times \widetilde{\Omega}$ is defined to be the right half time-spatial domain of $[0,T]\times \mathbb{R}^3$. Let $\omega_T:=[0,T]\times \omega$ and let $w^i_T:=[0,T]\times\omega^i$ for $i=\ell$ and $r$.

For the linear problem \eqref{LP}, we have the following theorem.
\begin{thm}\label{well-posedness and energy estimate of the LP}
Suppose assumptions \ref{h1}-\ref{h4} are fulfilled and $\partial^{k}_{t}f|_{t=0}=0$ for $k=0,1,2,\cdots, n_0+2$ with an integer $n_0\geq1$. Then there exists a smooth solution $w$ to \eqref{LP}. Moreover, there exists $\eta_0\ge 1$ and $T_0>0$ such that for all $s\le n_0+3$, the following estimate
\begin{align}
&\sum_{|\alpha|\le s} \eta\|e^{-\eta t}\md^\alpha w\|^2_{L^2(\Omega_{T})}+e^{-2\eta T}\sup_{0\le t\le T}\|\md^\alpha w(t,\cdot)\|_{L^2(\Omega)}^2+\|e^{-\eta t}\md^\alpha w\|^2_{L^2(\omega^\ell_T)}\nonumber\\
&\quad\lesssim \frac{1}{\eta}(e^{2\eta T}\|e^{-\eta t}u\|^2_{H^{s}(\Omega_T)}\|e^{-\eta t}f\|^2_{H^3(\Omega_T)}+\|e^{-\eta t}f\|^2_{H^{s-1}(\Omega_T)})+\|e^{-\eta t}g\|^2_{H^{s-1}(\omega^\ell_T)}\label{ineq: LP energy estimate}
\end{align}
holds for all $\eta\ge \eta_0$ and $0<T\le T_0$.
\end{thm}

We have used the notation $\lesssim$ in \eqref{ineq: LP energy estimate}. Hereafter $A\lesssim B$ means that $A\le C B$ for some positive constant $C$. Before giving a proof to theorem \ref{well-posedness and energy estimate of the LP}, we define an auxiliary problem \eqref{LPE}.
%Before that, l
%Firstly, we extend all the coefficients of $L'(u)$, the initial data and $f$ smoothly from $\Omega_T$ to $\omgup$. All the extended coefficients are still denoted by the original symbols. Then we define \eqref{LP1} as
%\begin{equation}\label{LP1}
%\begin{cases}
%    L'(u)w_1=0\quad&\mbox{in}\quad\omgup\\
%    \partial_{y_3}w_1=0\quad&\mbox{on}\quad\{y_3=0\}\\
%    (w_1,\partial_{y_0}w_1)=(v_0,v_1)\quad&\mbox{on}\quad\{y_0=0\}
%\end{cases}.\tag{LP1}
%\end{equation}
The coefficients of $L'(u)$, $f$, $g$ and the coefficients of $\mathcal{B}(u)$ will be extended from $\Omega_T$ to $\omgrt$ in the following way:
\begin{itemize}
\item[(\romannumeral1)] Extend $r_{03}=r_{30}$, $r_{13}=r_{31}$, $r_{23}=r_{32}$ and $b_3$ oddly with respect to $\{y_3=0\}$. To be precise, we take $r_{03}$ for example. Extend $r_{03}$ by letting $(Er_{03})(y_0,\mathbf{y})\defs r_{03}(y_0,\mathbf{y})$ when $y_3\ge 0$ and $(Er_{03})(y_0,\mathbf{y})\defs -r_{03}(y_0,y_1,y_2,-y_3)$ when $y_3<0$.  Coefficients other than $r_{03}$, $r_{13}$ and $r_{23}$ will be extend evenly with respect to $\{y_3=0\}$ by the same manner.
    \item[(\romannumeral2)] Extend $f$ and $g$ evenly with respect to $\{y_3=0\}$.
\end{itemize}
For notational simplicity, we omit the $``E"$ for all extended coefficients. The \eqref{LPE} is defined as follows
\begin{equation}\label{LPE}
\begin{cases}
    L'(u)w=f\quad&\mbox{in}\quad\omgrt\\
    \mathcal{B}(u)w=g\quad&\mbox{on}\quad \omega_T\\
    (w,\partial_{y_0}w)=(0,0)\quad&\mbox{on}\quad\Gamma_{in}\defs \{y_0=0\}
\end{cases}.\tag{LPE}
\end{equation}
\begin{rmk}\label{rem:3.1}
The compatibility conditions up to order $n_0+3$ for \eqref{LP} can be obtained by same arguments as Remark \ref{rem:2.1} away from the wedge $y_3=0$. Obviously, due to the regularity of the extended coefficients, \eqref{LPE} only satisfies the compatibility conditions as the one in Remark \ref{rem:2.1} on the wedge up to order $2$. But it is enough for us to show the existence of solutions of the \eqref{LPE} in $H^2(\omgrt)$. Then the better regularity in $\Omega_T$ of such solutions can be obtained by further argument.
\end{rmk}
%\subsection{The \textit{a priori} estimate of the linearized problem}
By employing the idea said in Remark \ref{rem:3.1} above, we will consider the well-posedness of the \eqref{LP} in the next proposition, by considering the well-posedness of the \eqref{LPE}, and proving that the unique solution to $\eqref{LPE}$ is the unique solution to $\eqref{LP}$ and it satisfies better estimate in $\Omega_T$.

\begin{prop}%[A priori estimate of \eqref{LP}]
\label{well-posedness of LPE}
 If hypothesis \ref{h1}-\ref{h4} hold, $\partial^{k}_{t}f|_{t=0}=0$ for $k=0,1,2,\cdots,s-1$, with $s\le n_0+2$, then  \eqref{LP} admits a smooth solution $w$.  Moreover, there exists $\eta_0\ge 1$ such that for $\eta\ge \eta_0$ and $T>0$, it holds that
    \begin{align}
    &\sum_{|\alpha|\le s+1}\left(\eta\|e^{-\eta t}\md^\alpha w\|^2_{L^2(\Omega_T)}+e^{-2\eta T}\|\md^\alpha w|_{t=T}\|^2_{L^2(\Omega)}+\|e^{-\eta t}\md ^\alpha w|_{y_1=0}\|^2_{L^2(\omega^\ell_T)}\right)\nonumber\\
    &\quad\lesssim \frac{1}{\eta}\sum_{|\beta|\le s}\|e^{-\eta t}L'(\md^\beta w)\|^2_{L^2(\Omega_T)}+\|e^{-\eta t}\mcb(\md^\beta w)|_{y_1=0}\|^2_{L^2(\omega^\ell_T)}+\|e^{-\eta t}g\|^2_{H^s(\Omega_T)}\nonumber\\
    &\quad\quad+\sum_{|\alpha|\le s-1}\|\md^\alpha f|_{t=0}\|^2_{L^2(\Omega)}.\label{neq:estimate of lpe}
\end{align}
\end{prop}
\begin{proof}
In the proof of this proposition, $\md^\ell v$ stands for the derivatives of function $v$ of order no higher than $\ell$ and $\abs{\md^\ell v}^p\defs\sum_{|\alpha|\le \ell}\abs{\md^\alpha v}^p$ for $p=1,2$. In what follows, the dependence of the operators $L'$ and $\mcb$ on $u$ is omitted. For brevity, one uses the notation $\partial_{i_1\cdots i_\ell}$ to represent the partial derivative with respect to the variables  $y_{i_1}$, $y_{i_2}$, $\cdots$, $y_{i_{\ell}}$. Since the proof of this proposition is long, we divide it into five steps. In the first two steps, we will illustrate how to derive the energy estimate up to the second order and to show the existence and uniqueness of solutions to \eqref{LPE}, which is indeed a solution to \eqref{LP}, with the aid of the property of the extension. %Then in the last step, one uses induction method to establish higher order estimate. More specifically, in the first step, one deduces the first order estimate, which requires a carefully chosen multiplier. In the second step, second order estimate is established, on the basis of the first order estimate. Then the existence and uniqueness of solutions to \eqref{LPE} is given at the end of the second step. With the aid of the property of the extension, one can further show that the unique solution to \eqref{LPE} is indeed a solution to \eqref{LP}. In addition, one is able to derive the energy estimate of the solution up to second order in the half-space domain. 
It should be emphasized that the estimate of derivatives higher than second order cannot be derived in the half-space domain directly, due to the restriction of the regularity of the extended coefficients. Hence we are forced to establish higher order estimate in the dihedral-space domain in the remaining three steps. For this purpose, two multipliers are constructed to deal with the boundary terms (see lemma \ref{Multiplier for two Dirichlet boundary conditions} and lemma \ref{Multiplier for almost normal boundary conditions}). In the fifth step (the final step), we treat the energy estimates of even order and odd order separately, since we will meet different types of boundary conditions. The estimate in this step still relies on the multipliers constructed in lemma \ref{Multiplier for two Dirichlet boundary conditions} and lemma \ref{Multiplier for almost normal boundary conditions}. It is useful to point out the observation that both $\partial_{y_0}$ and $\partial_{y_2}$ are tangential to the boundaries $\Gamma_s$ and $\Gamma_w$. Hence any established estimate of $w$ can be directly applied to $\partial_{y_0}w$ and $\partial_{y_2}w$, which helps to simplify the higher order estimate.

By the Sobolev embedding theorem and assumption \ref{h2}, one has
$$
\sup_{(y_0,y)\in [0,T]\times \mathbb{R}^3_+}\sum_{|\alpha|\le n_0}\left|\mathrm{D}^\alpha (u-u_b)(y_0,y)\right|\le C\delta.
$$
Since $s\ge[\frac{s+2}{2}]\,\mbox{ if } s\ge 4$, we deduce that if $n_0\geq 4$ then
\begin{equation}\label{Upper bound of un}
\sup_{(y_0,y)\in [0,T]\times \mathbb{R}^3_+}\sum_{|\alpha|\le \frac{n_0+2}{2} }\abs{\mathrm{D}^\alpha (u-u_b)(y_0,y)}\le C\delta.
\end{equation}
As a corollary of \eqref{Upper bound of un} and assumption \ref{h1}, we have
\begin{equation}
\sup_{(y_0,y)\in [0,T]\times \mathbb{R}^3_+}|\md r_{ij}(y_0,y)|\leq C\delta.
\end{equation}

\vskip 0.2cm
\textbf{Step 1:} First order estimate of the solution to \eqref{LPE}.

Multiplying $2e^{-2\eta t}\mcq w$ on both sides of $\eqref{LPE}_1$, where $\mcq \defs \sum_{\ell=0}^3Q_{\ell}\partial_\ell$ will be chosen properly later. Then integrate by parts over $\omgrt$ with respect to $(y_0,\mathbf{y})$, we have
\begin{align}
&\int_{\omgrt}  e^{-2\eta y_0}\left(\mcq w L'w+\mcp(w,\md w)\right)\mathrm{d}y_0\mathrm{d}\mathbf{y}\nonumber\\
&\quad=\int_{\widetilde{\Omega}}\left[e^{-2\eta y_0}H_0\right]^{t=T}_{t=0}\mathrm{d}\mathbf{y}-\int_0^T \!\!e^{-2\eta y_0}H_1|_{y_1=0}\mathrm{d}y_0\mathrm{d}y_2\mathrm{d}y_3
+2\eta\int_{\omgrt}\!e^{-2\eta y_0}H_0\mathrm{d}y_0\mathrm{d}\mathbf{y}\label{1st oder estimate, 1st step}
\end{align}
where
\begin{equation}\label{defn of H0 and H1}
    H_i(\md w;\mcq)=2\sum_{j,\ell=0}^3r_{ij}\partial_jw Q_\ell\partial_\ell w-Q_i\sum_{j,\ell=0}^3r_{\ell j}\partial_\ell w\partial_jw\mbox{ }(i=0,1)
\end{equation}
and $\mcp(w,\md w)$ is a quadratic polynomial in $w$ and $\md w$ with bounded coefficients.
For later use, we also define $H_3$ by
\begin{equation}\label{defn of H2}
     H_3(\md w;\mcq)\defs 2\sum_{j,\ell=0}^3r_{ij}\partial_j w Q_\ell\partial_\ell w-Q_3\sum_{j,\ell=0}^3r_{\ell j}\partial_\ell w\partial_j w.
\end{equation}
It is easy to see
\[
\mathcal{P}\le C(w^2+\abs{\md w}^2).
\]
Choosing $\mcq$ appropriately as
$$
\mathcal{Q}=\widetilde{B}+\nu(\widetilde{B}-\mcn)+\abs{\frac{\nu r_{01}}{\widetilde{B}_0}}\widetilde{B}, $$
where
$
\tilde{B}=\frac{r_{11}}{b_1}\sum_{j=0}^3b_j\partial_j-\sum_{j=0}^3r_{j1}\partial_j$,
$
\mcn=-\sum_{j=0}^3r_{j1}\partial_j
$
and
$\nu=\sum_{i,j=0}^3r^{ij}\widetilde{B}_i\widetilde{B}_j$, where $\widetilde{B}_j$ is the coefficient in $\widetilde{B}$ in front of $\partial_j$. Then by simple calculation, we obtain
\begin{equation}\label{ineq:interior and bdry ineq}
    H_0(\mathrm{D} w)\geq C \abs{\mathrm{D}w}^2\quad\mbox{and}\quad
    -H_1(\mathrm{D} w)\geq C\left(\abs{\mathrm{D}w}^2+\abs{w}^2-\abs{\mathcal{B}w}^2\right).
\end{equation}
In view of $\eqref{1st oder estimate, 1st step}$, $\eqref{ineq:interior and bdry ineq}$, and the Cauchy inequality, one has
\begin{align}
    &\eta \|e^{-\eta t}\md w\|^2_{L^2(\omgrt)}+e^{-2\eta T}\|\md w(T,\cdot)\|_{L^2(\widetilde{\Omega})}^2+\|e^{-\eta t}\md w|_{y_1=0}\|^2_{L^2(\omega_T)}\nonumber\\
    &\quad\lesssim \frac{1}{\varepsilon\eta}\|e^{-\eta t}L'(u)w\|^2_{L^2(\omgrt)}+\varepsilon \eta\|e^{-\eta t}\md w\|_{L^2(\omgrt)}^2+\|e^{-\eta t}(\mcb w,w)\|^2_{L^2(\omega_T)}\nonumber\\
    &\quad\quad+\|(w,\partial_tw)|_{t=0}\|^2_{L^2(\omgrt)}.
\end{align}
Set $\varepsilon=\frac{1}{2C}$, then the second term on the right side is absorbed by the left hand-side term, hence we get
\begin{align*}
    &\eta \|e^{-\eta t}\md w\|^2_{L^2(\omgrt)}+e^{-2\eta T}\|\md w(T,\cdot)\|^2_{L^2(\widetilde{\Omega})}+\|e^{-\eta t}\md w|_{y_1=0}\|^2_{L^2(\omega_T)}\nonumber\\
    &\quad\le C\left(\frac{1}{\eta}\|e^{-\eta t}L'(u)w\|^2_{L^2(\omgrt)}+\|e^{-\eta t}\mcb w\|^2_{L^2(\omega_T)}+\|e^{-\eta t}w\|^2_{L^2(\omega_T)}\right)\nonumber\\
     &\qquad+C\|(w,\partial_tw)|_{t=0}\|^2_{L^2(\omgrt)}.
\end{align*}
Apply $\eqref{eq:3-13}$ to the boundary term of $w$ on the right hand-side of above inequality, then let $\eta$ be properly large, so that $\|e^{-\eta t}w\|^2_{L^2(\omega_T)}$ be absorbed by the left hand-side terms. Then we obtain
\begin{align}
    &\sum_{|\alpha|\le 1}\left(\eta \|e^{-\eta t}\md^\alpha w\|^2_{L^2(\omgrt)}+e^{-2\eta T}\md^\alpha w(T,\cdot)\|_{L^2(\widetilde{\Omega})}^2+\|e^{-\eta t}\md^
    \alpha w|_{y_1=0}\|^2_{L^2(\omega_T)}\right)\nonumber\\
    &\quad\le C\left(\frac{1}{\eta}\|e^{-\eta t}L'(u)w\|^2_{L^2(\omgrt)}+\|e^{-\eta t}\mcb w|_{y_1=0}\|^2_{L^2(\omega_T)}+\|\md w|_{t=0}\|^2_{L^2(\widetilde{\Omega})}\right).\label{first order estimate}
\end{align}
\vskip 0.2cm
\textbf{Step 2:} In this step, we will establish the second order estimate and the well-posedness of \eqref{LPE}. At the end of this step, we show that the unique solution to \eqref{LPE} is indeed a solution to \eqref{LP}.
Applying $\eqref{first order estimate}$ to $\partial_{y_0}w$, $\partial_{y_2}w$ and $\partial_{y_3}w$, we obtain that
\begin{align}
    &\eta \|e^{-\eta t}\md \partial_{y_\ell} w\|^2_{L^2(\omgrt)}+e^{-2\eta T}\|\partial_{y_\ell}w(T,\cdot)\|_{L^2(\widetilde{\Omega})}^2+\|e^{-\eta t}\md \partial_{y_\ell}w|_{y_1=0}\|^2_{L^2(\omega_T)}\nonumber\\
    &\quad\lesssim \frac{1}{\eta}\|e^{-\eta t}L'\partial_{y_\ell}w\|^2_{L^2(\omgrt)}+\|e^{-\eta t}\mcb \partial_{y_\ell}w|_{y_1=0}\|^2_{L^2(\omega_T)}+\|\md\partial_{y_\ell}w|_{t=0}\|^2_{L^2(\widetilde{\Omega})}\label{second order estimate, 1st step}
\end{align}
holds for $\ell=0,2,3$.
By $\eqref{LPE}_1$, one has
\begin{align}
    \partial^2_{y_1}w=\frac{1}{r_{11}}\left(L'w-\sum_{(i,j)\neq(1,1)}r_{ij}\partial_{ij}w-\sum_{i=0}^2r_i\partial_i w-rw\right).
\end{align}
Hence
\begin{align}
    &\eta \|e^{-\eta t} \partial^2_{y_1} w\|^2_{L^2(\omgrt)}+e^{-2\eta T}\|\partial^2_{y_1}w(T,\cdot)\|_{L^2(\widetilde{\Omega})}^2+\|e^{-\eta t} \partial^2_{y_1}w|_{y_1=0}\|^2_{L^2(\omega_T)}\nonumber\\
    &\quad\lesssim \sum_{\ell=0,2,3}\eta \|e^{-\eta t}\md \partial_{y_\ell} w\|^2_{L^2(\omgrt)}+e^{-2\eta T}\|\md\partial_{y_\ell}w|_{t=T}\|_{L^2(\widetilde{\Omega})}^2+\|e^{-\eta t}\md \partial_{y_\ell}w|_{y_1=0}\|^2_{L^2(\omega_T)}\nonumber\\
    &\quad\quad+ \sum_{|\alpha|\le 1}\eta \|e^{-\eta t}\md^\alpha w\|^2_{L^2(\omgrt)}+e^{-2\eta T}\|\md^\alpha w|_{t=T}\|_{L^2(\widetilde{\Omega})}^2+\|e^{-\eta t}\md^\alpha w|_{y_1=0}\|^2_{L^2(\omega_T)}\nonumber\\
    &\quad\quad+\eta \|e^{-\eta t}L'w\|^2_{L^2(\omgrt)}+e^{-2\eta T}\|L'w\|_{L^2(\widetilde{\Omega})}+\|e^{-\eta t}L'w|_{y_1=0}\|^2_{L^2(\omega_T)}.\label{second order estimate, 2nd step}
\end{align}
By $\eqref{first order estimate}$, $\eqref{second order estimate, 1st step}$ and $\eqref{second order estimate, 2nd step}$, we have
\begin{align}
    &\eta \|e^{-\eta t} \partial^2_{y_1} w\|^2_{L^2(\omgrt)}+e^{-2\eta T}\|\partial^2_{y_1}w(T,\cdot)\|_{L^2(\widetilde{\Omega})}^2+\|e^{-\eta t} \partial^2_{y_1}w|_{y_1=0}\|^2_{L^2(\omega_T)}\nonumber\\
    &\quad\lesssim \sum_{|\alpha|\le 1}(\frac{1}{\eta}\|e^{-\eta t}L'(\md^\alpha w)\|^2_{L^2(\omgrt)}+\|e^{-\eta t}\mathcal{B}\md^\alpha w\|^2_{L^2(\omega_T)})+\sum_{\ell=0,2,3}\|\md\partial_{y_\ell}w|_{t=0}\|^2_{L^2(\widetilde{\Omega})}\nonumber\\
    &\quad\quad+\eta \|e^{-\eta t} w\|^2_{L^2(\omgrt)}+e^{-2\eta T}\| w|_{t=T}\|_{L^2(\widetilde{\Omega})}+\|e^{-\eta t} w|_{y_1=0}\|^2_{L^2(\omega_T)}\nonumber\\
    &\quad\quad+\eta \|e^{-\eta t}L'w\|^2_{L^2(\omgrt)}+e^{-2\eta T}\|L'w\|_{L^2(\widetilde{\Omega})}+\|e^{-\eta t}L'w|_{y_1=0}\|^2_{L^2(\omega_T)}.
\end{align}
By integration by parts with respect to $t$ and the trace theorem, we have
\begin{align}
    &\eta\|e^{-\eta t}w\|^2_{L^2(\omgrt)}+e^{-2\eta T}\|w\|^2_{L^2(\widetilde{\Omega})}+\|e^{-\eta t} w|_{y_1=0}\|^2_{L^2(\omega_T)}\nonumber\\
    &\le \frac{1}{\eta}\|e^{-\eta t}\partial_tw\|^2_{L^2(\omgrt)}+\|w|_{t=0}\|^2_{L^2(\widetilde{\Omega})}+\sum_{|\alpha|\le 1}\|e^{-\eta t}\md^\alpha w\|^2_{L^2(\omgrt)}.\label{eq:3-13}
\end{align}
So by \eqref{eq:3-13} and Cauchy inequality, one has
\begin{align}
    &\eta \|e^{-\eta t}L'w\|^2_{L^2(\omgrt)}+e^{-2\eta T}\|L'w\|_{L^2(\widetilde{\Omega})}+\|e^{-\eta t}L'w|_{y_1=0}\|^2_{L^2(\omega_T)}\nonumber\\
    &\quad\lesssim \|L'w|_{t=0}\|^2_{L^2(\widetilde{\Omega})}+\frac{1}{\varepsilon\eta}\sum_{|\alpha|\le 1}(\|e^{-\eta t}\md^\alpha w\|^2_{L^2(\omgrt)}+\|e^{-\eta t}L'(\md^\alpha w)\|^2_{L^2(\omgrt)})\nonumber\\
    &\quad\quad+\varepsilon\eta\sum_{|\alpha|\le 2}\|e^{-\eta t}\md^\alpha w\|^2_{L^2(\omgrt)}.\label{eq:3-14}
\end{align}
In light of \eqref{second order estimate, 1st step}, $\eqref{second order estimate, 2nd step}$, $\eqref{eq:3-13}$ and $\eqref{eq:3-14}$, we obtain the estimate of $\partial^2_{y_1}w$, i.e.,
\begin{align}
    &\eta \|e^{-\eta t} \partial^2_{y_1} w\|^2_{L^2(\omgrt)}+e^{-2\eta T}\|\partial^2_{y_1}w(T,\cdot)\|_{L^2(\widetilde{\Omega})}+\|e^{-\eta t} \partial^2_{y_1}w|_{y_1=0}\|^2_{L^2(\omega_T)}\nonumber\\
    &\quad\lesssim  \frac{1}{\eta}\|e^{-\eta t}\partial_t w\|^2_{L^2(\omgrt)}+\sum_{|\alpha|\le 1}\|e^{-\eta t}\md^\alpha w\|^2_{L^2(\omgrt)}\nonumber\\
    &\quad\quad+\|\md w|_{t=0}\|^2_{L^2(\widetilde{\Omega})}+\sum_{|\alpha|\le 1}(\frac{1}{\eta}\|e^{-\eta t}L'(\md^\alpha w)\|^2_{L^2(\omgrt)}+\|e^{-\eta t}\mathcal{B}\md^\alpha w\|^2_{L^2(\omega_T)})\nonumber\\
    &\quad\quad+\frac{1}{\varepsilon\eta}\sum_{|\alpha|\le 1}(\|e^{-\eta t}\md^\alpha w\|^2_{L^2(\omgrt)}+\|e^{-\eta t}L'(\md^\alpha w)\|^2_{L^2(\omgrt)})+\|L'w|_{t=0}\|^2_{L^2(\widetilde{\Omega})}\nonumber\\
    &\quad\quad+\varepsilon\eta\sum_{|\alpha|\le 2}\|e^{-\eta t}\md^\alpha w\|^2_{L^2(\omgrt)}+\sum_{|\alpha|\le 2}\|\md^\alpha w|_{t=0}\|^2_{L^2(\widetilde{\Omega})}.\label{eq:3-15}
\end{align}
Add up \eqref{first order estimate}, $\eqref{second order estimate, 1st step}$ for $\ell=0,2,3$ and $\eqref{eq:3-15}$, then set $\varepsilon$ and $\frac{1}{\eta}$ to be properly small, we have
\begin{align}
    &\sum_{|\alpha|\le 2}\eta \|e^{-\eta t} \md^\alpha w\|^2_{L^2(\omgrt)}+e^{-2\eta T}\|\md^\alpha w(T,\cdot)\|_{L^2(\widetilde{\Omega})}+\|e^{-\eta t} \md^\alpha w|_{y_1=0}\|^2_{L^2(\omega_T)}\nonumber\\
    &\quad\lesssim \frac{1}{\eta}\sum_{|\alpha|\le 1}\|e^{-\eta t}L'(\md^\alpha w)\|^2_{L^2(\omgrt)}+\|L'w|_{t=0}\|^2_{L^2(\widetilde{\Omega})}+\|e^{-\eta t}\mathcal{B}\md^\alpha w\|^2_{L^2(\omega_T)}\nonumber\\
    &\quad\quad+\sum_{|\alpha|\le 2}\|\md^\alpha w|_{t=0}\|^2_{L^2(\widetilde{\Omega})}\nonumber\\
    &\quad\lesssim \sum_{|\alpha|\le 2}\frac{1}{\eta}\left(\|e^{- \eta t}\md^\alpha w\|^2_{L^2(\omgrt)}+\|e^{-\eta t} \md^\alpha L'w\|^2_{L^2(\omgrt)}\right)+\sum_{|\alpha|\le 1}\|e^{-\eta t}\md^\alpha g|_{y_1=0}\|^2_{L^2(\omega_T)}.\label{eq:3-16}
\end{align}
Let $\eta$ be properly large, we obtain
\begin{align}
    &\sum_{|\alpha|\le 2}\eta \|e^{-\eta t} \md^\alpha w\|^2_{L^2(\omgrt)}+e^{-2\eta T}\|\md^\alpha w(T,\cdot)\|_{L^2(\widetilde{\Omega})}+\|e^{-\eta t} \md^\alpha w|_{y_1=0}\|^2_{L^2(\omega_T)}\nonumber\\
    &\quad\lesssim \frac{1}{\eta}\|e^{-\eta t}f\|^2_{H^1(\omgrt)}+\sum_{|\alpha|\le 1}\|e^{-\eta t}\md^\alpha g|_{y_1=0}\|^2_{L^2(\omega_T)}.\label{eq:3-18}
\end{align}
Based on energy estimate \eqref{eq:3-18}, it is easy to obtain the existence of an $H^2(\widetilde{\Omega}_T)$ solution $w$ of problem \eqref{LPE}. In fact, the existence of \eqref{LPE} has been proved in \cite[Theorem 3.3]{MT1987}, when the coefficients and source terms belong to $H^s(\widetilde{\Omega}_T)$ with $s>\left[\frac{N+1}{2}\right]+1$, where $N$ is the space dimension. Though the regularity of coefficients and source terms of \eqref{LPE} is not enough, we can still deduce the existence of \eqref{LPE}. Firstly, one mollifies the coefficients and the source terms by the convolution of the classical Friedrichs mollifier $\rho_{\varepsilon}$, then by \cite[Theorem 3.3]{MT1987}, there exists a smooth solution $w^\varepsilon$ to the regularized problem for each $\varepsilon>0$. Thanks to our uniform $H^2_{\eta}(\widetilde{\Omega}_T)$ estimate \eqref{eq:3-18}, $\{w^\epsilon\}_{\epsilon>0}$ is strongly compact in $H^1_{\eta}(\widetilde{\Omega}_T)$ and weakly compact in $H^2_{\eta}(\widetilde{\Omega}_T)$. Then passing the limit by letting $\epsilon\rightarrow 0^+$ in the regularized equation, we obtain a $H^2_{\eta}$-solution to the linear problem $\eqref{LPE}$.
If $f=g\equiv 0$, \eqref{eq:3-18} implies $w\equiv 0$ in $\omgrt$. This indicates that the solution to $\eqref{LPE}$ is unique, since $\eqref{LPE}$ is a linear problem. Due to our extension, it is easy to check that $w(y_0,y_1,y_2,-y_3)$ is also a solution to $\eqref{LPE}$. By the uniqueness, we have $w(y_0,y_1,y_2,y_3)=w(y_0,y_1,y_2,-y_3)$ for all $(y_0,\mby)\in \widetilde{\Omega}_T$. Differentiating with respect to $y_3$ on both sides of this equality  and letting $y_3=0$, one has
\[
\partial_{y_3}w|_{y_3=0}=-\partial_{y_3}w|_{y_3=0},
\]
which implies $\partial_{y_3}w|_{y_3=0}=0$. From \eqref{eq:3-18} and the trace theorem, we know $\partial_{y_3}w$ is  a $L^2$ function on $\{y_3=0\}$, so above process makes sense. Therefore, we conclude that the unique solution to $\eqref{LPE}$ is indeed the unique solution to $\eqref{LP}$. 

\vskip 0.2cm
\textbf{Step 3:} In the remaining steps, we will improve the regularity of the unique solution by deriving higher order estimate in the dihedral-space domain. In this step, we will consider the third order estimate.
Since both $\partial_{y_0}$ and $\partial_{y_2}$ are tangential to both the solid wall $\Gamma_w$ and shock front $\Gamma_s$ and all the coefficients are smooth in the directions of $y_0$ and $y_2$.
We can apply the first inequality of \eqref{eq:3-16} to $\partial_{y_0}w$ and $\partial_{y_2}w$ respectively, to obtain
\begin{align}
    &\|e^{-\eta t}\md^\alpha \partial_{y_i}w\|^2_{L^2(\omgrt)}+e^{-2\eta T}\|\md^\alpha \partial_{y_i}w\|^2_{L^2(\widetilde{\Omega})}+\|e^{-\eta t}\md^\alpha \partial_{y_i}w|_{y_1=0}\|^2_{L^2(\omega_T)}\nonumber\\
    &\qquad\lesssim \sum_{|\alpha|\le 1}(\frac{1}{\eta}\|e^{-\eta t}L'(\md^\alpha\partial_{y_i} w)\|^2_{L^2(\omgrt)}+\|e^{-\eta t}\mathcal{B}\md^\alpha\partial_{y_i}w\|^2_{L^2(\omega_T)})\nonumber\\
    &\qquad\quad+\|\md f|_{t=0}\|^2_{L^2(\widetilde{\Omega})}\label{ineq:3-18}
\end{align}
for $i=0,2$. Here we use the fact that $\sum_{|\alpha|\le 3}\|\md^\alpha w|_{t=0}\|^2_{L^2(\widetilde{\Omega})}\lesssim \sum_{j=0}^1\|\md f|_{t=0}\|^2_{L^2(\widetilde{\Omega})}$, which comes from the equation and the initial data. In the coming steps, the estimate we obtained in each step will be applied to $\partial_{y_0}w$ and $\partial_{y_2}w$ in the next step, because of the same reason as stated above.
To control all other derivatives of third order, we need to estimate derivatives in the form of $\partial^{k_1}_{y_1}\partial^{k_2}_{y_2}w$ with $k_1+k_2=3$.
Due to the limit of the regularity of the extended coefficients, we cannot obtain higher order estimate in $\omgrt$ directly. In the following steps, all estimates are restricted to the cornered time spatial domain $\Omega_T$. Before going on, we present the following lemma.
\begin{lem}\label{Multiplier for two Dirichlet boundary conditions}
Let $H_m$ be defined as in \eqref{defn of H0 and H1} and \eqref{defn of H2}. For any given $r_{ij}$ satisfying assumptions \ref{h1}-\ref{h4},
we can find a multiplier $\mcq^d=\sum_{i=0}^3Q^d_i$ such that
\begin{align}
&H_0(\mathrm{D}w;\mcq^d)\ge C_1\abs{\nabla_{\mby}w}^2-C_2\abs{\partial_{y_0}w}^2,\label{ineq:3-19}\\		-&H_1(\mathrm{D}w;
\mcq^d)\ge C_1\abs{\partial_{y_1}w}^2-C_2(\abs{\partial_{y_0}w}^2+\abs{\partial_{y_2}w}^2+\abs{\partial_{y_3}w}^2),\label{ineq:3-20}
\end{align}
where $\nabla_{\mby}\defs(\partial_{y_1},\partial_{y_2},\partial_{y_3})$.
Moreover, if $w=0$ on $\{y_3=0\}$, then
\begin{align}\label{ineq:3-21}
-&H_3(\mathrm{D}w;\mcq^d)\ge C\abs{\partial_{y_3}w}^2\quad\mbox{on}\quad\{y_3=0\}.
\end{align}
\end{lem}
\begin{proof}
%    The proof of this lemma is completely same as that of lemma $4.2$ of \cite{FXX}, so we omit the details.
    It is convenient to denote $\partial_{y_i}w$ by $\xi_i$ for $i=0,1,2,3$.
    At the background solution, by simple calculation, one has
    \begin{align}
    	-H_1(\md w;\mathcal{Q}^d)&=(-2r_{10}Q^d_0+r_{00}Q^d_1)\xi_0^2+Q^d_1(-r_{11}\xi_1^2+r_{22}\xi_2^2+r_{33}\xi_3^2)\nonumber\\
    	&\quad-2r_{11}Q^d_0\xi_0\xi_1-2r_{10}\xi_0\xi_2+2Q^d_3r_{10}\xi_0\xi_3-2r_{11}Q^d_2\xi_1\xi_2\nonumber\\
    	&\quad-2r_{11}Q^d_3\xi_1\xi_3.
    \end{align}
    Choosing $Q^d_1$ such that $-Q^d_1r_{11}>0$, then \eqref{ineq:3-20} follows easily. At the background solution $u_b$, we know $r_{11}=-\frac{c_+^2}{(q_--q_+)^2}<0$. So one just needs to let $Q^d_1>0$.
    For $H_0(\md w;\mathcal{Q}^d)$, at the background solution one has
   \begin{align}
   	H_0(\md w;\mathcal{Q}^d)&=2r_{10}\xi_1(Q^d_0\xi_0+Q^d_1\xi_1+Q^d_2\xi_2+Q^d_3\xi_3)\nonumber\\
   	&\quad-Q^d_0(r_{00}\xi_0^2+r_{11}\xi_1^2+r_{22}\xi_2^2+r_{33}\xi_3^2+2r_{10}\xi_0\xi_1).
   \end{align}
   If we can let the coefficient before $\xi^2_1$ be positive, then \eqref{ineq:3-20} follows immediately. In fact, it suffices to let $2r_{10}Q^d_1-Q^d_0r_{11}>0$. Since $Q^d_1$ has been set to be positive, $r_{10}>0$, and $r_{11}<0$ at the background solution, it is sufficient to let $Q^d_0$ be positive.
   At the background solution, one has
   \begin{align}
   	-H_3(\md w;\mathcal{Q}^d_1)=-r_{33}Q^d_3\xi^2_3.
   \end{align}
   Hence \eqref{ineq:3-21} follows if we let $Q^d_3>0$, since $r_{33}<0$ at the background solution.
\end{proof}
Armed with lemma \ref{Multiplier for two Dirichlet boundary conditions}, we can obtain the first order estimate of $\partial_{y_1y_3}w$. In fact, $\partial_{y_1y_3}w$ satisfies
\begin{equation}\label{equation-satisfied-by-partial-y1y3w}
\left\{\!\!
\begin{array}{ll}
L'\left(\partial_{y_1y_3}w\right)=\partial_{y_1y_3}f-[\partial_{y_1y_3},L']w\qquad&\mbox{in }\Omega_T,\\
\partial_{y_1y_3}w=0,\qquad&\mbox{on }\omega^r_T,\\
\partial_{y_1y_3}w=\Lambda,\qquad&\mbox{on }
\omega^\ell_T,\\
(\partial_{y_0}(\partial_{y_1y_3}w),\partial_{y_1y_3}w)=(0,0)\qquad&\mbox{on }\Gamma_{in}.
\end{array}
\right.
\end{equation}
where
\begin{align}\label{defn of Lambda}
\Lambda=&\frac{1}{b_1}\left(\mathcal{B}(\partial_{y_3}w)-b_0\partial_{y_0y_3}w-b_2\partial_{y_2y_3}w-b_3\partial_{y_3y_3}w-b\partial_{y_3}w\right).
\end{align}
Problem \eqref{equation-satisfied-by-partial-y1y3w} is an initial boundary value problem in a dihedral space-domain with two Dirichlet boundary conditions.
Multiplying $\eqref{equation-satisfied-by-partial-y1y3w}_1$ by $2e^{-2 \eta t}\mcq^d (\partial_{y_1y_3}w)$, where $\mcq^d$ is given in lemma \ref{Multiplier for two Dirichlet boundary conditions}. Then integrating on both sides with respect to $(y_0,\mathbf{y})$ over $\Omega_T$ and by using Cauchy inequality, we obtain
\begin{align}
    &\eta\int_{\Omega_T} e^{-2\eta t}H_0\mathrm{d}y_0\mathrm{d}\mby+e^{-2\eta T}\int_{\Omega}H_0|_{t=T}\mathrm{d}\mby-\int_{\omega^\ell_T}e^{-2\eta t}H_1\mathrm{d}y_0\mathrm{d}\mby-\int_{\omega^r_T}e^{-2\eta t}H_3\mathrm{d}y_0\mathrm{d}\mby\nonumber\\
    &\qquad\lesssim\frac{1}{\varepsilon\eta}\|e^{-\eta t}L'(\partial_{y_1y_3}w)\|^2_{L^2(\Omega_T)}+(\varepsilon \eta+1)\|e^{-\eta t}\md \partial_{y_1y_3}w\|^2_{L^2(\Omega_T)}.\label{ineq:3-24}
\end{align}
By lemma \ref{Multiplier for two Dirichlet boundary conditions}, one knows that
\begin{align}
    -H_3&\geq C|\partial_{y_3}(\partial_{y_1y_3}w)|^2,\label{ineq:3-25}\\
     -H_1
    &\geq C_1|\partial_{y_1}(\partial_{y_1y_3}w)|^2-C_2(|\partial_{y_0}\Lambda|^2+|\partial_{y_2}\Lambda|^2+|\partial_{y_3}\Lambda|^2),\label{ineq:3-26}\\
    H_0&\geq C_1\abs{\nabla_{\mby}\partial_{y_1y_3}w}^2-C_2\abs{\partial_{y_0}(\partial_{y_1y_3}w)}^2.\label{ineq:3-27}
\end{align}
From \eqref{ineq:3-24}-\eqref{ineq:3-27} and letting $\varepsilon$ be properly small, one obtains
\begin{align}
    &\eta\|\nabla_{\mby}\partial_{13}w\|^2_{L^2(\Omega_T)}
    +e^{-2\eta T}\|\nabla_{\mby}\partial_{13}w\|^2_{L^2(\Omega)}+\|e^{-\eta t} \partial_{113}w\|^2_{L^2(\omega_T^{\ell})}+\|e^{-\eta t} \partial_{313}w\|^2_{L^2(\omega_T^{r})}\nonumber\\
    &\quad\quad\lesssim\frac{1}{\eta}\|e^{-\eta t}L'( \partial_{13}w)\|^2_{L^2(\Omega_T)}+\eta\|e^{-\eta t}\partial_{0}\partial_{13}w\|^2_{L^2(\Omega_T)}+e^{-2\eta T}\|\partial_{0}\partial_{13}w\|^2_{L^2(\Omega)}\nonumber\\
    &\qquad\quad+\|e^{-\eta t}(\partial_{y_0}\Lambda,\partial_{y_2}\Lambda,\partial_{y_3}\Lambda)\|^2_{L^2(\omega^\ell_T)}.\label{ineq:3-28}
\end{align}
From $\eqref{defn of Lambda}$, we have
\begin{align}
&\abs{\partial_{y_0}\Lambda}+\abs{\partial_{y_2}\Lambda}+\abs{\partial_{y_3}\Lambda}\nonumber\\
&\qquad\quad\lesssim\sum_{j\neq 1}\abs{\mathcal{B}\partial_{3j}w}+\sum_{|\alpha|\le 2}(\abs{\md^\alpha\partial_{y_0}w}+\abs{\md^\alpha w})\nonumber\\
&\qquad\qquad+\|b_2\|_{L^\infty}(\abs{\partial_{223}w}+\abs{\partial_{233}w})+\|b_3\|_{L^\infty}\abs{\partial_{333}w}.\label{ineq:3-29}
\end{align}
Combining $\eqref{eq:3-18}$, $\eqref{ineq:3-28}$ and $\eqref{ineq:3-29}$ and the second order estimate, we obtain
\begin{align}
    &\eta\|\nabla_{\mby}\partial_{13}w\|^2_{L^2(\Omega_T)}
    +e^{-2\eta T}\|\nabla_{\mby}\partial_{13}w\|^2_{L^2(\Omega)}+\|e^{-\eta t} \partial_{113}w\|^2_{L^2(\omega_T^{\ell})}+\|e^{-\eta t} \partial_{313}w\|^2_{L^2(\omega_T^{r})}\nonumber\\
    &\quad\lesssim\sum_{|\alpha|\le 1}(\frac{1}{\eta}\|e^{-\eta t}L'(\md^\alpha\partial_{y_0} w)\|^2_{L^2(\Omega_T)}+\|e^{-\eta t}\mathcal{B}\md^\alpha\partial_{y_0}w\|^2_{L^2(\omega_T^\ell)})\nonumber\\
    &\qquad+\frac{1}{\eta}\sum_{|\alpha|\le 1}\|e^{-\eta t}L'(\md^\alpha w)\|^2_{L^2(\Omega_T)}+\|e^{-\eta t}\mathcal{B}\md^\alpha w\|^2_{L^2(\omega_T^\ell)}+\frac{1}{\eta}\|e^{-\eta t}L'( \partial_{13}w)\|^2_{L^2(\Omega_T)}\nonumber\\
    &\qquad+\|b_2\|^2_{L^\infty(\Omega_T)}\cdot\|(e^{-\eta t}\partial_{223}w,e^{-\eta t}\partial_{233}w)\|^2_{L^2(\omega^\ell_T)}+\sum_{|j\neq 1|}\|e^{-\eta t}\mcb\partial_{3j}w\|^2_{L^2(\omega_T^\ell)}\nonumber\\
&\qquad+\|b_3\|^2_{L^\infty(\omega_T^\ell)}\cdot\|e^{-\eta t}\partial_{333}w\|^2_{L^2(\omega_T^\ell)}+\|\md f|_{t=0}\|^2_{L^2(\Omega)}\nonumber\\
&\quad\lesssim\sum_{|\alpha|\le 2}\left( \frac{1}{\eta}\|e^{-\eta t}L'(\md^\alpha w)\|^2_{L^2(\Omega_T)}+\|e^{-\eta t}\mathcal{B}\md^\alpha w\|^2_{L^2(\omega_T^\ell)}\right)\nonumber\\
&\qquad+\|b_2\|^2_{L^\infty(\Omega_T)}\cdot(\|e^{-\eta t}\partial_{223}w\|^2_{L^2(\omega^\ell_T)}+\|e^{-\eta t}\partial_{233}w\|^2_{L^2(\omega_T^\ell)})\nonumber\\
&\qquad+\|b_3\|^2_{L^\infty(\omega_T^\ell)}\cdot\|e^{-\eta t}\partial_{333}w\|^2_{L^2(\omega_T^\ell)}+\|\md f|_{t=0}\|^2_{L^2(\widetilde{\Omega})}.\label{ineq:3-30}
\end{align}
But \eqref{ineq:3-29} implies that
\begin{align}
    &\|e^{-\eta t}\partial_{213}w\|^2_{L^2(\omega^\ell_T)}+|e^{-\eta t}\partial_{313}w\|^2_{L^2(\omega^\ell_T)}\nonumber\\
    &\qquad\lesssim\sum_{|\alpha|\le 2}\left( \frac{1}{\eta}\|e^{-\eta t}L'(\md^\alpha w)\|^2_{L^2(\Omega_T)}+\|e^{-\eta t}\mathcal{B}\md^\alpha w\|^2_{L^2(\omega_T^\ell)}\right)\nonumber\\
&\qquad\quad+\|b_2\|^2_{L^\infty(\Omega_T)}\cdot(\|e^{-\eta t}\partial_{223}w\|^2_{L^2(\omega^\ell_T)}+\|e^{-\eta t}\partial_{233}w\|^2_{L^2(\omega_T^\ell)})\nonumber\\
&\quad\qquad+\|b_3\|^2_{L^\infty(\omega_T^\ell)}\cdot\|e^{-\eta t}\partial_{333}w\|^2_{L^2(\omega_T^\ell)}+\|\md f|_{t=0}\|^2_{L^2(\widetilde{\Omega})}.\label{ineq:3-31}
\end{align}
Then the sum of \eqref{ineq:3-30} and \eqref{ineq:3-31} indicates that
\begin{align}
     &\eta\|\nabla_{\mby}\partial_{13}w\|^2_{L^2(\Omega_T)}
    +e^{-2\eta T}\|\nabla_{\mby}\partial_{13}w\|^2_{L^2(\Omega)}+\|e^{-\eta t} \nabla_{\mby}\partial_{13}w\|^2_{L^2(\omega_T^{\ell})}+\|e^{-\eta t} \partial_{313}w\|^2_{L^2(\omega_T^{r})}\nonumber\\
     &\qquad\lesssim\sum_{|\alpha|\le 2}\left( \frac{1}{\eta}\|e^{-\eta t}L'(\md^\alpha w)\|^2_{L^2(\Omega_T)}+\|e^{-\eta t}\mathcal{B}\md^\alpha w\|^2_{L^2(\omega_T^\ell)}\right)\nonumber\\
&\qquad\quad+\|b_2\|^2_{L^\infty(\Omega_T)}\cdot(\|e^{-\eta t}\partial_{223}w\|^2_{L^2(\omega^\ell_T)}+\|e^{-\eta t}\partial_{233}w\|^2_{L^2(\omega_T^\ell)})\nonumber\\
&\quad\qquad+\|b_3\|^2_{L^\infty(\omega_T^\ell)}\cdot\|e^{-\eta t}\partial_{333}w\|^2_{L^2(\omega_T^\ell)}+\|\md f|_{t=0}\|^2_{L^2(\widetilde{\Omega})}.\label{ineq:3-32}
\end{align}
By \ref{h2} and \ref{h3}, we know that $\|b_2\|_{L^\infty}$ and $\|b_3\|_{L^\infty}$ are small, provided the $\delta$ in \ref{h3} is set to be sufficiently small. It will be shown later that the third order derivatives on the right hand-side of \eqref{ineq:3-32} can be absorbed by the left hand-side terms.

 Armed with the second order estimate of $\partial_{y_0}w$ and $\partial_{y_2}w$ and the estimate of $\nabla_{\mby}\partial_{13}w$, one can deduce the estimate of other third order derivatives. It is easy to see
\begin{align}
    \partial_{111}w=\frac{1}{r_{11}}\left(L'(\partial_{y_1}w)-\sum_{(i,j)\neq(1,1)}r_{ij}\partial_{ij}\partial_{y_1}w-\sum_{j=0}^3r_j\partial_j\partial_{y_1}w-r\partial_{y_1}w\right).\label{ineq:3-35}
\end{align}
Hence one has
\begin{align}
    \abs{\partial_{111}w}\lesssim \abs{L'(\partial_{y_1}w)}+\abs{\nabla_{\mby}\partial_{13}w}+\sum_{|\alpha|\le 2}(\abs{\md^\alpha \partial_{y_0}w}+\abs{\md^\alpha \partial_{y_2}w}+\abs{\md^\alpha w}).
\end{align}
This leads to the estimate of $\partial_{111}w$ in terms of the controlled terms on the right hand-side of \eqref{ineq:3-35}. In fact, one has
\begin{align}
    &\eta\|e^{-\eta t}\partial_{111}w\|^2_{L^2(\Omega_T)}+e^{-2\eta T}\|\partial_{111}w\|^2_{L^2(\Omega)}+\|e^{-\eta t}\partial_{111}w\|^2_{L^2(\omega^\ell_T)}\nonumber\\
    &\qquad\lesssim\sum_{|\alpha|\le 2}\left( \frac{1}{\eta}\|e^{-\eta t}L'(\md^\alpha w)\|^2_{L^2(\Omega_T)}+\|e^{-\eta t}\mathcal{B}\md^\alpha w\|^2_{L^2(\omega_T^\ell)}\right)\nonumber\\
&\qquad\quad+\|b_2\|^2_{L^\infty(\Omega_T)}\cdot(\|e^{-\eta t}\partial_{223}w\|^2_{L^2(\omega^\ell_T)}+\|e^{-\eta t}\partial_{233}w\|^2_{L^2(\omega_T^\ell)})\nonumber\\
&\qquad\quad+\|b_3\|^2_{L^\infty(\omega_T^\ell)}\cdot\|e^{-\eta t}\partial_{333}w\|^2_{L^2(\omega_T^\ell)}+\|\md f|_{t=0}\|^2_{L^2(\widetilde{\Omega})}.\label{ineq:3-36}
\end{align}
For $\partial_{333}w$, we have
\begin{align}
    \partial_{333}w=\frac{1}{r_{33}}\left(L'(\partial_{y_3}w)-\sum_{(i,j)\neq(3,3)}r_{ij}\partial_{ij}\partial_{y_3}w-\sum_{j=0}^3r_j\partial_j\partial_{y_3}w-r\partial_{y_3}w\right).
\end{align}
It is clear that $\partial_{333}w$ is the finite combination of $\md^2\partial_{y_0}w$, $\md^2\partial_{y_2}w$ $\nabla_{\mby}\partial_{13}w$ and lower order terms, whose estimate has been established. Hence we conclude that
\begin{align}
    &\eta\|e^{-\eta t}\partial_{333}w\|^2_{L^2(\Omega_T)}+e^{-2\eta T}\|\partial_{333}w\|^2_{L^2(\Omega)}+\|e^{-\eta t}\partial_{333}w\|^2_{L^2(\omega^\ell_T)}\nonumber\\
    &\qquad\lesssim\sum_{|\alpha|\le 2}\left( \frac{1}{\eta}\|e^{-\eta t}L'(\md^\alpha w)\|^2_{L^2(\Omega_T)}+\|e^{-\eta t}\mathcal{B}\md^\alpha w\|^2_{L^2(\omega_T^\ell)}\right)\nonumber\\
&\quad\qquad+\|b_2\|^2_{L^\infty(\Omega_T)}\cdot(\|e^{-\eta t}\partial_{223}w\|^2_{L^2(\omega^\ell_T)}+\|e^{-\eta t}\partial_{233}w\|^2_{L^2(\omega_T^\ell)})\nonumber\\
&\quad\qquad+\|b_3\|^2_{L^\infty(\omega_T^\ell)}\cdot\|e^{-\eta t}\partial_{333}w\|^2_{L^2(\omega_T^\ell)}+\|\md f|_{t=0}\|^2_{L^2(\widetilde{\Omega})}.\label{ineq:3-37}
\end{align}
It is easy to see that $\md^2\partial_{y_0}w$, $\md^2\partial_{y_2}w$, $\nabla_{\mby}\partial_{13}w$, $\partial_{111}w$, and $\partial_{333}w$ cover all third order derivatives of $w$. Thus by adding \eqref{ineq:3-18} for $i=0,2$, \eqref{ineq:3-32}, \eqref{ineq:3-35} and \eqref{ineq:3-37} together, we obtain
\begin{align}
    &\sum_{|\alpha|\le 3}\eta\|e^{-\eta t}\md^\alpha w\|^2_{L^2(\Omega_T)}+e^{-2\eta T}\|\md^\alpha w\|^2_{L^2(\Omega)}+\|e^{-\eta t}\md^\alpha w\|^2_{L^2(\omega^\ell_T)}\nonumber\\
     &\qquad\lesssim\sum_{|\alpha|\le 2}\left( \frac{1}{\eta}\|e^{-\eta t}L'(\md^\alpha w)\|^2_{L^2(\Omega_T)}+\|e^{-\eta t}\mathcal{B}\md^\alpha w\|^2_{L^2(\omega_T^\ell)}\right)\nonumber\\
&\qquad\quad+\|b_2\|^2_{L^\infty(\Omega_T)}\cdot\left(\|e^{-\eta t}\partial_{223}w\|^2_{L^2(\omega^\ell_T)}+\|e^{-\eta t}\partial_{233}w\|^2_{L^2(\omega_T^\ell)}\right)\nonumber\\
&\quad\qquad+\|b_3\|^2_{L^\infty(\omega_T^\ell)}\cdot\|e^{-\eta t}\partial_{333}w\|^2_{L^2(\omega_T^\ell)}+\|\md f|_{t=0}\|^2_{L^2(\widetilde{\Omega})}.\label{ineq:3-38}
\end{align}
As stated before, let the $\delta$ in \ref{h3} be properly small, such that the boundary integrals on $\omega^\ell_T$ on the right hand-side of \eqref{ineq:3-38} be absorbed by the left hand-side terms. Then we conclude the third order estimate as follows
\begin{align}
    &\sum_{|\alpha|\le 3}\eta\|e^{-\eta t}\md^\alpha w\|^2_{L^2(\Omega_T)}+e^{-2\eta T}\|\md^\alpha w\|^2_{L^2(\Omega)}+\|e^{-\eta t}\md^\alpha w\|^2_{L^2(\omega^\ell_T)}\nonumber\\
     &\qquad\lesssim\sum_{|\alpha|\le 2}\left( \frac{1}{\eta}\|e^{-\eta t}L'(\md^\alpha w)\|^2_{L^2(\Omega_T)}+\|e^{-\eta t}\mathcal{B}\md^\alpha w\|^2_{L^2(\omega_T^\ell)}\right)+\|\md f|_{t=0}\|^2_{L^2(\widetilde{\Omega})}.\label{ineq:3-39}
\end{align}

\vskip 0.2cm
\textbf{Step 4:}
In this step, we will establish the fourth order estimate in the dihedral-space domain.
Applying \eqref{ineq:3-38} to functions $\partial_{y_0}w$ and $\partial_{y_2}w$, respectively, one obtains
\begin{align}
&\sum_{i=0,2}\sum_{|\alpha|\le 3}\eta\|e^{-\eta t}\md^\alpha\partial_{y_i} w\|^2_{L^2(\Omega_T)}+e^{-2\eta T}\|\md^\alpha \partial_{y_i}w\|^2_{L^2(\Omega)}+\|e^{-\eta t}\md^\alpha \partial_{y_i}w\|^2_{L^2(\omega^\ell_T)}\nonumber\\
&\qquad\lesssim\sum_{|\alpha|\le 3}\left( \frac{1}{\eta}\|e^{-\eta t}L'(\md^\alpha w)\|^2_{L^2(\Omega_T)}+\|e^{-\eta t}\mathcal{B}\md^\alpha w\|^2_{L^2(\omega_T^\ell)}\right)+\|\md^2 f|_{t=0}\|^2_{L^2(\Omega)}.\label{ineq:3-40}
\end{align}
Before going on, we first prove the following lemma, which is crucial to the fourth order estimate.
\begin{lem}\label{Multiplier for almost normal boundary conditions}
Let $H_m$ be defined as in \eqref{defn of H0 and H1} and \eqref{defn of H2}. For any given $r_{ij}$ satisfying assumptions \ref{h1}-\ref{h4},
we can find a multiplier $\mcq^e=\sum_{i=0}^3Q^e_i$ such that
\begin{align}
&H_0(\mathrm{D}w;\mcq^e)\ge C_1\abs{\nabla_{\mby}w}^2-C_2\abs{\partial_{y_0}w}^2,\label{ineq:3-41}\\		
-&H_1(\mathrm{D}w;
\mcq^e)\ge C_1\abs{\partial_{y_3}w}^2-C_2(\abs{\partial_{y_0}w}^2+\abs{\partial_{y_1}w}^2+\abs{\partial_{y_2}w}^2),\label{ineq:3-42}
\end{align}
where $\nabla_{\mby}\defs(\partial_{y_1},\partial_{y_2},\partial_{y_3})$.
Moreover, if $w=0$ on $\{y_2=0\}$, then
\begin{align}\label{ineq:3-43}
-&H_3(\mathrm{D}w;\mcq^e)\ge C\abs{\partial_{y_3}w}^2\qquad\mbox{on }\{y_3=0\}.
\end{align}
\end{lem}
\begin{proof}
For the ease of presentation, in the proof of this lemma, denote $\md w$ by $(\xi_0,\xi_1,\xi_2,\xi_3)$.
Then at the background solution $u_b$, we have
\begin{align}	
-H_1(\md w;\mcq^e)=&(-2r_{10}Q^e_0+r_{00}Q^e_1)\xi_0^2+Q^e_1(-r_{11}\xi_1^2+r_{22}\xi_2^2+r_{33}\xi_3^2)\nonumber\\
&-2r_{11}Q^e_0\xi_0\xi_1-2r_{10}Q^e_2\xi_0\xi_2 +2r_{10}Q^e_3\xi_0\xi_3 -2r_{11}Q^e_2\xi_1\xi_2\nonumber\\
&-2r_{11}Q^e_3\xi_1\xi_3.
\end{align}
Choose $Q_1^e$ such that
\begin{equation}
	r_{33}Q^e_1>0,
\end{equation}
then \eqref{ineq:3-42} follows easily.
We know that $r_{33}=-\frac{c_+^2}{(q_--q_+)^2}<0$ at the background solution $u_b$. So we just need to let
	\begin{equation}
	Q^e_1<0.
	\end{equation}
	
Next, since at the background solution, we have
\begin{align}
    -H_3(\md w;\mcq^e)=-r_{33}Q^e_3\abs{\partial_{y_3}w}^2,
\end{align}
\eqref{ineq:3-43} follows if we let $Q^e_3>0$.	
Finally, for \eqref{ineq:3-41}, at the background solution $u_b$, we know
\begin{align}
    H_0(\md w;\mcq^e)=&2r_{10}\xi_1(Q^e_0\xi_0+Q^e_1\xi_1+Q^e_2\xi_2+Q^e_3\xi_3)\nonumber\\
    &-Q^e_0(r_{00}\xi_0^2+r_{11}\xi_1^2+r_{22}\xi_2^2+r_{33}\xi_3^2+2r_{10}\xi_0\xi_1)\nonumber\\
    =&(2r_{10}Q^e_1-r_{11}Q^e_0)\xi_1^2-Q^e_0(r_{22}\xi_2^2+r_{33}\xi_3^2)\nonumber\\
    &+2r_{10}\xi_1(Q^e_0\xi_0+Q^e_2\xi_2+Q^e_3\xi_3)-Q^e_0(r_{00}\xi_0^2+2r_{10}\xi_0\xi_1)\nonumber\\
    \ge&(2r_{10}Q^e_1-r_{11}Q^e_0-r_{10}|Q^e_2|-r_{10}Q^e_3)\xi_1^2+(-Q^e_0r_{22}-r_{10}|Q^e_2|)\xi_2^2\nonumber\\
    &+(-Q^e_0r_{33}-r_{10}Q^e_3)\xi_3^2-Q^e_0r_{00}\xi_0^2.
\end{align}
%\begin{equation}
%	\mathcal{H}_0(\mathrm{D}w;\mcq^e)\ge\frac12\abs{\partial_{y_1}w}^2\left(2r_{10}Q^e_1-Q^e_0r_{11}+r_{10}Q^e_2\right)+\frac12\abs{\partial_{y_2}v}^2\left(Q^e_2r_{10}-Q^e_0r_{22}\right)-C|\partial_{y_0}w|^2.
%\end{equation}
Because $r_{11}$, $r_{22}$ and $r_{33}$ are negative,
we may let \begin{equation}
	Q^e_0>\max\left\{\frac{-2r_{10}Q^e_1+r_{10}|Q^e_2|+r_{10}Q^e_3}{-r_{11}},\frac{r_{10}|Q^e_2|}{-r_{22}},\frac{r_{10}Q^e_3}{-r_{33}}\right\}>0,
\end{equation}
then \eqref{ineq:3-41} follows.
\end{proof}
With the help of this lemma, we are able to derive the first order estimate of $\partial_{113}w$. Firstly we notice that
$\partial_{113}w$ satisfies
\begin{align}
   L'(\partial_{113}w)=\partial_{113}f-[\partial_{113},L']w\quad \mbox{in}\quad \Omega_T&,\label{eq:satisfied by partial_113w}\\
   \partial_{113}w=0\quad\mbox{on}\quad \omega^r&,\\
   (\partial_{113}w,\partial_{y_0}(\partial_{113}w))=(0,0)\quad\mbox{on}\quad \Gamma_{in}&.
\end{align}
Next we need to deduce the boundary condition that $\partial_{113}w$ satisfies on the vertical boundary $\omega^\ell$.
Let
\begin{align}
    L_2&=2r_{12}\partial_{12}+r_{22}\partial_{22}+2r_{23}\partial_{32},\nonumber\\
    L_1&=r_{11}\partial_{11}+2r_{13}\partial_{13}+r_{33}\partial_{33},\nonumber\\
    L_0&=L'-L_1-L_2.\nonumber
\end{align}
So we have
\begin{align}
    &(r_{11}\partial_{y_1}\partial_{113}+2r_{13}\partial_{y_3}\partial_{113})w\nonumber\\
    &\qquad=L_1\partial_{13}w-r_{33}\partial_{1333}w\nonumber\\
    &\qquad=(L'-L_0-L_2)(\partial_{13}w)-r_{33}\partial_{1333}w.\label{eq:3-53}
\end{align}
For the terms on the right hand-side of above equality, only $r_{33}\partial_{1333}w$ has not been controlled yet. Indeed, $L'(\partial_{13}w)$ is what we need in the estimate and $L_0(\partial_{13}w)$ and $L_2(\partial_{13}w)$ have been controlled by \eqref{ineq:3-40}.
But by the boundary condition of $w$ on $\omega^\ell$, we notice that
\begin{align*}
    \partial_{1333}w=\frac{1}{b_1}(\mcb(\partial_{333}w)-b_0\partial_{y_0}\partial_{333}w-b_2\partial_{y_2}\partial_{333}w-b_3\partial_{y_3}\partial_{333}w).
\end{align*}
Therefore we deduce that
\begin{align}
   \abs{\partial_{1333}w}\lesssim \abs{\mcb(\partial_{333}w)}+\abs{\md^3\partial_{y_0}w}+\abs{\md^3\partial_{y_2}w}+\|b_3\|_{L^\infty}\cdot\abs{\partial_{3333}w}. \label{eq:3-54}
\end{align}
On the right hand-side of \eqref{eq:3-54}, the first term is what we need, the second and the third terms are controlled by \eqref{ineq:3-40}. For the last term, by \ref{h2} and \ref{h3}, we know that $\|b_3\|_{L^\infty}$ is small, provided the $\delta$ in \ref{h3} is appropriately small. Hence it can be absorbed by the left hand-side of the estimate coming later, which will cover all fourth order derivatives. On the boundary $\omega^\ell$, combining \eqref{eq:3-53} and \eqref{eq:3-54}, we obtain
\begin{align}
    \abs{\partial_{1113}w}\lesssim&\abs{L'(\partial_{13}w)}+\abs{\mcb(\partial_{333}w)}+\abs{\md^3\partial_{y_0}w}+\abs{\md^3\partial_{y_2}w}\nonumber\\
    &+\|b_3\|_{L^\infty}\cdot\abs{\partial_{3333}w}+\|r_{13}\|_{L^\infty}\cdot\abs{\partial_{1133}w}.\label{ineq:3-55}
\end{align}
Multiplying $2e^{-2\eta t}\mcq^e(\partial_{113}w)$, where $\mcq^e$ is given in lemma \ref{Multiplier for almost normal boundary conditions}, on both sides of \eqref{eq:satisfied by partial_113w}, integration by parts over $\Omega_T$ and by the use of Cauchy inequality, one has

\begin{align}
    &2\eta\int_{\Omega_T}e^{-2\eta t}H_0(\md\partial_{113}w;\mcq^e)\mathrm{d}\mby\mathrm{d}y_0+e^{-2\eta T}\int_{\Omega}e^{-2\eta t}H_0(\md\partial_{113}w; \mcq^e)\mathrm{d}\mby|_{t=T}\nonumber\\
    &\quad-\int_{\omega^\ell_T}e^{-2\eta t}H_1(\md\partial_{113}w;\mcq^e)\mathrm{d}\mby\mathrm{d}y_0-\int_{\omega^r_T}e^{-2\eta t}H_3(\md\partial_{113}w;\mcq^e)\mathrm{d}\mby\mathrm{d}y_0\nonumber\\
    &\qquad\lesssim \frac{1}{\varepsilon_1 \eta}\|e^{-\eta t}L'(\partial_{113}w)\|_{L^2(\Omega_T)}^2+(\varepsilon_1\eta+1)\|\md\partial_{113}w\|^2_{L^2(\Omega_T)}\nonumber\\
    &\quad\qquad+\int_{\Omega}H_0(\md\partial_{113}w;\mcq^e)\mathrm{d}\mby|_{t=0}.
\end{align}
By \eqref{ineq:3-41},\eqref{ineq:3-42} and \eqref{ineq:3-43} together with the fact
 \[H_0(\md\partial_{113}w;\mcq^e)\lesssim \abs{\md\partial_{113}w }^2,
 \]
we deduce that
\begin{align}
    &\eta \|e^{-\eta t}\nabla_{\mby}\partial_{113}w\|^2_{L^2(\Omega_T)}+e^{-2\eta T}\|\nabla_{\mby}\partial_{113}w\|^2_{L^2(\Omega)}+\|e^{-\eta t}
    \partial_{y_3}\partial_{113}w\|^2_{L^2(\omega^\ell_T)}\nonumber\\
    &\quad+\|e^{-\eta t}\partial_{y_3}\partial_{113}w\|^2_{L^2(\omega^r_T)}\nonumber\\
    &\quad\qquad\lesssim \frac{1}{\varepsilon_1 \eta}\|e^{-\eta t}L'(\partial_{113}w)\|_{L^2(\Omega_T)}^2+(\varepsilon_1\eta+1)\|\md\partial_{113}w\|^2_{L^2(\Omega_T)}\nonumber\\
    &\qquad\qquad+\eta\|e^{-\eta t}\partial_{y_0}\partial_{113}w\|^2_{L^2(\Omega_T)}+e^{-2\eta T}\|\partial_{y_0}\partial_{113}w\|^2_{L^2(\Omega)}\nonumber\\
    &\qquad\qquad+\sum_{i=0}^2\|e^{-\eta t}\partial_{y_i}\partial_{113}w\|^2_{L^2(\omega^\ell_T)}+\|\md^2 f|_{t=0}\|^2_{L^2(\Omega)}.
\end{align}
Recalling \eqref{ineq:3-40} and \eqref{ineq:3-55}, we obtain
\begin{align}
      &\eta \|e^{-\eta t}\md\partial_{113}w\|^2_{L^2(\Omega_T)}+e^{-2\eta T}\|\md\partial_{113}w\|^2_{L^2(\Omega)}+\|e^{-\eta t}
    \md\partial_{113}w\|^2_{L^2(\omega^\ell_T)}\nonumber\\
    &\quad+\|e^{-\eta t}\partial_{y_3}\partial_{113}w\|^2_{L^2(\omega^r_T)}\nonumber\\
    &\qquad\lesssim \frac{1}{\varepsilon_1 \eta}\|e^{-\eta t}L'(\partial_{113}w)\|_{L^2(\Omega_T)}^2+(\varepsilon_1\eta+1)\|\md\partial_{113}w\|^2_{L^2(\Omega_T)}\nonumber\\
    &\qquad\quad+\sum_{|\alpha|\le 3}\left( \frac{1}{\eta}\|e^{-\eta t}L'(\md^\alpha w)\|^2_{L^2(\Omega_T)}+\|e^{-\eta t}\mathcal{B}\md^\alpha w\|^2_{L^2(\omega_T^\ell)}\right)\nonumber\\
    &\qquad\quad+\|b_3\|^2_{L^\infty}\cdot\|e^{-\eta t}\partial_{3333}w\|^2_{L^2(\omega^\ell_T)}
    +\|r_{13}\|^2_{L^\infty}\cdot\|e^{-\eta t}\partial_{1133}w\|^2_{L^2(\omega^{\ell}_T)}\nonumber\\
   &\qquad\quad +\|e^{-\eta t}L'(\partial_{13}w)\|^2_{L^2(\omega^\ell_T)}+\|\md^2 f|_{t=0}\|^2_{L^2(\Omega)}.\label{ineq:3-58}
\end{align}
With \eqref{ineq:3-40} and \eqref{ineq:3-58} in hand, we can deduce the estimate of the left derivatives of fourth order, i.e., $\partial_{1111}w$, $\partial_{3333}w$ and $\partial_{1333}w$. It is clear that
\begin{align}
   \partial_{1111}w=\frac{1}{r_{11}}\left((L'-L_0-L_2)\partial_{11}w-2r_{13}\partial_{1113}w-r_{33}\partial_{1133}w\right).
\end{align}
Hence one has
\begin{align}
    &\eta \|e^{-\eta t}\partial_{1111}w\|^2_{L^2(\Omega_T)}+e^{-2\eta T}\|\partial_{1111}w\|^2_{L^2(\Omega)}+\|e^{-\eta t}
    \partial_{1111}w\|^2_{L^2(\omega^\ell_T)}\nonumber\\
    &\quad\lesssim \sum_{|\alpha|\le 3}\left(\eta \|e^{-\eta t}\md^\alpha\partial_{y_0}w\|^2_{L^2(\Omega_T)}+e^{-2\eta T}\|\md^\alpha \partial_{y_0}w\|^2_{L^2(\Omega)}+\|e^{-\eta t}
    \md^\alpha \partial_{y_0}w\|^2_{L^2(\omega^\ell_T)}\right)\nonumber\\
    &\qquad+\sum_{|\alpha|\le 3}\left(\eta \|e^{-\eta t}\md^\alpha\partial_{y_2}w\|^2_{L^2(\Omega_T)}+e^{-2\eta T}\|\md^\alpha \partial_{y_2}w\|^2_{L^2(\Omega)}+\|e^{-\eta t}
    \md^\alpha \partial_{y_2}w\|^2_{L^2(\omega^\ell_T)}\right)\nonumber\\
    &\qquad+\eta \|e^{-\eta t}\md\partial_{113}w\|^2_{L^2(\Omega_T)}+e^{-2\eta T}\|\md\partial_{113}w\|^2_{L^2(\Omega)}+\|e^{-\eta t}
    \md\partial_{113}w\|^2_{L^2(\omega^\ell_T)}\nonumber\\
    &\qquad+\eta \|e^{-\eta t}L'(\partial_{11}w)\|^2_{L^2(\Omega_T)}+e^{-2\eta T}\|L'(\partial_{11}w)\|^2_{L^2(\Omega)}+\|e^{-\eta t}
    L'(\partial_{11}w)\|^2_{L^2(\omega^\ell_T)}.\label{ineq:3-60}
\end{align}
For $\partial_{3333}w$ and $\partial_{1333}w$, it is easy to check that
\begin{align}
    &\partial_{1333}w=\frac{1}{r_{33}}\left((L'-L_0-L_2)\partial_{13}w-r_{11}\partial_{1113}w-2r_{13}\partial_{1133}w\right),\nonumber\\
   & \partial_{3333}w=\frac{1}{r_{33}}\left((L'-L_0-L_2)\partial_{33}w-2r_{13}\partial_{1333}w-r_{11}\partial_{1133}w\right).\nonumber
\end{align}
Thus both $\partial_{3333}w$ and $\partial_{1333}w$ can be controlled by estimated terms. In fact, we have
\begin{align}
   &\sum_{i=1,3}\left( \eta \|e^{-\eta t}\partial_{i333}w\|^2_{L^2(\Omega_T)}+e^{-2\eta T}\|\partial_{i333}w\|^2_{L^2(\Omega)}+\|e^{-\eta t}
    \partial_{i333}w\|^2_{L^2(\omega^\ell_T)}\right)\nonumber\\
    &\quad\lesssim \frac{1}{\varepsilon_1 \eta}\|e^{-\eta t}L'(\partial_{113}w)\|_{L^2(\Omega_T)}^2+(\varepsilon_1\eta+1)\|\md\partial_{113}w\|^2_{L^2(\Omega_T)}+\|\md^2 f|_{t=0}\|^2_{L^2(\Omega)}\nonumber\\
    &\qquad+\sum_{|\alpha|\le 3}\left( \frac{1}{\eta}\|e^{-\eta t}L'(\md^\alpha w)\|^2_{L^2(\Omega_T)}+\|e^{-\eta t}\mathcal{B}\md^\alpha w\|^2_{L^2(\omega_T^\ell)}\right)+\|e^{-\eta t}L'(\partial_{13}w)\|^2_{L^2(\omega^\ell_T)}\nonumber\\
    &\qquad+\|b_3\|^2_{L^\infty}\cdot\|e^{-\eta t}\partial_{3333}w\|^2_{L^2(\omega^\ell_T)}
    +\|r_{13}\|^2_{L^\infty}\cdot\|e^{-\eta t}\partial_{1133}w\|^2_{L^2(\omega^{\ell})}\nonumber\\
    &\qquad+\eta \|e^{-\eta t}\md\partial_{113}w\|^2_{L^2(\Omega_T)}+e^{-2\eta T}\|\md\partial_{113}w\|^2_{L^2(\Omega)}+\|e^{-\eta t}
    \md\partial_{113}w\|^2_{L^2(\omega^\ell_T)}\nonumber\\
    &\qquad+\sum_{i=1,3}\left(\eta \|e^{-\eta t}L'(\partial_{1i}w)\|^2_{L^2(\Omega_T)}+e^{-2\eta T}\|L'(\partial_{1i}w)\|^2_{L^2(\Omega)}+\|e^{-\eta t}
    L'(\partial_{1i}w)\|^2_{L^2(\omega^\ell_T)}\right)\nonumber\\
    &\qquad+\eta \|e^{-\eta t}L'(\partial_{33}w)\|^2_{L^2(\Omega_T)}+e^{-2\eta T}\|L'(\partial_{33}w)\|^2_{L^2(\Omega)}+\|e^{-\eta t}
    L'(\partial_{33}w)\|^2_{L^2(\omega^\ell_T)}.\label{ineq:3-61}
\end{align}
It is not difficult to see that $\md^3\partial_{y_0}w$, $\md^3\partial_{y_2}w$, $\md\partial_{113}w$, $\partial_{1111}w$, $\partial_{1333}w$ and $\partial_{3333}w$ cover all derivatives of fourth order of $w$. We add \eqref{ineq:3-40}, \eqref{ineq:3-58}, \eqref{ineq:3-60} and \eqref{ineq:3-61} up, let the $\varepsilon_1$ in \eqref{ineq:3-58} and $\delta$ be properly small and let $\eta$ be properly large, such that the terms with smallness be absorbed by the corresponding left hand-side terms. Then we obtain
\begin{align}
   &\sum_{|\alpha|\le 4} \eta \|e^{-\eta t}\md^\alpha w\|^2_{L^2(\Omega_T)}+e^{-2\eta T}\|\md^\alpha w\|^2_{L^2(\Omega)}+\|e^{-\eta t}
    \md^\alpha w\|^2_{L^2(\omega^\ell_T)}\nonumber\\
    &\quad\lesssim\sum_{|\alpha|\le 3}\left( \frac{1}{\eta}\|e^{-\eta t}L'(\md^\alpha w)\|^2_{L^2(\Omega_T)}+\|e^{-\eta t}\mathcal{B}\md^\alpha w\|^2_{L^2(\omega_T^\ell)}\right)+\|\md^2 f|_{t=0}\|^2_{L^2(\Omega)}\nonumber\\
    &\qquad+\eta \|e^{-\eta t}L'(\partial_{11}w)\|^2_{L^2(\Omega_T)}+e^{-2\eta T}\|L'(\partial_{11}w)\|^2_{L^2(\Omega)}+\|e^{-\eta t}
    L'(\partial_{11}w)\|^2_{L^2(\omega^\ell_T)}\nonumber\\
    &\qquad+\eta \|e^{-\eta t}L'(\partial_{13}w)\|^2_{L^2(\Omega_T)}+e^{-2\eta T}\|L'(\partial_{13}w)\|^2_{L^2(\Omega)}+\|e^{-\eta t}
    L'(\partial_{13}w)\|^2_{L^2(\omega^\ell_T)}\nonumber\\
    &\qquad+\eta \|e^{-\eta t}L'(\partial_{33}w)\|^2_{L^2(\Omega_T)}+e^{-2\eta T}\|L'(\partial_{33}w)\|^2_{L^2(\Omega)}+\|e^{-\eta t}
    L'(\partial_{33}w)\|^2_{L^2(\omega^\ell_T)}.\label{ineq:3-62}
\end{align}
Exploiting integration by parts to $\int_{\Omega_T}e^{-2\eta t}\xi^2\mathrm{d}\mby\mathrm{d}y_0$ with respect to $t$, we can derive following inequality
\begin{align}
    \eta\int_{\Omega_T}{\!\!e^{-2\eta t}\abs{\xi}^2}\mathrm{d}\mby\mathrm{d}y_0+{e^{-2\eta T}\int_{\Omega}\abs{\xi}^2}\mathrm{d}\mby
   \le \frac{1}{\eta}\int_{\Omega_T}{\!\!e^{-2\eta t}\abs{\partial_{t}\xi}^2}\mathrm{d}\mby\mathrm{d}y_0+\int_{\Omega}{\abs{\xi(0)}^2}\mathrm{d}\mby.
\end{align}
Hence we obtain
\begin{align}
    &\eta\|e^{-\eta t}L'(\partial_{ij}w)\|^2_{L^2(\Omega_T)}+e^{-2\eta T}\|L'(\partial_{ij}w)\|^2_{L^2(\Omega)}\nonumber\\
    &\le \frac{1}{\eta}\|e^{-\eta t}\partial_tL'(\partial_{ij}w)\|^2_{L^2(\Omega_T)}+\|L'(\partial_{ij}w)|_{t=0}\|^2_{L^2(\Omega)}\nonumber\\
    &\lesssim \frac{1}{\eta}\left(\sum_{|\alpha|\le 4}\|e^{-\eta t}\md^\alpha w\|^2_{L^2(\Omega_T)}+\|e^{-\eta t}L'(\partial_t\partial_{ij}w)\|^2_{L^2(\Omega_T)}\right)+\|\md^2 f|_{t=0}\|^2_{L^2(\Omega)}.\label{ineq:3-64}
\end{align}
By Gauss theorem, we also have
\begin{align}
    &\|e^{-\eta t}L'(\partial_{ij}w)\|^2_{L^2(\omega^\ell_T)}\nonumber\\
    &\quad=\int_{\Omega_T}-\partial_{y_1}(e^{-2\eta t}\abs{L'(\partial_{ij}w)}^2)\mathrm{d}\mby\mathrm{d}y_0\nonumber\\
    &\quad\le \int_{\Omega_T}e^{-2\eta t}2|L'(\partial_{ij}w)|\cdot|\partial_{y_1}L'(\partial_{ij}w)|\mathrm{d}\mby\mathrm{d}y_0\nonumber\\
    &\quad\le \int_{\Omega_T}e^{-2\eta t}\left(\frac{1}{\varepsilon\eta}\abs{\partial_{y_1}L'(\partial_{ij}w)}^2+\varepsilon\eta\abs{L'(\partial_{ij}w)}^2\right)\mathrm{d}\mby\mathrm{d}y_0\nonumber\\
    &\quad\le\int_{\Omega_T}e^{-2\eta t}\left(\frac{1}{\varepsilon\eta}\left(2\abs{[\partial_{y_1},L']\partial_{ij}w}^2+2\abs{L'(\partial_{y_1}\partial_{ij}w)}^2\right)+\varepsilon\eta\abs{L'(\partial_{ij}w)}^2\right)\mathrm{d}\mby\mathrm{d}y_0\nonumber\\
    &\quad\lesssim \frac{1}{\varepsilon\eta}\sum_{|\alpha|\le 3}\left(\|e^{-\eta t}\md\md^\alpha w\|^2_{L^2(\Omega_T)}+\|e^{-\eta t}L'(\md^\alpha w)\|^2_{L^2(\Omega_T)}\right)\nonumber\\
    &\qquad+\varepsilon\eta\sum_{|\alpha|\le 4}\|e^{-\eta t}\md^\alpha w\|^2_{L^2(\Omega_T)}.\label{ineq:3-65}
\end{align}
Substitute \eqref{ineq:3-64} and \eqref{ineq:3-65} into \eqref{ineq:3-62}, let the $\varepsilon$ in \eqref{ineq:3-65} be properly small and then let $\eta$ be appropriately large, we conclude the fourth order estimate as follows
\begin{align}
   &\sum_{|\alpha|\le 4} \eta \|e^{-\eta t}\md^\alpha w\|^2_{L^2(\Omega_T)}+e^{-2\eta T}\|\md^\alpha w\|^2_{L^2(\Omega)}+\|e^{-\eta t}
    \md^\alpha w\|^2_{L^2(\omega^\ell_T)}\nonumber\\
    &\quad\lesssim\sum_{|\alpha|\le 3}\left( \frac{1}{\eta}\|e^{-\eta t}L'(\md^\alpha w)\|^2_{L^2(\Omega_T)}+\|e^{-\eta t}\mathcal{B}\md^\alpha w\|^2_{L^2(\omega_T^\ell)}\right)+\|\md^2 f|_{t=0}\|^2_{L^2(\Omega)}.\label{fourth order estimate}
\end{align}
\vskip 0.2cm
\textbf{Step 5:} Higher order estimate. In this step we will prove higher order estimate by the induction method. Assume the estimate of $2k$-th order has been established, i.e., we have
 \begin{align}
   &\sum_{|\alpha|\le 2k} \eta \|e^{-\eta t}\md^\alpha w\|^2_{L^2(\Omega_T)}+e^{-2\eta T}\|\md^\alpha w\|^2_{L^2(\Omega)}+\|e^{-\eta t}
    \md^\alpha w\|^2_{L^2(\omega^\ell_T)}\nonumber\\
    &\quad\lesssim\!\!\sum_{|\alpha|\le 2k-1}\!\!\left( \frac{1}{\eta}\|e^{-\eta t}L'(\md^\alpha w)\|^2_{L^2(\Omega_T)}+\|e^{-\eta t}\mathcal{B}\md^\alpha w\|^2_{L^2(\omega_T^\ell)}\right)+\|\md^{2k-2}f|_{t=0}\|^2_{L^2(\Omega)}.\label{ineq:induction assumption}
\end{align}
Then one proceeds to establish the estimate of $(2k+1)$-th order and $(2k+2)$-th order on the basis of the estimate of $(2k)$-th order. In what follows, we deal with the estimate of $(2k)$-th order first.
Since both $\partial_{y_0}$ and $\partial_{y_2}$ are tangential to the boundaries $\Gamma_s$ and $\Gamma_w$, the application of \eqref{ineq:induction assumption} to $\partial_{y_0}w$ and $\partial_{y_2}w$ yields
\begin{align}
&\sum_{i=0,2}\sum_{|\alpha|\le 2k}\eta\|e^{-\eta t}\md^\alpha\partial_{y_i} w\|^2_{L^2(\Omega_T)}+e^{-2\eta T}\|\md^\alpha \partial_{y_i}w\|^2_{L^2(\Omega)}+\|e^{-\eta t}\md^\alpha \partial_{y_i}w\|^2_{L^2(\omega^\ell_T)}\nonumber\\
&\quad\lesssim\sum_{|\alpha|\le 2k}\left( \frac{1}{\eta}\|e^{-\eta t}L'(\md^\alpha w)\|^2_{L^2(\Omega_T)}+\|e^{-\eta t}\mathcal{B}\md^\alpha w\|^2_{L^2(\omega_T^\ell)}\right)\nonumber\\
&\quad\quad+\|\md^{2k-2}f|_{t=0}\|^2_{L^2(\Omega_T)}.\label{ineq:3-68}
\end{align}
Analogous to the estimate of third order, one tries to derive the first order estimate of $\partial^{2k-1}_{y_1}\partial_{y_3}w$.
It is clear that
\begin{align}
   L'(\partial^{2k-1}_{y_1}\partial_{y_3}w)=\partial^{2k-1}_{y_1}\partial_{y_3}f-[\partial^{2k-1}_{y_1}\partial_{y_3},L']w\quad \mbox{in}\quad \Omega_T,&\label{eq:3-69}\\
   \partial^{2k-1}_{y_1}\partial_{y_3}w=0\quad\mbox{on}\quad \omega^r,&\\
   (\partial^{2k-1}_{y_1}\partial_{y_3}w,\partial_{y_0}(\partial^{2k-1}_{y_1}\partial_{y_3}w))=(0,0)\quad\mbox{on}\quad \Gamma_{in}.&
\end{align}
Next we need to deduce the boundary condition on $\omega^\ell$ for $\partial^{2k-1}_{y_1}\partial_{y_3}w$.
% In order to express this process clearly, denote $\partial^j_{y_1}\partial^{2k-j}_{y_3}w$ by $\beta_j$.
 It is not difficult to check that
\begin{align}
    \partial^{2k-1}_{y_1}\partial_{y_3}w=(L'-L_0-L_2-2r_{13}\partial_{13}-r_{33}\partial_{33})\partial_{y_1}^{2k-3}\partial_{y_3}w.\label{eq:3-73}
\end{align}
%Now we deduce a recursive formula for the sequence $\{\beta_j\}_{j=1}^{2k-1}$.
%In fact, we have the following relation
%\begin{align}
%    \partial^{j+1}_{y_1}\partial^{2k-j-1}_{y_3}w=(L'-L_0-L_2-2r_{13}\partial_{13}-r_{33}\partial_{33})\partial^{j-1}_{y_1}\partial^{2k-j-1}_{y_3}w, \label{eq:3-73}
%\end{align}
%for $j=1,2,\cdots,2k-2$.
%Here we point out that the term $2r_{13}\partial^j_{y_1}\partial^{2k-j}_{y_3}w$ in $\eqref{eq:3-73}$ should be regarded as known function, due to the smallness of $r_{13}$ by \ref{h1} and \ref{h3}. Moreover, $L_0(\partial^{j-1}_{\partial_{y_1}}\partial^{2k-j-1}_{y_3})$ and $L_2(\partial^{j-1}_{\partial_{y_1}}\partial^{2k-j-1}_{y_3})$ in \eqref{eq:3-72} are controlled by \eqref{ineq:3-68} and hence
%Let $j$ be equal $2$ in \eqref{eq:3-72} recursive formula, we can derive $\partial_{y_1}^3\partial_{y_3}^{2k-3}w$ from $\partial_{y_1}\partial^{2k-1}_{y_3}w$ which can be regarded as known function by \eqref{eq:3-72}. Then let $j=4$, we obtain $\partial^5_{y_1}\partial^{2k-5}_{y_3}w$. Continue this process by letting $j=6,8,\cdots,2k-2$, we will finally arrive at $\partial^{2k-1}_{y_1}\partial_{y_3}w$.
It follows from \eqref{eq:3-73} that
\begin{align}
    \partial^{2k-1}_{y_1}\partial_{y_3}w=\Lambda_1\quad\mbox{on }\quad\Gamma_s.
\end{align}
Moreover, from \eqref{eq:3-73}, one has
\begin{align}
    \abs{\partial_{y_0}\Lambda_1}+&\abs{\partial_{y_2}\Lambda_1}+\abs{\partial_{y_3}\Lambda_1}\nonumber\\
    &\lesssim \sum_{|\alpha|\leq2k}(|\md^{\alpha}\partial_{y_0}w|+|\md^{\alpha}w|+|\md^\alpha \partial_{y_2}w|)+\delta|\md^{2k+1}w|\nonumber\\
    &\quad+\sum_{|\alpha|\le 2k-1}|L'(\md^\alpha w)|+\abs{\partial_{y_1}^{2k-3}\partial^4_{y_3}w}\label{ineq:3-74},
\end{align}
where the $\delta$ before $|\md^{2k+1}w|$ comes from the smallness of $r_{13}$ due to \ref{h1} and \ref{h3}.
So we have to estimate $\abs{\partial_{y_1}^{2k-3}\partial^4_{y_3}w}$.
We already know $\mathcal{B}w=g$ on $\omega^\ell$, then it is easy to verify that
\begin{align}
    \partial_{y_1}\partial^{2k}_{y_3}w=\frac{1}{b_1}(\mathcal{B}\partial^{2k}_{y_3}w-(b_0\partial_{y_0}+b_2\partial_{y_2}+b_3\partial_{y_3}+b)\partial^{2k}_{y_3}w).\label{eq:3-75}
\end{align}
Remembering that $\partial_{y_0}\partial^{2k}_{y_3}w$,  $\partial_{y_2}\partial^{2k}_{y_3}w$ and $\partial_{y_3}^{2k}w$ have been controlled by \eqref{ineq:3-68} and $\|b_3\|_{L^\infty}$ is close to zero, so $\partial_{y_1}\partial_{y_3}^{2k}w$ can be regarded as known function on $\omega^\ell$. Furthermore, from \eqref{eq:3-75}, we have
\begin{align}
    \abs{\partial_{y_1}\partial_{y_3}^{2k}w}\lesssim \sum_{|\alpha|\leq2k}(|\md^{\alpha}\partial_{y_0}w|+|\md^{\alpha}w|+|\md^\alpha\partial_{y_2}w|+|\mathcal{B}(\md^\alpha w)|)+\delta|\md^{2k+1}w|\label{ineq:3-76}.
\end{align}
It is easy to check that
\begin{align}
    \partial_{y_1}^{2k-2j-1}\partial_{y_3}^{2j+2}w=\frac{1}{r_{33}}\left(L'-L_0-L_2-2r_{13}\partial_{13}-r_{11}\partial_{11}\right)\partial_{y_1}^{2k-2j-1}\partial_{y_3}^{2j}w.\label{eq:3-77}
\end{align}
For the ease of presentation, let
\begin{align}
   \beta_j\defs&\partial_{y_1}^{2k-2j+1}\partial_{y_3}^{2j}w,\label{eq:3-78}\\
   A_j\defs&\left(L'-L_0-L_2-2r_{13}\partial_{13}\right)\partial_{y_1}^{2k-2j-1}\partial_{y_3}^{2j}w \label{eq:3-79}
\end{align}
 for $j=2,3,\cdots,k$.
Then it is clear that
\begin{align}
    \abs{A_j}\lesssim \sum_{|\alpha|\leq2k}(|\md^{\alpha}\partial_{y_0}w|+|\md^{\alpha}w|+\md^\alpha\partial_{y_2}w|)+\delta|\md^{2k+1}w|+\sum_{|\alpha|\le 2k-1}|L'(\md^\alpha w)|.\label{ineq:3-80}
\end{align}
From \eqref{eq:3-77}-\eqref{eq:3-79}, we obtain
\begin{align}
    \beta_{j+1}=\frac{1}{r_{33}}(A_j-r_{11}\beta_j),
\end{align}
which implies
\begin{align}
    \beta_j=\frac{A_j-r_{33}\beta_{j+1}}{r_{11}}.\label{eq:3-82}
\end{align}
Gathering \eqref{ineq:3-76}, \eqref{eq:3-79}, and \eqref{eq:3-82}, one derives a sequence $\{\beta_j\}_{j=2}^k$ that satisfies
\begin{equation}\label{ineq:3-83}
    \left\{\!\!
    \begin{array}{ll}
    \beta_k\lesssim \sum_{|\alpha|\leq2k}(|\md^{\alpha}\partial_{y_0}w|+|\md^{\alpha}w|+|\mathcal{B}(\md^\alpha w)|+|\md^\alpha\partial_{y_2}w|)+\delta|\md^{2k+1}w|,\\
    \beta_j=\dfrac{A_j-r_{33}\beta_{j+1}}{r_{11}},\quad j=k-1,k-2,\cdots,3,2.\\
    \abs{A_j}\lesssim \sum_{|\alpha|\leq2k}(|\md^{\alpha}\partial_{y_0}w|+|\md^{\alpha}w|+\delta|\md^{2k+1}w|)+\sum_{|\alpha|\le 2k-1}|L'(\md^\alpha w)|.
\end{array}
    \right.
\end{equation}
For $j=2,3,\cdots,k$, we claim that $\beta_j$ satisfies
\begin{align}
    \abs{\beta_j}&\lesssim \sum_{|\alpha|\leq2k}(|\md^{\alpha}\partial_{y_0}w|+|\md^{\alpha}w|+|\mathcal{B}(\md^\alpha w)|+|\md^\alpha\partial_{y_2}w|)\nonumber\\
    &\quad+\delta|\md^{2k+1}w|+\sum_{|\alpha|\le 2k-1}\abs{L'(\md^\alpha w)},\label{ineq:3-84}
\end{align}
and hence so does $\partial_{y_1}^{2k-3}\partial^{4}_{y_3}w\defs\beta_2$.
Indeed, from \eqref{ineq:3-76} it is clear to see that $\beta_k$ satisfies \eqref{ineq:3-84}. Assume
$\beta_\ell$ satisfies \eqref{ineq:3-84} for some $\ell \le k$, then by \eqref{ineq:3-80} and $\eqref{ineq:3-83}_2$, we obtain
\begin{align}
    \abs{\beta_{\ell-1}}&\lesssim \abs{A_{\ell-1}}+\abs{\beta_{\ell}}\nonumber\\
    &\lesssim \sum_{|\alpha|\leq2k}(|\md^{\alpha}\partial_{y_0}w|+|\md^{\alpha}w|+|\mathcal{B}(\md^\alpha w)|+|\md^\alpha\partial_{y_2}w|)\nonumber\\
    &\quad+\delta|\md^{2k+1}w|+\!\!\sum_{|\alpha|\le 2k-1}\abs{L'(\md^\alpha w)},\nonumber
\end{align}
which implies $\beta_{\ell-1}$ also satisfies \eqref{ineq:3-84}. Hence our claim holds.
Therefore one can deduce from \eqref{ineq:3-74} that
\begin{align}
&\abs{(\partial_{y_0}\Lambda_1}+\abs{\partial_{y_2}\Lambda_1}+\abs{\partial_{y_3}\Lambda_1)}\nonumber\\
&\quad\lesssim \sum_{|\alpha|\leq2k}(|\md^{\alpha}\partial_{y_0}w|+|\md^{\alpha}w|+|\mathcal{B}(\md^\alpha w)|+\delta|\md^{2k+1}w|\nonumber\\
&\qquad+\sum_{|\alpha|\le 2k-1}\abs{L'(\md^\alpha w)}.\label{ineq:3-85}
\end{align}
With the help of lemma \ref{Multiplier for two Dirichlet boundary conditions} and \eqref{ineq:3-85}, we are able to obtain the first order estimate of $\partial_{y_1}^{2k-1}\partial_{y_3}w$.
Multiplying $2e^{-2\eta t}\mcq^d(\partial^{2k-1}_{y_1}\partial_{y_3}w)$ on both sides of \eqref{eq:3-69}, integrating by parts over $\Omega_T$ and then apply \eqref{ineq:3-19}-\eqref{ineq:3-21} in lemma \ref{Multiplier for two Dirichlet boundary conditions}, one deduces that
\begin{align}
   &\eta \|e^{-\eta t}\nabla_{\mby}\partial_{y_1}^{2k-1}\partial_{y_3}w\|^2_{L^2(\Omega_T)}+e^{-2\eta T}\|\nabla_{\mby}\partial_{y_1}^{2k-1}\partial_{y_3}w\|^2_{L^2(\Omega)}+\|e^{-\eta t}
    \partial_{y_1}^{2k}\partial_{y_3}w\|^2_{L^2(\omega^\ell_T)}\nonumber\\
    &\qquad\quad+\|e^{-\eta t}\partial_{y_1}^{2k-1}\partial^2_{y_3}w\|^2_{L^2(\omega^r_T)}\nonumber\\
    &\qquad\lesssim \frac{1}{\varepsilon\eta}\|e^{-\eta t}L'(\partial_{y_1}^{2k-1}\partial_{y_3}w)\|_{L^2(\Omega_T)}^2+(1+\varepsilon\eta)\|e^{-\eta t}\md\partial_{y_1}^{2k-1}\partial_{y_3}w\|^2_{L^2(\Omega_T)}\nonumber\\
    &\qquad\quad+\|e^{-\eta t}\partial_{y_0}\partial_{y_1}^{2k-1}\partial_{y_3}w\|^2_{L^2(\Omega_T)}+\sum_{i\neq 1}\|e^{-\eta t}\partial_{y_i}\partial_{y_1}^{2k-1}\partial_{y_3}w\|^2_{L^2(\omega^\ell_T)}.\label{ineq:3-86}
\end{align}
In light of \eqref{ineq:3-68}, \eqref{ineq:3-85} and \eqref{ineq:3-86}, we obtain
\begin{align}
      &\eta \|e^{-\eta t}\md\partial_{y_1}^{2k-1}\partial_{y_3}w\|^2_{L^2(\Omega_T)}+e^{-2\eta T}\|\md\partial_{y_1}^{2k-1}\partial_{y_3}w\|^2_{L^2(\Omega)}+\|e^{-\eta t}
    \md\partial_{y_1}^{2k-1}\partial_{y_3}w\|^2_{L^2(\omega^\ell_T)}\nonumber\\
    &\quad\lesssim \frac{1}{\varepsilon\eta}\|e^{-\eta t}L'(\partial_{y_1}^{2k-1}\partial_{y_3}w)\|_{L^2(\Omega_T)}^2+(1+\varepsilon\eta)\|e^{-\eta t}\md\partial_{y_1}^{2k-1}\partial_{y_3}w\|^2_{L^2(\Omega_T)}\nonumber\\
    &\qquad+\sum_{|\alpha|\le 2k}\left( \frac{1}{\eta}\|e^{-\eta t}L'(\md^\alpha w)\|^2_{L^2(\Omega_T)}+\|e^{-\eta t}\mathcal{B}\md^\alpha w\|^2_{L^2(\omega_T^\ell)}\right)+\|\md^{2k-2}f|_{t=0}\|^2_{L^2(\Omega_T)}\nonumber\\
    &\qquad+\delta\|e^{-\eta t}\md^{2k+1}w\|^2_{L^2(\omega^\ell_T)}+\sum_{|\alpha|\le 2k-1}\|e^{-\eta t}L'(\md^\alpha w))\|^2_{L^2(\omega^\ell_T)}.\label{ineq:3-87}
\end{align}
Now one turns to the estimate of derivatives other than $\md^{2k}\partial_{y_0}w$, $\md^{2k}\partial_{y_2}w$ and $\md \partial^{2k-1}_{y_1}\partial_{y_3}w$, i.e., the estimate of $\partial^{2k+1}_{y_1}w$ and the estimate of derivatives in the form of $\partial^{2k-j+1}_{y_1}\partial_{y_3}^jw$ with $3\le j\le 2k+1$.
For $\partial^{2k+1}_{y_1}w$, it is easy to check that
\begin{align}
    \partial^{2k+1}_{y_1}w=\frac{1}{r_{11}}(L'-L_0-L_2-2r_{13}\partial_{13}-r_{33}\partial_{33})\partial^{2k-1}_{y_1}w.
\end{align}
Hence $\partial_{y_1}^{2k+1}w$ can be controlled by estimated terms. In fact, we have
\begin{align}
   & \eta\|e^{-\eta t}\partial^{2k+1}_{y_1}w\|^2_{L^2(\Omega_T)}+e^{-2\eta T}\|\partial^{2k+1}_{y_1}w\|^2_{L^2(\Omega)}+\|e^{-\eta t}\partial^{2k+1}_{y_1}w\|^2_{L^2(\omega^\ell_T)}\nonumber\\
   &\qquad\lesssim \eta\|e^{-\eta t}\md^{2k}\partial_{y_0}w\|^2_{L^2(\Omega_T)}+e^{-2\eta T}\|\md^{2k}\partial_{y_0}w\|^2_{L^2(\Omega)}+\|e^{-\eta t}\md^{2k}\partial_{y_0}w\|^2_{L^2(\omega^\ell_T)}\nonumber\\
   &\qquad\quad+\eta\|e^{-\eta t}\md^{2k}\partial_{y_2}w\|^2_{L^2(\Omega_T)}+e^{-2\eta T}\|\md^{2k}\partial_{y_2}w\|^2_{L^2(\Omega)}+\|e^{-\eta t}\md^{2k}\partial_{y_2}w\|^2_{L^2(\omega^\ell_T)}\nonumber\\
   &\qquad\quad+\eta\|e^{-\eta t}\md \partial^{2k-1}_{y_1}\partial_{y_3}w\|^2_{L^2(\Omega_T)}+e^{-2\eta T}\|\md \partial^{2k-1}_{y_1}\partial_{y_3}w\|^2_{L^2(\Omega)}+\|e^{-\eta t}\md \partial^{2k-1}_{y_1}\partial_{y_3}w\|^2_{L^2(\omega^\ell_T)}\nonumber\\
   &\qquad\quad+\eta\|e^{-\eta t} L'(\partial_{y_1}^{2k-1}w)\|^2_{L^2(\Omega_T)}+e^{-2\eta T}\|L'(\partial_{y_1}^{2k-1}w)\|^2_{L^2(\Omega)}+\|e^{-\eta t}L'(\partial_{y_1}^{2k-1}w)\|^2_{L^2(\omega^\ell_T)}\nonumber\\
    &\qquad\lesssim \frac{1}{\varepsilon\eta}\|e^{-\eta t}L'(\partial_{y_1}^{2k-1}\partial_{y_3}w)\|_{L^2(\Omega_T)}^2+(1+\varepsilon\eta)\|e^{-\eta t}\md\partial_{y_1}^{2k-1}\partial_{y_3}w\|^2_{L^2(\Omega_T)}\nonumber\\
    &\quad\qquad+\sum_{|\alpha|\le 2k}\left( \frac{1}{\eta}\|e^{-\eta t}L'(\md^\alpha w)\|^2_{L^2(\Omega_T)}+\|e^{-\eta t}\mathcal{B}\md^\alpha w\|^2_{L^2(\omega_T^\ell)}\right)+\|\md^{2k-2}f|_{t=0}\|^2_{L^2(\Omega_T)}\nonumber\\
    &\quad\qquad+\delta\|e^{-\eta t}\md^{2k+1}w\|^2_{L^2(\omega^\ell_T)}+\eta\|e^{-\eta t} L'(\partial_{y_1}^{2k-1}w)\|^2_{L^2(\Omega_T)}+e^{-2\eta T}\|L'(\partial_{y_1}^{2k-1}w)\|^2_{L^2(\Omega)}\nonumber\\
    &\quad\qquad+\sum_{|\alpha|\le 2k-1}\|e^{-\eta t}L'(\md^\alpha w)\|^2_{L^2(\omega^\ell_T)}.\label{ineq:3-89}
\end{align}
We remark that the last three terms in \eqref{ineq:3-89} can be estimated by same argument as \eqref{ineq:3-64} and \eqref{ineq:3-65}.
For all $j=0,1,2,\cdots,2k+1$ we claim that
 \begin{align}
    &\eta\|e^{-\eta t}\partial_{y_1}^{2k-j+1}\partial^j_{y_3}w\|^2_{L^2(\Omega_T)}+e^{-2\eta T}\|\partial_{y_1}^{2k-j+1}\partial^j_{y_3}w\|^2_{L^2(\Omega)}+\|e^{-\eta t}\partial_{y_1}^{2k-j+1}\partial^j_{y_3}w\|^2_{L^2(\omega^\ell_T)}\nonumber\\
    &\lesssim  \frac{1}{\varepsilon\eta}\|e^{-\eta t}L'(\partial_{y_1}^{2k-1}\partial_{y_3}w)\|_{L^2(\Omega_T)}^2+(1+\varepsilon\eta)\|e^{-\eta t}\md\partial_{y_1}^{2k-1}\partial_{y_3}w\|^2_{L^2(\Omega_T)}\nonumber\\
    &\quad+\sum_{|\alpha|\le 2k}\left( \frac{1}{\eta}\|e^{-\eta t}L'(\md^\alpha w)\|^2_{L^2(\Omega_T)}+\|e^{-\eta t}\mathcal{B}\md^\alpha w\|^2_{L^2(\omega_T^\ell)}\right)+\|\md^{2k-2}f|_{t=0}\|^2_{L^2(\Omega_T)}\nonumber\\
    &\quad+\delta\|e^{-\eta t}\md^{2k+1}w\|^2_{L^2(\omega^\ell_T)}+\sum_{|\alpha|\le 2k-1}\eta\|e^{-\eta t} L'(\md^\alpha w)\|^2_{L^2(\Omega_T)}+e^{-2\eta T}\|L'(\md^\alpha w)\|^2_{L^2(\Omega)}\nonumber\\
    &\quad+\sum_{|\alpha|\le 2k-1}\|e^{-\eta t}L'(\md^\alpha w)\|^2_{L^2(\omega^\ell_T)}.\label{ineq:3-90}
\end{align}
Indeed, from \eqref{ineq:3-87} and \eqref{ineq:3-89}, we know \eqref{ineq:3-90} is valid for $j=0,1,2$. Suppose \eqref{ineq:3-90} holds for all $j\le \ell$. We proceed to show \eqref{ineq:3-90} also holds for $j=\ell+1$.
In fact, one has
\begin{align}
    \partial^{2k-\ell}_{y_1}\partial^{\ell+1}_{y_3}w=\frac{1}{r_{33}}(L'-L_0-L_2-2r_{13}\partial_{13}-r_{11}\partial_{11})\partial_{y_1}^{2k-\ell}\partial_{y_3}^{\ell-1}w.
\end{align}
Hence we have
\begin{align}
    &\eta\|e^{-\eta t}\partial^{2k-\ell}_{y_1}\partial^{\ell+1}_{y_3}w\|^2_{L^2(\Omega_T)}+e^{-2\eta T}\|\partial^{2k-\ell}_{y_1}\partial^{\ell+1}_{y_3}w\|^2_{L^2(\Omega)}+\|e^{-\eta t}\partial^{2k-\ell}_{y_1}\partial^{\ell+1}_{y_3}w\|^2_{L^2(\omega^\ell_T)}\nonumber\\
    &\lesssim \eta\|e^{-\eta t}L'(\partial_{y_1}^{2k-\ell}\partial^{\ell-1}_{y_3}w)\|^2_{L^2(\Omega_T)}+e^{-2\eta T}\|L'(\partial_{y_1}^{2k-\ell}\partial^{\ell-1}_{y_3}w\|^2_{L^2(\Omega)}\nonumber\\
    &\qquad\qquad\qquad\qquad\qquad\qquad\qquad\qquad\qquad\qquad+\|e^{-\eta t}L'(\partial_{y_1}^{2k-\ell}\partial^{\ell-1}_{y_3}w\|^2_{L^2(\omega^\ell_T)}\nonumber\\
    &\quad+\eta\|e^{-\eta t}\md^{2k}\partial_{y_0}w\|^2_{L^2(\Omega_T)}+e^{-2\eta T}\|\md^{2k}\partial_{y_0}w\|^2_{L^2(\Omega)}+\|e^{-\eta t}\md^{2k}\partial_{y_0}w\|^2_{L^2(\omega^\ell_T)}\nonumber\\
   &\quad+\eta\|e^{-\eta t}\md^{2k}\partial_{y_2}w\|^2_{L^2(\Omega_T)}+e^{-2\eta T}\|\md^{2k}\partial_{y_2}w\|^2_{L^2(\Omega)}+\|e^{-\eta t}\md^{2k}\partial_{y_2}w\|^2_{L^2(\omega^\ell_T)}\nonumber\\
   &\quad+\eta\|e^{-\eta t}\partial^{2k-\ell+2}_{y_1}\partial_{y_3}^{\ell-1}w\|^2_{L^2(\Omega_T)}+e^{-2\eta T}\|\partial^{2k-\ell+2}_{y_1}\partial_{y_3}^{\ell-1}w\|^2_{L^2(\Omega)}\nonumber\\
   &\qquad\qquad\qquad\qquad\qquad\qquad\qquad\qquad\qquad\qquad+\|e^{-\eta t}\partial^{2k-\ell+2}_{y_1}\partial_{y_3}^{\ell-1}w\|^2_{L^2(\omega^\ell_T)}\nonumber\\
   &\quad+\eta\|e^{-\eta t}\partial^{2k-\ell+1}_{y_1}\partial_{y_3}^{\ell}w\|^2_{L^2(\Omega_T)}+e^{-2\eta T}\|\partial^{2k-\ell+1}_{y_1}\partial_{y_3}^{\ell}w\|^2_{L^2(\Omega)}\nonumber\\
   &\qquad\qquad\qquad\qquad\qquad\qquad\qquad\qquad\qquad\qquad+\|e^{-\eta t}\partial^{2k-\ell+1}_{y_1}\partial_{y_3}^{\ell}w\|^2_{L^2(\omega^\ell_T)}.
\end{align}
By our induction assumption that \eqref{ineq:3-90} is valid for $j\le \ell$ and \eqref{ineq:3-68}, we deduce that
\begin{align}
    &\eta\|e^{-\eta t}\partial^{2k-\ell}_{y_1}\partial^{\ell+1}_{y_3}w\|^2_{L^2(\Omega_T)}+e^{-2\eta T}\|\partial^{2k-\ell}_{y_1}\partial^{\ell+1}_{y_3}w\|^2_{L^2(\Omega)}+\|e^{-\eta t}\partial^{2k-\ell}_{y_1}\partial^{\ell+1}_{y_3}w\|^2_{L^2(\omega^\ell_T)}\nonumber\\
    &\lesssim  \frac{1}{\varepsilon\eta}\|e^{-\eta t}L'(\partial_{y_1}^{2k-1}\partial_{y_3}w)\|_{L^2(\Omega_T)}^2+(1+\varepsilon\eta)\|e^{-\eta t}\md\partial_{y_1}^{2k-1}\partial_{y_3}w\|^2_{L^2(\Omega_T)}\nonumber\\
    &\quad+\sum_{|\alpha|\le 2k}\left( \frac{1}{\eta}\|e^{-\eta t}L'(\md^\alpha w)\|^2_{L^2(\Omega_T)}+\|e^{-\eta t}\mathcal{B}\md^\alpha w\|^2_{L^2(\omega_T^\ell)}\right)+\|\md^{2k-2}f|_{t=0}\|^2_{L^2(\Omega_T)}\nonumber\\
    &\quad+\!\!\sum_{|\alpha|\le 2k-1}\left(\|e^{-\eta t}L'(\md^\alpha w)\|^2_{L^2(\omega^\ell_T)}+\eta\|e^{-\eta t} L'(\md^\alpha w)\|^2_{L^2(\Omega_T)}+e^{-2\eta T}\|L'(\md^\alpha w)\|^2_{L^2(\Omega)}\right)\nonumber\\
    &\quad+\delta\|e^{-\eta t}\md^{2k+1}w\|^2_{L^2(\omega^\ell_T)},\label{ineq:3-93}
\end{align}
which implies \eqref{ineq:3-90} holds for $j=\ell+1$ and this completes the induction.
Now we are able to conclude the estimate of $(2k+1)$-th order. Since $\md^{2k}\partial_{y_0}w$, $\md^{2k}\partial_{y_2}w$, $\md\partial_{y_1}^{2k-1}\partial_{y_3}w$, $\partial^{2k}_{y_1}\partial_{y_3}w$ and $\partial_{y_1}^{2k-j+1}\partial^{j}_{y_3}w(0\le j\le 2k+1)$ cover all  derivatives of $(2k+1)$-th order, the sum of \eqref{ineq:3-68} and \eqref{ineq:3-90} for $0\le j\le 2k+1$ yields
\begin{align}
    &\sum_{|\alpha|\le 2k+1}\left(\eta\|e^{-\eta t}\md^\alpha w\|^2_{L^2(\Omega_T)}+e^{-2\eta T}\|\md^\alpha w\|^2_{L^2(\Omega)}+\|e^{-\eta t}\md^\alpha w\|^2_{L^2(\omega^\ell_T)}\right)\nonumber\\
    &\quad\lesssim  \frac{1}{\varepsilon\eta}\|e^{-\eta t}L'(\partial_{y_1}^{2k-1}\partial_{y_3}w)\|_{L^2(\Omega_T)}^2+(1+\varepsilon\eta)\|e^{-\eta t}\md\partial_{y_1}^{2k-1}\partial_{y_3}w\|^2_{L^2(\Omega_T)}\nonumber\\
    &\qquad+\sum_{|\alpha|\le 2k}\left( \frac{1}{\eta}\|e^{-\eta t}L'(\md^\alpha w)\|^2_{L^2(\Omega_T)}+\|e^{-\eta t}\mathcal{B}\md^\alpha w\|^2_{L^2(\omega_T^\ell)}\right)\nonumber\\
    &\qquad+\!\!\sum_{|\alpha|\le 2k-1}\left(\|e^{-\eta t}L'(\md^\alpha w)\|^2_{L^2(\omega^\ell_T)}\right.+\eta\|e^{-\eta t} L'(\md^\alpha w)\|^2_{L^2(\Omega_T)}\nonumber\\
    &\qquad\quad+\left.e^{-2\eta T}\|L'(\md^\alpha w)\|^2_{L^2(\Omega)}\right)+\delta\|e^{-\eta t}\md^{2k+1}w\|^2_{L^2(\omega^\ell_T)}+\|\md^{2k-2}f|_{t=0}\|^2_{L^2(\Omega_T)}.\label{ineq:3-94}
\end{align}
Let $\delta$ and $\varepsilon$ be appropriately small and estimate the terms on the second last line of \eqref{ineq:3-94} by same arguments as \eqref{ineq:3-64} and \eqref{ineq:3-65}, then let $\eta$ be properly large, we are led to
\begin{align}
    &\sum_{|\alpha|\le 2k+1}\left(\eta\|e^{-\eta t}\md^\alpha w\|^2_{L^2(\Omega_T)}+e^{-2\eta T}\|\md^\alpha w\|^2_{L^2(\Omega)}+\|e^{-\eta t}\md^\alpha w\|^2_{L^2(\omega^\ell_T)}\right)\nonumber\\
    &\lesssim\sum_{|\alpha|\le 2k}\left( \frac{1}{\eta}\|e^{-\eta t}L'(\md^\alpha w)\|^2_{L^2(\Omega_T)}+\|e^{-\eta t}\mathcal{B}\md^\alpha w\|^2_{L^2(\omega_T^\ell)}\right)+\|\md^{2k-2}f|_{t=0}\|^2_{L^2(\Omega_T)},\label{ineq:3-95}
\end{align}
which is nothing but the estimate of $(2k+1)$-th order.

Next, we continue to derive the estimate of $(2k+2)$-th order, on the basis of the estimate of $(2k+1)$-th order. Apply \eqref{ineq:3-95} to $\partial_{y_0}w$ and $\partial_{y_2}w$, we have
\begin{align}
    &\sum_{|\alpha|\le 2k+1}\left(\eta\|e^{-\eta t}\md^\alpha\partial_{y_0} w\|^2_{L^2(\Omega_T)}+e^{-2\eta T}\|\md^\alpha\partial_{y_0} w\|^2_{L^2(\Omega)}+\|e^{-\eta t}\md^\alpha \partial_{y_0}w\|^2_{L^2(\omega^\ell_T)}\right)\nonumber\\
    &+\sum_{|\alpha|\le 2k+1}\left(\eta\|e^{-\eta t}\md^\alpha\partial_{y_2} w\|^2_{L^2(\Omega_T)}+e^{-2\eta T}\|\md^\alpha\partial_{y_2} w\|^2_{L^2(\Omega)}+\|e^{-\eta t}\md^\alpha \partial_{y_2}w\|^2_{L^2(\omega^\ell_T)}\right)\nonumber\\
    &\lesssim\sum_{|\alpha|\le 2k+1}\left( \frac{1}{\eta}\|e^{-\eta t}L'(\md^\alpha w)\|^2_{L^2(\Omega_T)}+\|e^{-\eta t}\mathcal{B}\md^\alpha w\|^2_{L^2(\omega_T^\ell)}\right)+\|\md^{2k}f|_{t=0}\|^2_{L^{2}(\Omega_T)}.\label{ineq:3-96}
\end{align}
Then we will firstly establish the first order derivative of $\partial^{2k}_{y_1}\partial_{y_3}w$.
It is clear that
\begin{align}
   L'(\partial^{2k}_{y_1}\partial_{y_3}w)=\partial^{2k}_{y_1}\partial_{y_3}f-[\partial^{2k}_{y_1}\partial_{y_3},L']w\quad \mbox{in}\quad \Omega_T,&\label{eq:3-97}\\
   \partial^{2k}_{y_1}\partial_{y_3}w=0\quad\mbox{on}\quad \omega^r_T,&\\
   (\partial^{2k}_{y_1}\partial_{y_3}w,\partial_{y_0}(\partial^{2k}_{y_1}\partial_{y_3}w))=(0,0)\quad\mbox{on}\quad \Gamma_{in}.&
\end{align}
We have to deduce the boundary condition on $\omega^\ell$ for $\partial^{2k}_{y_1}\partial_{y_3}w$. By the definitions of $L'$, $L_0$ and $L_2$, it is clear that
\begin{align}
    \partial_{y_1}(\partial_{y_1}^{2k}\partial_{y_3}w)=(L'-L_0-L_2-2r_{13}\partial_{13}-r_{33}\partial_{33})\partial_{y_1}^{2k-1}\partial_{y_3}w.\label{eq:3-100}
\end{align}
Hence, we need to determine $\partial_{y_1}^{2k-1}\partial^3_{y_3}w$ on $\omega^\ell$.
From the boundary condition $\mathcal{B}w=g$ on $w^\ell$, we notice that
\begin{align}
    \partial_{y_1}\partial^{2k+1}_{y_3}w=(\mathcal{B}-b_0\partial_{y_0}-\partial_{y_2}\partial_{y_2}-b_3\partial_{y_3})\partial_{y_1}^{2k+1}w.\nonumber
\end{align}
Thus we have
\begin{align}
    \abs{\partial_{y_1}\partial^{2k+1}_{y_3}w}\lesssim \mathcal{B}\partial^{2k+1}_{y_1}w+\abs{\md^{2k+1}\partial_{y_0}w}+\abs{\md^{2k+1}\partial_{y_2}w}+\delta\abs{\md^{2k+2}w}.\label{ineq:3-101}
\end{align}
Again by the definitions of $L'$, $L_0$ and $L_2$,  we can further deduce
\begin{align}
    \partial_{y_1}^{2k-2j-1}\partial_{y_3}^{2j+3}w=\frac{1}{r_{33}}(L'-L_0-L_2-2r_{13}-r_{11}\partial_{11})\partial_{y_1}^{2k-2j-1}\partial_{y_3}^{2j+1}w.\label{eq:3-102}
\end{align}
For $1\le j\le k$, let
\begin{align}
    \alpha_j&\defs \partial_{y_1}^{2k-2j+1}\partial_{y_3}^{2j+1}w,\label{eq:3-103}\\
    B_j&\defs\frac{1}{r_{33}}(L'-L_0-L_2-2r_{13})\partial_{y_1}^{2k-2j-1}\partial_{y_3}^{2j+1}w.\label{eq:3-104}
\end{align}
From \eqref{ineq:3-101}-\eqref{eq:3-104}, we obtain a finite sequence $\{\alpha_j\}_{j=1}^k$ satisfying
%\begin{align}
%    \alpha_{j+1}=B_j-r_{11}\alpha_j.
%\end{align}
\begin{equation*}
    \left\{\!\!
    \begin{array}{ll}
    \alpha_{k}\lesssim \abs{\mathcal{B}\partial^{2k+1}_{y_1}w}+\abs{\md^{2k+1}\partial_{y_0}w}+\abs{\md^{2k+1}\partial_{y_2}w}+\delta\abs{\md^{2k+2}w},\\
    \alpha_{j+1}=B_j-\frac{r_{11}}{r_{33}}\alpha_j\mbox{\ }(j=1,2,\cdots k),\\
    \abs{B_j}\lesssim \abs{L'(\partial_{y_1}^{2k-2j-1}\partial_{y_3}^{2j+1}w)}+\abs{\md^{2k+1}\partial_{y_0}w}+\abs{\md^{2k+1}\partial_{y_2}w}+\delta\abs{\md^{2k+2}w}.
\end{array}
    \right.
\end{equation*}
Analogous to the sequence $\{\beta_j\}$, by induction on $j$, one can deduce that $\alpha_j$ $(1\le j\le k)$ satisfies
\begin{align}
    \abs{\alpha_j}\lesssim \sum_{|\alpha|\le 2k}\abs{L'(\md^\alpha w)}+\abs{\md^{2k+1}\partial_{y_0}w}+\abs{\md^{2k+1}\partial_{y_2}w}+\delta\abs{\md^{2k+2}w}+\!\!\!\sum_{|\alpha|\le 2k+1}\!\!\abs{\mathcal{B}\md^\alpha w},\nonumber
\end{align}
and so does $\alpha_1=\partial^{2k-1}_{y_1}\partial_{y_3}^{3}w$. Armed with this estimate for $\partial^{2k-1}_{y_1}\partial_{y_3}^{3}w$, we obtain from \eqref{eq:3-100} that
\begin{align}
    \abs{\partial_{y_1}(\partial^{2k}_{y_1}\partial_{y_3}w)}&\lesssim \sum_{|\alpha|\le 2k}\abs{L'(\md^\alpha w)}+\abs{\md^{2k+1}\partial_{y_0}w}+\abs{\md^{2k+1}\partial_{y_2}w}\nonumber\\
    &\qquad\quad+\delta\abs{\md^{2k+2}w}+\!\!\!\sum_{|\alpha|\le 2k+1}\!\!\abs{\mathcal{B}\md^\alpha w}.\label{ineq:3-105}
\end{align}
Thanks to lemma \ref{Multiplier for almost normal boundary conditions}, we are able to derive the first order estimate of $\partial^{2k}_{y_1}\partial_{y_3}w$. Multiplying $2e^{-2\eta t}\mcq^e(\partial^{2k}_{y_1}\partial_{y_3}w)$ on both sides of \eqref{eq:3-97}, integration by parts over $\Omega_T$, applying lemma \ref{Multiplier for almost normal boundary conditions} and Cauchy inequality, one has
\begin{align}
   & \eta\|e^{-\eta t}\nabla_{\mby}\partial^{2k}_{y_1}\partial_{y_3}w\|^2_{L^2(\Omega_T)}+e^{-2\eta T}\|\nabla_{\mby}\partial^{2k}_{y_1}\partial_{y_3}w\|^2_{L^2(\Omega)}\nonumber\\
    &+\|e^{-\eta t}\partial_{y_3}\partial^{2k}_{y_1}\partial_{y_3}w\|^2_{L^2(\omega^\ell_T)}+\|e^{-\eta t}\partial_{y_3}\partial^{2k}_{y_1}\partial_{y_3}w\|^2_{L^2(\omega^r_T)}\nonumber\\
    &\qquad\lesssim \frac{1}{\varepsilon\eta}\|e^{- \eta t}L'(\partial^{2k}_{y_1}\partial_{y_3}w)\|^2_{L^2(\Omega_T)}+(1+\varepsilon\eta)\|e^{-\eta t}\md\partial_{y_1}^{2k}\partial_{y_3}w\|^2_{L^2(\Omega_T)}\nonumber\\
    &\quad\qquad+\eta\|e^{-\eta t}\partial_{y_0}\partial^{2k}_{y_1}\partial_{y_3}w\|^2_{L^2(\Omega_T)}+e^{-2\eta T}\|\partial_{y_0}\partial^{2k}_{y_1}\partial_{y_3}w\|^2_{L^2(\Omega)}\nonumber\\
    &\quad\qquad+\sum_{i=0}^2\|e^{-\eta t}\partial_{y_i}\partial^{2k}_{y_1}\partial_{y_3}w\|^2_{L^2(\omega^\ell_T)}.\nonumber
\end{align}
In view of \eqref{ineq:3-96} and \eqref{ineq:3-105}, we obtain from above inequality that
\begin{align}
    & \eta\|e^{-\eta t}\md\partial^{2k}_{y_1}\partial_{y_3}w\|^2_{L^2(\Omega_T)}+e^{-2\eta T}\|\md\partial^{2k}_{y_1}\partial_{y_3}w\|^2_{L^2(\Omega)}\nonumber\\
    &+\|e^{-\eta t}\md\partial^{2k}_{y_1}\partial_{y_3}w\|^2_{L^2(\omega^\ell_T)}+\|e^{-\eta t}\partial_{y_3}\partial^{2k}_{y_1}\partial_{y_3}w\|^2_{L^2(\omega^r_T)}\nonumber\\
    &\lesssim \frac{1}{\varepsilon\eta}\|e^{- \eta t}L'(\partial^{2k}_{y_1}\partial_{y_3}w)\|^2_{L^2(\Omega_T)}+(1+\varepsilon\eta)\|e^{-\eta t}\md\partial_{y_1}^{2k}\partial_{y_3}w\|^2_{L^2(\Omega_T)}\nonumber\\
    &\quad+\sum_{|\alpha|\le 2k+1}\left( \frac{1}{\eta}\|e^{-\eta t}L'(\md^\alpha w)\|^2_{L^2(\Omega_T)}+\|e^{-\eta t}\mathcal{B}\md^\alpha w\|^2_{L^2(\omega_T^\ell)}\right)+\|\md^{2k}f|_{t=0}\|^2_{L^{2}(\Omega_T)}\nonumber\\
    &\quad+\sum_{|\alpha|\le 2k}\|e^{-\eta t}L'(\md^\alpha w)\|^2_{L^2(\omega^\ell_T)}+\delta\|\md^{2k+2}w\|^2_{L^2(\omega^\ell_T)}.\label{ineq:3-106}
\end{align}
Now we turn to the estimate of $\partial_{y_1}^{2k+2}w$ and the estimate of the derivatives in the form of $\partial^{2k-j+2}_{y_1}\partial^j_{y_3}w$ with $3\le j\le 2k+2$.
By the definitions of $L'$, $L_0$ and $L_2$, one has
\begin{align}
    \partial_{y_1}^{2k+2}w=\frac{1}{r_{11}}\left(L'-L_0-L_2-2r_{13}\partial_{13}-r_{33}\partial_{33}\right)\partial_{y_1}^{2k}w.
\end{align}
So $\partial_{y_1}^{2k+2}w$ can be estimated by controlled terms, i.e.,
\begin{align}
    \abs{\partial_{y_1}^{2k+2}w}\lesssim \abs{L'(\partial_{y_1}^{2k}w)}+\abs{\md^{2k+1}\partial_{y_0}w}+\abs{\md^{2k+1}\partial_{y_2}w}+\abs{\md\partial_{y_1}^{2k}\partial_{y_3}w}.
\end{align}
This together with \eqref{ineq:3-96} and \eqref{ineq:3-106} imply
\begin{align}
   &\eta\|e^{-\eta t}\partial_{y_1}^{2k+2}w\|^2_{L^2(\Omega_T)}+e^{-2\eta T}\|\partial_{y_1}^{2k+2}w\|^2_{L^2(\Omega)}+
    \|e^{-\eta t}\partial_{y_1}^{2k+2}w\|^2_{L^2(\omega^\ell_T)}\nonumber\\
     &\lesssim \frac{1}{\varepsilon\eta}\|e^{- \eta t}L'(\partial^{2k}_{y_1}\partial_{y_3}w)\|^2_{L^2(\Omega_T)}+(1+\varepsilon\eta)\|e^{-\eta t}\md\partial_{y_1}^{2k}\partial_{y_3}w\|^2_{L^2(\Omega_T)}\nonumber\\
    &\quad+\sum_{|\alpha|\le 2k+1}\left( \frac{1}{\eta}\|e^{-\eta t}L'(\md^\alpha w)\|^2_{L^2(\Omega_T)}+\|e^{-\eta t}\mathcal{B}\md^\alpha w\|^2_{L^2(\omega_T^\ell)}\right)+\|\md^{2k}f|_{t=0}\|^2_{L^{2}(\Omega_T)}\nonumber\\
    &\quad+\sum_{|\alpha|\le 2k}\|e^{-\eta t}L'(\md^\alpha w)\|^2_{L^2(\omega^\ell_T)}+\delta\|\md^{2k+2}w\|^2_{L^2(\omega^\ell_T)}+\eta\|e^{-\eta t}L'(\partial_{y_1}^{2k}w)\|^2_{L^2(\Omega_T)}\nonumber\\
    &\quad+e^{-2\eta T}\|L'(\partial_{y_1}^{2k}w)\|^2_{L^2(\Omega)}+
    \|e^{-\eta t}L'(\partial_{y_1}^{2k}w)\|^2_{L^2(\omega^\ell_T)}.
\end{align}
Then by simple induction argument as we use in \eqref{ineq:3-90}-\eqref{ineq:3-93}, one deduces for all $3\le j\le 2k+2$ that
\begin{align}
    &\eta\|e^{-\eta t}\partial^{2k-j+2}_{y_1}\partial^j_{y_3}w\|^2_{L^2(\Omega_T)}+e^{-2\eta T}\|\partial^{2k-j+2}_{y_1}\partial^j_{y_3}w\|^2_{L^2(\Omega)}+\|e^{-\eta t}\partial^{2k-j+2}_{y_1}\partial^j_{y_3}w\|^2_{L^2(\omega^\ell_T)}\nonumber\\
    &\lesssim \frac{1}{\varepsilon\eta}\|e^{-\eta t}L'(\partial_{y_1}^{2k}\partial_{y_3}w)\|_{L^2(\Omega_T)}^2+(1+\varepsilon\eta)\|e^{-\eta t}\md\partial_{y_1}^{2k}\partial_{y_3}w\|^2_{L^2(\Omega_T)}\nonumber\\
    &\quad+\sum_{|\alpha|\le 2k+1}\left( \frac{1}{\eta}\|e^{-\eta t}L'(\md^\alpha w)\|^2_{L^2(\Omega_T)}+\|e^{-\eta t}\mathcal{B}\md^\alpha w\|^2_{L^2(\omega_T^\ell)}\right)+\|\md^{2k}f|_{t=0}\|^2_{L^{2}(\Omega_T)}\nonumber\\
    &\quad+\delta\|e^{-\eta t}\md^{2k+1}w\|^2_{L^2(\omega^\ell_T)}+\sum_{|\alpha|\le 2k}\eta\|e^{-\eta t} L'(\md^\alpha w)\|^2_{L^2(\Omega_T)}+e^{-2\eta T}\|L'(\md^\alpha w)\|^2_{L^2(\Omega)}\nonumber\\
    &\quad+\sum_{|\alpha|\le 2k}\|e^{-\eta t}L'(\md^\alpha w)\|^2_{L^2(\omega^\ell_T)}.\label{ineq:3-110}
\end{align}
To this end, adding \eqref{ineq:3-96}, \eqref{ineq:3-106} and \eqref{ineq:3-110} for all $3\le j\le 2k+2$ together, then let $\varepsilon$, $\delta$ be properly small and $\eta$ be appropriately large, one concludes that
\begin{align}
    &\sum_{|\alpha|\le 2k+2}\left(\eta\|e^{-\eta t}\md^\alpha w\|^2_{L^2(\Omega_T)}+e^{-2\eta T}\|\md^\alpha w\|^2_{L^2(\Omega)}+\|e^{-\eta t}\md^\alpha w\|^2_{L^2(\omega^\ell_T)}\right)\nonumber\\
    &\lesssim \sum_{|\alpha|\le 2k+1}\left( \frac{1}{\eta}\|e^{-\eta t}L'(\md^\alpha w)\|^2_{L^2(\Omega_T)}+\|e^{-\eta t}\mathcal{B}\md^\alpha w\|^2_{L^2(\omega_T^\ell)}\right)\nonumber\\
    &\quad+\|\md^{2k}f|_{t=0}\|^2_{L^{2}(\Omega_T)}.\label{ineq:3-111}
\end{align}
This completes the induction process from the estimate of $2k$-th order to $(2k+2)$-th order and hence finishes our proof of proposition \ref{well-posedness of LPE}.
\end{proof}
\subsection{Proof of theorem \ref{well-posedness and energy estimate of the LP}}
Based on proposition \eqref{well-posedness of LPE}, we are able to prove theorem \ref{well-posedness and energy estimate of the LP} by carefully estimating $L'(\md^\alpha w)$ for $|\alpha|\le s\le n_0+2$.

\begin{proof}[Proof of Theorem \ref{well-posedness of LPE}]
%The existence and uniqueness follows from the energy estimate \eqref{neq:estimate of lpe} and the fact that the boundary conditions are admissible in the sense of Friedrichs \cite{Friedrichs} (also see \cite[pp. 168-170]{Rauch}). Because the argument is standard, we omit the details for the shortness and only show \eqref{ineq: LP energy estimate}.
It is clear that the estimate in Proposition \ref{well-posedness and energy estimate of the LP} holds for all $T\ge 0$. Hence, in order to prove theorem \ref{well-posedness and energy estimate of the LP}, we just need to estimate $L'\left(\mathrm{D}^\alpha w\right)$ and $\mathcal{B}(\mathrm{D}^\alpha w)$.
First, for $L'\left(\mathrm{D}^\alpha w\right)$, actually
	\begin{equation}
	\begin{split}
	L'\left(\mathrm{D}^\alpha w\right)&=-\left[\mathrm{D}^\alpha,L'\right]w+\mathrm{D}^\alpha\left(L'w\right)\\
	&=-[\mathrm{D}^\alpha,L']w+\mathrm{D}^\alpha f.
	\end{split}
	\end{equation}
Then we need to estimate the commutator $[\mathrm{D}^\alpha,L']w$.
By definition

\[
[\mathrm{D}^\alpha,L']w=\mathrm{D}^\alpha(r_{ij}\partial_{ij}w)-r_{ij}\mathrm{D}^\alpha\partial_{ij}w+\mathrm{D}^\alpha(r_i\partial_{i}w_1)-a_i\mathrm{D}^\alpha(\partial_{i}w)+\mathrm{D}^\alpha(rw)-r\mathrm{D}^\alpha w.
\]
For the commutator, we claim:
\begin{equation}\label{3.115}
\begin{array}{ll}
&\mbox{$[\mathrm{D}^\alpha,L']w$ is a linear combination of finitely many terms, and each}\\
&\mbox{term is a product of derivatives of $u$ and $w$, in which at most one factor} \\
&\mbox{has $u$ and $w$ differentiated more than $\frac{|\alpha|+2}{2}$ times.}
\end{array}
\end{equation}
To show claim \eqref{3.115}, we observe that
$\mathrm{D}^\alpha(r_{ij}\partial_{ij}w)-r_{ij}\mathrm{D}^\alpha\partial_{ij}w$ is a linear combination of the terms of the following form:
	$$
	p_{ij}(\nabla\mcw(u),\mathrm{D}u)\partial^{\mu_1}\nabla\mcw\partial^{\mu_2}\nabla\mcw\cdots\partial^{\mu_l}\nabla\mcw\partial^{\gamma_1}\mathrm{D}u\partial^{\gamma_2}\mathrm{D}u\cdots\partial^{\gamma_k}\mathrm{D}u\partial^{\sigma}\partial_{ij}w
	$$
	in which $|\mu_1|+|\mu_2|+\cdots+|\mu_l|+|\gamma_1|+|\gamma_2|+\cdots+|\gamma_k|+|\sigma|=|\alpha|$ and $|\sigma|\le |\alpha|-1$.

\textit{Case }1: If $|\sigma|\ge \frac{|\alpha|-2}{2}$, then
		$|\mu_1|+|\mu_2|+\cdots+|\mu_l|+|\gamma_1|+|\gamma_2|+\cdots+|\gamma_k|\le \frac{(|\alpha|+2)}{2}$. So $|\mu_j|\le \frac{(|\alpha|+2)}{2}$ and $|\gamma_j|\le \frac{(|\alpha|+2)}{2}$ and they cannot achieve $\frac{|\alpha|+2}{2}$ at the same time.
	
\textit{Case }2: If $|\sigma|< \frac{|\alpha|-2}{2}$, then $|\mu_1|+|\mu_2|+\cdots+|\mu_l|+|\gamma_1|+|\gamma_2|+\cdots+|\gamma_k|\le |\alpha|$. So there is at most one index among $\{\mu_1,\cdots,\gamma_k\}$ whose value is larger than $\frac{|\alpha|}{2}$.
Because $|\alpha|\le s-1$, we have $\frac{(|\alpha|+2)}{2}\le s$.
Similar argument to $\mathrm{D}^\alpha(r_i\partial_{i}w)-r_i\mathrm{D}^\alpha(\partial_{i}w_1)$ and $\mathrm{D}^\alpha(r\partial_{i}w)-r\mathrm{D}^\alpha(\partial_{i}w)$ implies both of them have similar forms as the one for $\mathrm{D}^\alpha(r_{ij}\partial_{ij}w)-r_{ij}\mathrm{D}^\alpha\partial_{ij}w$. Therefore, claim \eqref{3.115} holds.
Based on \ref{h3}, equation $\eqref{Upper bound of un}$, and the claim \eqref{3.115}, we have
{\small\begin{equation*}
\left|[\mathrm{D}^\alpha,L']w\right|  \le C_\delta(\mathop{\sum_{1\le|\gamma|\le|\alpha|}}_ {1\leq\sigma\leq\frac{|\alpha|-2}{2}}\abs{\mathrm{D}^\gamma u}\|\md^{\sigma}\partial_{ij}w\|_{L^{\infty}}
+\!\!\!\sum_{|\gamma|=|\alpha|+1}\abs{\mathrm{D}^\gamma u}|\partial_{ij}w|+\!\!\!\sum_{2\le|\gamma|\le|\alpha|+1}\abs{\mathrm{D}^\gamma w}).
\end{equation*}}	
Note that $|\alpha|\leq s-1$. So if $7\le s\le n_0+3$, then the Sobolev embedding theorem indicates $\|\md^{\sigma}\partial_{ij}w\|_{L^{\infty}}\leq\|w\|_{H^{s}}$. Therefore,
\begin{align}
\sum_{|\alpha|\le s-1}\|L'(\mathrm{D}^\alpha w)\|_{L^2}
\lesssim&\|w\|_{H^{s}}+\|f\|_{H^{s-1}}+\delta\|\md^2w\|_{L^\infty(\Omega)}\|u\|_{H^{s}}.
\end{align}

Then by choosing $\eta$ be large and $\delta$ be small, it follows from Proposition \ref{well-posedness of LPE} that
\begin{align}
& \sum_{|\alpha|\leq s} \eta\|e^{-\eta t}\md^\alpha w\|^2_{L^2(\Omega_T)}+e^{-2 \eta T}\|\md^\alpha w\|^2_{L^2(\Omega)}+\|e^{-\eta t}\md^\alpha w\|^2_{L^2(\omega^\ell_T)}\nonumber\\
&\quad\lesssim \frac{1}{\eta}\left(\|e^{-\eta t}u\|^2_{H^s(\Omega_T)}\cdot\sup_{0\le t\le T}\|w(t,\cdot)\|^2_{H^4(\Omega)}+\|e^{-\eta t}f\|^2_{H^{s-1}(\Omega_T)}\right)\nonumber\\
&\quad\quad+\!\!\sum_{|\alpha|\le s-1}\|e^{-\eta t}\mcb\md^\alpha w\|^2_{L^2(\omega^\ell_T)}.\label{3-117}
\end{align}
%\sum_{j\le s}\int_0^{y_0}e^{-2\eta t}(\left\|\partial_t^{j+1}h_2(t,\cdot)\right\|^2_{H^{s-j}(\mathbb{R})}+\left\|\partial_t^{j}h_2(t,\cdot)\right\|^2_{H^{s-j}(\mathbb{R})})\mathrm{d}t+\|v_0\|^2_{H^{s+1}(\mathbb{R}^2_+)}+\|v_1\|^2_{H^{s}(\mathbb{R}^2_+)}\nonumber\\
%&+\frac{1}{\eta^2}e^{2\eta y_0}\|e^{-\eta t}w_1\|^2_{H^4(\Omega_{y_0})}\cdot\sum_{|\alpha|\leq s+1} \halfinterior{e^{-2\eta t}\abs{\mathrm{D}^{\alpha}u}^2}\nonumber\\
%\lesssim &
%\sum_{j\le s}\int_0^{y_0}e^{-2\eta t}(\left\|\partial_t^{j+1}h_2(t,\cdot)\right\|^2_{H^{s-j}(\mathbb{R})}+\left\|\partial_t^{j}h_2(t,\cdot)\right\|^2_{H^{s-j}(\mathbb{R})})\mathrm{d}t+\|v_0\|^2_{H^{s+1}(\mathbb{R}^2_+)}+\|v_1\|^2_{H^{s}(\mathbb{R}^2_+)}\nonumber\\
%&+\sum_{|\alpha|\leq s+1} \frac1\eta\halfinterior{e^{-2\eta t}\abs{\mathrm{D}^{\alpha}u}^2} \cdot\left(\sum_{j\le 3}\int_0^{y_0}e^{-2\eta t}(\left\|\partial_t^{j+1}h_2(t,\cdot)\right\|^2_{H^{3-j}(\mathbb{R})}+\left\|\partial_t^{j}h_2(t,\cdot)\right\|^2_{H^{3-j}(\mathbb{R})})\mathrm{d}t\right.\nonumber\\
%&\qquad\qquad\qquad\qquad\left.+\|v_0\|^2_{H^{4}(\mathbb{R}^2_+)}+\|v_1\|^2_{H^{3}(\mathbb{R}^2_+)}\right),\label{(s+1) estimation of w_1 }
%\end{align}
where constant $C$ does not depend on $u$ and we choose $y_0$ sufficiently small.

%Next, let us consider the term $\quarterboundaryv{e^{-2\eta t}\abs{\mathcal{B}\mathrm{D}^\alpha w_2}^2}\Big|_{y_1=0}$ in \eqref{3.148x}.
Note that
\begin{equation}
\begin{split}
\mathcal{B}\left(\mathrm{D}^\alpha w\right)&=-\left[\mathrm{D}^\alpha,\mathcal{B}\right]w+\mathrm{D}^\alpha\left(\mathcal{B}w\right)\\
&=-[\mathrm{D}^\alpha,\mathcal{B}]w+\mathrm{D}^\alpha g.
\end{split}
\end{equation}

Similar to claim \eqref{3.115}, $\left[\mathrm{D}^\alpha,\mathcal{B}\right]w$ is a linear combination of finitely many terms, and each term is a product of derivatives of $u$ and $w$, in which at most one factor has $u$ and $w$ differentiated more than $\frac{(|\alpha|+1)}{2}$ times.

Therefore, by the trace theorem, one has
\begin{align}
 &\sum_{|\alpha|\le s-1}\int_0^{y_0}e^{-2\eta t}\|\mathcal{B}\mathrm{D}^\alpha w(t,0,\cdot)\|^2_{L^2(\mathbb{R_+})}\mathrm{d}t\nonumber\\
&\quad\lesssim  \sum_{|\alpha|\le s-1}\int_0^{y_0}e^{-2\eta t}(\delta\|\mathrm{D}^\alpha w(t,0,\cdot)\|^2_{L^2(\mathbb{R_+})}
+\|\mathrm{D}^\alpha g(t,0,\cdot)\|^2_{L^2(\mathbb{R_+})})\mathrm{d}t\nonumber\\
&\quad\lesssim  \sum_{|\alpha|\le s}\delta\int_0^{y_0}e^{-2\eta t}\|\mathrm{D}^\alpha w(t,\cdot)\|^2_{L^2(\Omega)}\mathrm{d}t+\!\!\!\sum_{|\alpha|\le s-1}\int_0^{y_0}\!\!e^{-2\eta t}\|\mathrm{D}^\alpha g(t,0,\cdot)\|^2_{L^2(\mathbb{R_+})}\mathrm{d}t.\label{3-119}
\end{align}
%from \eqref{3.143} that
%\begin{align}
%&\eta\sum_{|\alpha|\le s}\quarterinterior{e^{-2\eta t}\abs{\mathrm{D}^\alpha v}^2}+\quarternotime{e^{-2\eta t}\abs{\mathrm{D}^\alpha v}^2}\Big|_{t=y_0}\nonumber\\
%&+\quarterboundaryv{e^{-2\eta t}\abs{\mathrm{D}^\alpha w_2}^2}\Big|_{y_1=0}\nonumber\\
%\lesssim & \|e^{-\eta t}g\|^2_{H^{s}(\Omega_{y_0})} %\sum_{|\alpha|\le s-1}\quarterinterior{e^{-2\eta t}\abs{\mathrm{D}^\alpha g}^2}
%+\|v_0\|^2_{H^{s}(\mathbb{R}^2_+)}+\|v_1\|^2_{H^{s-1}(\mathbb{R}^2_+)}\nonumber\\
%&+\sum_{\alpha_0+\alpha_1\le s-1}(\int_0^{y_0}\int_{\mathbb{R}}e^{-2\eta t}|\partial_t^{\alpha_0+1}\partial_{y_1}^{\alpha_1}h_2|^2\mathrm{d}y_1\mathrm{d}t
%+\eta\halfboundaryh{e^{-2\eta t}\abs{\partial_t^{\alpha_0}\partial_{y_1}^{\alpha_1}h_2}^2})\nonumber\\
%&+\quarterboundaryv{e^{-2\eta t}\abs{\mathcal{B}\mathrm{D}^\alpha w_2}^2}\Big|_{y_1=0}.\label{3.148}
%\end{align}

%%%%%%%%%%%%%%%%%%%%%%%%%%%%%%%%%%%%%%%%%%%%%%%%%%%%%%%%%%%%%%%%%%%%%%%%%%%
Substitute \eqref{3-119} into \eqref{3-117}, let $s=4$ in \eqref{3-117} and repeat above process, then let $\delta$, $T$ and $\frac{1}{\eta}$ be small, one can deduce that solution $w$ of problem \eqref{LP} satisfies the estimate $\eqref{ineq: LP energy estimate}$.
%\begin{align}
%	&\eta\sum_{|\alpha|\le s}\quarterinterior{e^{-2\eta t}\abs{\mathrm{D}^\alpha v}^2}+\quarternotime{e^{-2\eta t}\abs{\mathrm{D}^\alpha v}^2}\Big|_{t=y_0}\nonumber\\
%	&+\quarterboundaryv{e^{-2\eta t}\abs{\mathrm{D}^\alpha w_2}^2}\Big|_{y_1=0}\nonumber\\
%	\lesssim & \|e^{-\eta t}g\|^2_{H^{s-1}(\Omega_{y_0})}%\sum_{|\alpha|\le s-1}\quarterinterior{e^{-2\eta t}\abs{\mathrm{D}^\alpha g}^2}
%	+\|v_0\|^2_{H^{s+1}(\mathbb{R}^2_+)}+\|v_1\|^2_{H^{s}(\mathbb{R}^2_+)}\nonumber\\
%	&+\sum_{\alpha_0+\alpha_1\le s}(\int_0^{y_0}\int_{\mathbb{R}}e^{-2\eta t}|\partial_t^{\alpha_0+1}\partial_{y_1}^{\alpha_1}h_2|^2\mathrm{d}y_1\mathrm{d}t
%	+\eta\halfboundaryh{e^{-2\eta t}\abs{\partial_t^{\alpha_0}\partial_{y_1}^{\alpha_1}h_2}^2})\nonumber\\
%	&+\sum_{|\alpha|\le s-1}\int_0^{y_0}e^{-2\eta t}\|\mathrm{D}^\alpha h_1(t,0,\cdot)\|^2_{L^2(\mathbb{R_+})}\mathrm{d}t\nonumber\\
%	&+\sum_{|\alpha|\leq s+1} \halfinterior{e^{-2\eta t}\abs{\mathrm{D}^{\alpha}u}^2}\cdot\left(\sum_{j\le 3}\int_0^{y_0}e^{-2\eta t}(\left\|\partial_t^{j+1}h_2(t,\cdot)\right\|^2_{H^{3-j}(\mathbb{R})}+\left\|\partial_t^{j}h_2(t,\cdot)\right\|^2_{H^{3-j}(\mathbb{R})})\mathrm{d}t\right.\nonumber\\
%	&\qquad\qquad\left.+\|e^{-\eta t}g\|^2_{H^{3}(\Omega_{y_0})}+\|v_0\|^2_{H^{4}(\mathbb{R}^2_+)}+\|v_1\|^2_{H^{3}(\mathbb{R}^2_+)}\right).
%\end{align}
This completes the proof of this theorem.
\end{proof}
\section{Well-posedness of the non-linear problem}\label{sec:4}
\subsection{Reformulation of the non-linear problem}
We firstly reformulate the non-linear problem $\eqref{NLP}$.
Let
$$\bar{u}\defs u+y_3\frac{\partial_{x_1}\mcw}{1+|\partial_{x_1}\mcw|^2+|\partial_{x_2}\mcw|^2},$$
where $u$ is the solution to the \eqref{NLP}. Since
\begin{align}
    \frac{\partial \bar{u}}{\partial u}&=1+y_3\left(\frac{\partial_{x_1x_1}\mcw+\partial_{x_1x_2}\mcw\dfrac{\partial x_2}{\partial u}}{1+|\partial_{x_1}\mcw|^2+|\partial_{x_2}\mcw|^2}\right)\left(-\frac{1-|\partial_{x_1}\mcw|^2+|\partial_{x_2}\mcw|^2}{1+|\partial_{x_1}\mcw|^2+|\partial_{x_2}\mcw|^2}\right)\nonumber\\
    &-y_3\frac{2\partial_{x_1}\mcw\partial_{x_2}\mcw}{(1+|\partial_{x_1}\mcw|^2+|\partial_{x_2}\mcw|^2)^2}\left(\partial_{x_1x_2}\mcw+\partial_{x_2x_2}\mcw\dfrac{\partial x_2}{\partial u}\right),
\end{align}
one has $\frac{\partial \bar{u}}{\partial u}>0$, when $y_3$ and $\|\mcw\|_{W^{2,\infty}}$ is sufficiently small. Then by the implicit function theorem, $u$ can be expressed as a function with respect to $\bar{u}$, $y_2$ and $y_3$. We assume $u=\kappa(\bar{u},y_2,y_3)$ for some smooth function $\kappa$. By the property of our background solution, i.e., the nozzle wall $\Gamma_0$ is flat at the background solution, we have $\bar{u}=u_b$, if $u=u_b$. That is to say $\kappa(u_b,y_2,y_3)=u_b$. For notational simplicity, let
\[
N(x_1,x_2)=\left(\frac{\partial_{x_1}\mcw}{1+|\partial_{x_1}\mcw|^2+|\partial_{x_2}\mcw|^2}\right)(x_1,x_2).
\]
Then by direct computation, one has
\begin{equation}\label{4-2}
    \left\{\!\!
    \begin{array}{ll}
     \partial_{\bar{u}}\kappa=\dfrac{1+y_3\partial_{x_2}N}{1+y_3(\partial_{x_1}N+\partial_{x_2}N)},\\
     \partial_{y_2}\kappa=-\dfrac{y_3\partial_{x_2}N}{1+y_3(\partial_{x_1}N+\partial_{x_2}N)},\\
    \partial_{y_3}\kappa=-\dfrac{N}{1+y_3(\partial_{x_1}N+\partial_{x_2}N)}.
\end{array}
    \right.
\end{equation}
It is easy to see that $\partial_{\bar{u}}\kappa$ is close to one and $\partial_{\bar{u}}\kappa|_{\omega^r}\equiv 1$, while $\partial_{y_2}\kappa$ and $\partial_{y_3}\kappa$ are close to zero. The second order derivatives of $\kappa$ with respect ot $\bar{u}$, $y_2$, $y_3$ is listed in the appendix.
%\begin{equation}
%    \left\{\!\!
%    \begin{array}{ll}
%     \partial_{\bar{u}\bar{u}}\kappa=1\\
%     \partial_{y_2y_3}\kappa=1\\
%    \partial_{y_3y_3}\kappa=1\\
%    \partial_{y_2y_2}\kappa=1
%\end{array}.
%    \right.
%\end{equation}
From $\eqref{equation satisfied by u}$-$\eqref{eq:2-31}$ we deduce that $\bar{u}$ satisfies
\begin{align}
    \partial_{\bar{u}}\kappa\sum_{i,j=0}^3\tilde{a}_{ij}\partial_{ij}\bar{u}=F(\bar{u},\md\bar{u}),\label{eq:4-2}
\end{align}
where
\begin{align}
    F(\bar{u},\md\bar{u})&=\partial_{\bar{u}\bar{u}}\kappa\tilde{a}_{ij}\partial_i\bar{u}\partial_j\bar{u}+2\sum_{j=0}^3\kappa_{\bar{u}y_2}\tilde{a}_{2j}\partial_j\bar{u}+2\sum_{j=0}^3\kappa_{\bar{u}y_3}\tilde{a}_{3j}\partial_j\bar{u}\nonumber\\
    &\quad+\tilde{a}_{22}\partial_{y_2y_2}\kappa+2\tilde{a}_{23}\partial_{y_2y_3}\kappa+\tilde{a}_{33}\partial_{y_3y_3}\kappa-(\tilde{a}_2\partial_{y_2}u+\tilde{a}_{3}\partial_{y_3}u)\nonumber\\
    &\quad-(a_{12}\partial_{x_1x_2}p(\partial_{y_1}u)^3+(\partial_{y_1}u)^3\sum_{i,j=0}^3\limits a_{ij}\partial_{x_ix_j}\Phi^-),
\end{align}
where $\partial_{y_i}u$ can be replaced by $\partial_{\bar{u}}\kappa\partial_{y_i}\bar{u}+\partial_{y_2}\kappa\delta_{i2}+\partial_{y_3}\kappa\delta_{i3}$.

The initial conditions for $\bar{u}$ now become
\begin{align}
    \bar{u}|_{y_0=0}&=u_0+y_3N(u_0,x_2(u_0,y_2,y_3)),\label{eq:4-4}\\
    \partial_{y_0}\bar{u}|_{y_0=0}&=u_1\cdot(1+y_3\partial_{x_1}\left(\frac{\partial_{x_1}N}{1+y_3\partial_{x_2}N}\right)(u_0,x_2(u_0,y_2,y_3)).\label{eq:4-5}
\end{align}
The boundary conditions for $\bar{u}$ are
\begin{align}
    \partial_{y_3}\bar{u}&=0\quad\mbox{on}\quad \omega^r_T,\\
    G(\kappa(\bar{u},y_2,y_3),\md u)&=0\quad\mbox{on}\quad\omega^\ell_T,\label{eq:4-7}
\end{align}
where $\partial_{y_i}u$ should be replaced by $\partial_{\bar{u}}\kappa\partial_{y_i}\bar{u}+\partial_{y_2}\kappa\delta_{i2}+\partial_{y_3}\kappa\delta_{i3}$.

Let $\bar{u}_j(\mby)\defs\partial_{y_0}^j \bar{u}(y_0,\mby)|_{y_0=0}$, which can be derived by differentiating \eqref{eq:4-2} with respect to $y_0$. Obviously, $\bar{u}_0$ and $\bar{u}_1$ are give by \eqref{eq:4-4} and \eqref{eq:4-5} respectively. Let
\[
\psi(y_0,\mby)\defs \bar{u}_0+\bar{u}_1y_0+\frac{\bar{u}_2}{2!}y_0^2+\cdots+\frac{\bar{u}_{s_0}}{s_0!}y_0^{s_0}.
\]

We introduce a new unknown $\tilde{u}\defs \bar{u}-\psi$ and define $\tilde{u}_b\defs u_b-\psi$. Then $\tilde{u}$ satisfies
\begin{equation}\label{4.9}
\left\{\!\!
\begin{array}{ll}
     \partial_{\bar{u}}\kappa\sum_{i,j=0}^3\limits\tilde{a}_{ij}\partial_{ij}\tilde{u}=F(\tilde{u}+\psi,\md(\tilde{u}+\psi))-\partial_{\bar{u}}\kappa\sum_{i,j=0}^3\limits\tilde{a}_{ij}\partial_{ij}\psi\quad&\mbox{in}\quad \Omega_T,\\
     G(\kappa(\tilde{u}+\psi,y_2,y_3),\md (\kappa(\tilde{u}+\psi,y_2,y_3)))=0\quad&\mbox{on} \quad\omega^\ell_T,\\
     \partial_{y_3}\tilde{u}=0\quad&\mbox{on} \quad\omega^r_T,\\
     (\tilde{u},\partial_{y_0}\tilde{u})|_{y_0=0}=(0,0)\quad&\mbox{on} \quad\Gamma_{in}.
\end{array}
\right.
\end{equation}
If we can solve this problem for $\tilde{u}$, then clearly $\tilde{u}+\psi$ is the desired solution to the non-linear problem \eqref{NLP}.
\subsection{Proof of theorem \ref{wellposedness of the NLP}}\label{sec:5}

In this section, we introduce a iterative scheme to deduce the existence of smooth solution to the non-linear problem \eqref{4.9}.
 Let $\tilde{u}_0\defs 0$ and $\tilde{u}_{m+1}$ ($m\ge 0$) is defined as the solution to the following initial boundary value problem
\begin{equation}\label{eq:5.1}
\left\{\!\!
\begin{array}{ll}
     \partial_{\bar{u}}\kappa\sum_{i,j=0}^3\limits\tilde{a}^m_{ij}\partial_{ij}\tilde{u}_{m+1}=
F_m-\partial_{\bar{u}}\kappa\tilde{a}^{m}_{ij}\partial_{ij}\psi
 \quad&\mbox{in}\quad \Omega_T,\\
     \mathcal{B}\tilde{u}_{m+1}=\mathcal{B}\tilde{u}_{m}-G_m\quad&\mbox{on} \quad\omega^\ell_T,\\
     \partial_{y_3}\tilde{u}_{m+1}=0\quad&\mbox{on} \quad\omega^r_T,\\
     (\tilde{u}_{m+1},\partial_{y_0}\tilde{u}_{m+1})|_{y_0=0}=(0,0)\quad&\mbox{on}\quad\Gamma_{in},
\end{array}
\right.
\end{equation}
where $\tilde{a}_{ij}^m=\tilde{a}_{ij}|_{u=\kappa(\tilde{u}_m+\psi,y_2,y_3)}$, $F_m\defs F(\tilde{u}_m+\psi,\md(\tilde{u}_m+\psi))$, $G_m\defs G|_{u=\kappa(\tilde{u}_m+\psi,y_2,y_3)}$ and
\[
\mathcal{B}=\partial_{\bar{u}}\kappa\sum_{i=0}^3\frac{\partial G}{\partial u_{y_i}}\bigg|_{u=u_b}\cdot\partial_{y_i}+\partial_{\bar{u}}\kappa\frac{\partial G}{\partial u}\bigg|_{u=u_b}
\]
and $G=G(u, \md u)$ is defined in \eqref{eq:2-34}.

Before proving the convergence of above iterative scheme, we have to verify hypothesis \ref{h1}-\ref{h4}. Actually we have following lemma:
\begin{lem}\label{lem 5.1}
$(\tilde{a}_{ij})_{0\le i,j\le 3}$ and $\mathcal{B}$ satisfy all assumptions \ref{h1}-\ref{h4}.
\end{lem}
\begin{proof}
    It is clear that $\tilde{a}_{ij}$ are smooth functions depending on $u$ and $\md u$. As a direct consequence of \eqref{eq:potential_bdry_wedge} and \eqref{eq:a03}, $\tilde{a}_{03}$ vanishes on $\Gamma_{w}$.
    In view of the slip boundary condition $\eqref{eq:potential_bdry_wedge}$ and \eqref{eq:a13}, it is clear that
\begin{align}
\tilde{a}_{13}|_{\scriptscriptstyle{\Gamma_w}}&=c^2\partial_{y_1}u(\partial_{x_1}\mcw(\partial_{x_1}\mcw\partial_{y_3}u-\partial_{x_1}p\partial_{y_2}u+1))\nonumber\\
&\quad+c^2\partial_{y_1}u(\partial_{x_2}\mcw(\partial_{x_3}p\partial_{x_2}\mcw\partial_{y_2}u+\partial_{x_2}\mcw\partial_{y_3}u-\partial_{y_2}u)+\partial_{x_3}p\partial_{y_2}u+\partial_{y_3}u)
\end{align}
But $\eqref{eq:potential_bdry_wedge}$ implies
$$
-\partial_{x_1}\mcw(\Phi^--\phi)_{x_1}-\partial_{x_2}\mcw(\Phi^--\phi)_{x_2}+(\Phi^--\phi)_{x_3}=0
$$
on $\Gamma_0$, so by the slip boundary condition of $\Phi^-$ and the expressions of $\phi_{x_i}$ ($i=1,2,3$) given by \eqref{Dphi in terms of Du}, it is equivalent to say that
\begin{align}\label{horizontal boundary condition for u}
    &\partial_{x_1}\mcw(\partial_{x_1}\mcw\partial_{y_3}u-\partial_{x_1}p\partial_{y_2}u+1)\nonumber\\
    &\qquad+\partial_{x_2}\mcw(\partial_{x_3}p\partial_{x_2}\mcw\partial_{y_2}u+\partial_{x_2}\mcw\partial_{y_3}u-\partial_{y_2}u)\nonumber\\
    &\qquad+\partial_{x_3}p\partial_{y_2}u+\partial_{y_3}u=0\quad\mbox{on}\quad \Gamma_{w}.
\end{align}
 Thus, independent of the choice of $p(\mathbf{x})$, one deduces that $\tilde{a}_{13}$ vanishes on $\Gamma_w$. Also by the slip boundary condition \eqref{eq:potential_bdry_wedge}, we deduce that
\begin{equation}
    \tilde{a}_{23}|_{\scriptscriptstyle{\Gamma_{w}}}=c^2(\partial_{y_1}u)^2(\partial_{x_1}\mcw\partial_{x_1}p+\partial_{x_2}\mcw(\partial_{x_2}p+1)-\partial_{x_3}p).
\end{equation}
So requiring $\tilde{a}_{23}=0$ on $\Gamma_w$ is equivalent to require
\begin{equation}\label{condition 1 for p}
    \partial_{x_1}\mcw\partial_{x_1}p+\partial_{x_2}\mcw(\partial_{x_2}p+1)-\partial_{x_3}p=0\quad \mbox{on} \quad\Gamma_w.
\end{equation}
With $G$ given in \eqref{eq:2-34}, by simple calculation, we have
\[
G=(\rho^+-\rho^-)(\phi_t+\nabla\phi\cdot\nabla\Phi^-)-\abs{\nabla\phi}^2\rho^+,
\]
where $\rho^{\pm}=((\gamma-1)(B_0-\Phi^{\pm}_t-\frac{1}{2}\abs{\nabla\Phi^{\pm}}^2)+1)^{\frac{1}{\gamma-1}}$ and $\Phi^+\defs\Phi$ is the velocity potential ahead of the shock front.
Replacing $\md \phi$ in $G$ by $\md u$ via \eqref{Dphi in terms of Du}, then differentiate $G $ with respect to $u$ and $u_{y_i}$, respectively, one can obtain the expressions of $b$ and $b_j$ $(j=0,1,2,3)$.

By simple calculation, one has
\begin{align}
    b_3&=\frac{\partial \rho^+}{\partial (\partial_{y_3}u)}(\phi_t+\nabla\phi\cdot\nabla\Phi^-)+(\rho^+-\rho^-)\frac{\partial}{\partial (\partial_{y_3}u)}(\phi_t+\nabla\phi\cdot\nabla\Phi^-)\nonumber\\
    &=\mathcal{E}_1\cdot(\phi_t+\nabla\phi\cdot\nabla\Phi^-)+(\rho^+-\rho^-)\cdot\mathcal{E}_2
\end{align}
where
\begin{align*}
    &\mathcal{E}_1=-\dfrac{(\partial_{x_3}p-\partial_{x_1}p\partial_{x_1}\mcw+\partial_{x_3}p|\partial_{x_2}\mcw|^2-\partial_{x_2}\mcw)\partial_{y_2}u}{(\rho^+)^{\gamma-2}(\partial_{y_1}u)^2}\\
    &\qquad+\frac{(1+|\partial_{x_1}\mcw|^2+|\partial_{x_2}\mcw|^2)\partial_{y_3}u+\partial_{x_1}\mcw}{(\rho^+)^{\gamma-2}(\partial_{y_1}u)^2}\\
    &\qquad-\dfrac{\partial_{x_1}\mcw\partial_{x_1}\Phi^-+\partial_{x_2}\mcw\partial_{x_2}\Phi^--\partial_{x_3}\Phi^-}{\partial_{y_1}u},\\
    &\mathcal{E}_2=\dfrac{\partial_{x_1}\mcw\partial_{x_1}\Phi^-+\partial_{x_2}\mcw\partial_{x_2}\Phi^--\partial_{x_3}\Phi^-}{\partial_{y_1}u}.
\end{align*}

Since $\Phi^-$ satisfies \eqref{eq:potential_bdry_wedge} on $\Gamma_{w}$ (equivalently on $\{y_3=0\}$) and $u$ satisfies \eqref{Neumann bdry condition of u on horizontal boundary} on $\{y_3=0\}$, in order to let $b_3|_{y_3=0}=0$, it suffices to require
\begin{equation}\label{condition 2 for p}
    \partial_{x_3}p(1+|\partial_{x_2}\mathcal{W}|^2)-\partial_{x_1}p\partial_{x_1}\mcw-\partial_{x_2}\mcw=0\quad\mbox{on}\quad \Gamma_w.
\end{equation}
It is easy to verify that \eqref{condition 1 for p} and \eqref{condition 2 for p} are satisfied,
if we let
\begin{equation}
  p(x)\defs\frac{\partial_{x_2}\mcw}{1+|\partial_{x_1}\mcw|^2+|\partial_{x_2}\mcw|^2}(x_3-\mcw).
\end{equation}
With such $p(\mbx)$, by calculating the Jacobian of $\mathscr{P}$, one can easily check that $\mathscr{P}$ is indeed invertible, when $u$ is close to $u_b$, $\partial_{x_2}\mathcal{W}$ is small, and $x_3$ is close to $\mathcal{W}(x_1,x_2)$ (this means $y_3$ is small in ($y_0,\mby$)-coordinate).
Here we do not need the exact expression of $b_0$, $b$ and $b_j$ $(j=0,1,2)$. At the background solution $(u,\mcw)=(u_b,0)$, one has
\begin{align}
    b_0\defs\frac{\partial{G}}{{\partial u_{y_0}}}=&(q_--q_+)\left(-\frac{\rho_+q_+}{c_+^2}(q_--q_+)-(\rho_+-\rho_-)\right)<0,\\
   b_1\defs\frac{\partial{G}}{{\partial (\partial_{y_1}u})}=&(q_--q_+)^2\left(-\frac{q_+^2\rho_+}{c_+^2}(q_--q_+)+\rho_+(q_--q_+)\right)>0,
\end{align}
and
\[ b_2\defs\frac{\partial{G}}{{\partial (\partial_{y_2}u})}=0,\quad
   b_3\defs\frac{\partial{G}}{{\partial (\partial_{y_3}u})}=0,\quad
   b\defs\frac{\partial{G}}{{\partial u}}=0.
\]
Moreover, with the choice of $p(\mbx)$, one can see that $\tilde{a}_2$ is zero at the background solution. Hence $\tilde{a}_2$ is close to zero near the background solution. This allows us to put the term $\tilde{a}_2\partial_{y_2}u$ to the right side in the coming iteration scheme, so that the coefficient before $\partial_{y_2}u$ be zero. Then above computations together with \eqref{eq:2-40}-\eqref{eq:2-43} implies \ref{h1} and \ref{h2} are fulfilled.
% &=-(\rho^+)^{2-\gamma}(\partial_{y_1}u)^{-2}(\partial_{x_1}\mcw+(\partial_{x_3}p-\partial_{x_1}p\partial_{x_1}\mcw+\partial_{x_3}p\partial_{x_2}\mcw^2-\partial_{x_2}\mcw)\partial_{y_2}u)\nonumber\\
%    &\quad-(\rho^+)^{2-\gamma}(\partial_{y_1}u)^{-2}(1+\partial_{x_1}\mcw^2+\partial_{x_2}\mcw^2)\partial_{y_3}u
 We still need to verify \ref{h4}. It is clear that $b_1$ is bounded away from zero when $u$ is sufficiently close to $u_b$. By simple calculation, we have
\begin{align}
   \frac{\tilde{a}_{11}}{b_1}b_0-\tilde{a}_{10}&=-\frac{q_+}{q_--q_+}+\frac{(c_+^2-q_+^2)(c_+^2(\rho_--\rho_+)+q_+(q_+-q_-)\rho_+)}{(q_--q_+)\rho_+q_+^2-c_+^2(-2q_+\rho_++q_-(\rho_-+\rho_+))}\nonumber\\
   &=-\frac{q_+}{q_--q_+}+\frac{(c_+^2-q_+^2)(c_+^2(\rho_--\rho_+)+q_+(q_+-q_-)\rho_+)}{(q_+^2-c_+^2)(q_--q_+)\rho_+}\nonumber\\
   &=-\frac{q_+}{q_--q_+}-\frac{c_+^2(\rho_--\rho_+)+q_+(q_+-q_-)\rho_+}{(q_--q_+)\rho_+}\nonumber\\
   &=\frac{1}{(q_--q_+)\rho_+}(-q_+\rho_+-c_+^2(\rho_--\rho_+)-q_+(q_+-q_-)\rho_+)\nonumber\\
   &>\frac{1}{(q_--q_+)\rho_+}(-\rho_+q_+-q_+^2(\rho_--\rho_+)-q_+(q_+-q_-)\rho_+)\nonumber\\
   &=\frac{q_+}{(q_--q_+)\rho_+}(q_-\rho_+-q_+\rho_--\rho_+)>0.
\end{align}
Moreover, at the background solution we have
\[
\sum_{i,j=0}^3\tilde{a}^{ij}(\tilde{a}_{11}\frac{b_i}{b_1}-\tilde{a}_{i1})(\tilde{a}_{11}\frac{b_j}{b_1}-\tilde{a}_{j1})=\frac{(q_--q_+)^2(c_+^2-q_+^2)}{c_+^2} (\frac{\tilde{a}_{11}}{b_1}b_0-\tilde{a}_{10})^2>0.
\]
So $\gamma_0$ exists and hence \ref{h4} is satisfied. For \ref{h3}, we can see from our proof of the convergence of the scheme that the solution $\tilde{u}+\psi$ is still close to $u_b$.
\end{proof}

For $s\in \mathbb{N} $, let
 \[
 \eta\sum_{|\alpha|\le s}\|e^{-\eta t}\md^\alpha u\|^2_{L^2(\Omega_T)}\defs\|u\|^2_{\mch^s_\eta(\Omega_T)}
 \]
and denote by $\|u\|_{s,\eta,T}$ the usual Sobolev norm $\|e^{-\eta t}u\|_{H^{s}(\Omega_T)}$. Furthermore, for simplification, one may use $\|u|_{y_1=0}\|_{s,\eta,T}$ to represent the usual Sobolev norm on the boundary $\{y_1=0\}$.
 \begin{lem}\label{lem:4.4}
For any smooth function $u$ and any $\eta\ge 1$, we have
\begin{equation}\label{5.27}
\|e^{-\eta t}u\|^2_{H^s(\Omega_T)} \lesssim \frac{1}{\eta}\|u\|^2_{\mch^s_\eta(\Omega_T)},
\end{equation}
provided that $\partial^j_{t}u=0,{\ }j=0,1,\cdots,s-1.$ Here $\|e^{-\eta t}u\|_{H^s(\Omega_T)}$ is the standard sobolev norm.
\end{lem}
\begin{proof}
Let
$$
A(y_0)=\int_0^{y_0}\int_{\Omega}e^{-2\eta t}|u|^2\mathrm{d}y\mathrm{d}t
$$
Then we have
\begin{align*}
A(y_0)&=-\frac{1}{2\eta}\int_0^{y_0}\!\!\int_{\Omega}(e^{-2\eta t})_t|u|^2\mathrm{d}y\mathrm{d}t \\
&=-\frac{1}{2\eta}\int_0^{y_0}\!\!\int_{\Omega}(e^{-2\eta t}|u|^2)_t-2e^{-2\eta t}u\partial_tu\mathrm{d}y\mathrm{d}t\\
&=-\frac{1}{2\eta}\int_{\Omega}e^{-2\eta y_0}|u(y_0,\cdot)|^2\mathrm{d}y+\frac{1}{2\eta}\int_0^{y_0}\int_{\Omega}2e^{-2\eta t}u\partial_tu\mathrm{d}y\mathrm{d}t\\
&\le -\frac{1}{2\eta}\int_{\Omega}e^{-2\eta y_0}|u(y_0,\cdot)|^2\mathrm{d}y+\frac{1}{2\eta}\int_0^{y_0}\int_{\Omega}e^{-2\eta t}(\eta |u|^2+\frac{1}{\eta}|\partial_tu|^2)\mathrm{d}y\mathrm{d}t\\
&= -\frac{1}{2\eta}\int_{\Omega}e^{-2\eta y_0}|u(y_0,\cdot)|^2\mathrm{d}y+\frac{1}{2}A(y_0)+\frac{1}{2\eta^2}\int_0^{y_0}\int_{\Omega}e^{-2\eta t}|\partial_tu|^2\mathrm{d}y\mathrm{d}t.
\end{align*}
Here for the third identity, we have used the assumption that $u|_{t=0}=0$.
This implies
\begin{equation}\label{estimate of interior termx}
\eta A(y_0)+\int_{\Omega}e^{-2\eta y_0}|u(y_0,\cdot)|^2\mathrm{d}y\le \frac{1}{\eta}\int_0^{y_0}\int_{\Omega}e^{-2\eta t}|\partial_tu|^2\mathrm{d}y\mathrm{d}t.
\end{equation}
In particular,
\begin{equation}\label{estimate of interior term}
\eta^2 A(y_0)\le \int_0^{y_0}\int_{\Omega}e^{-2\eta t}|\partial_tu|^2\mathrm{d}y\mathrm{d}t\le\int_0^{y_0}\int_{\Omega}e^{-2\eta t}|\mathrm{D}u|^2\mathrm{d}y\mathrm{d}t.
\end{equation}

It follows from $\eqref{estimate of interior termx}$ that
\begin{equation}
\begin{split}
&\int_0^{y_0}\int_{\Omega}e^{-2\eta t}|\mathrm{D}u|^2\mathrm{d}y\mathrm{d}t+
\eta^2\int_0^{y_0}\int_{\Omega}e^{-2\eta t}|u|^2\mathrm{d}y\mathrm{d}t+
\int_0^{y_0}\int_{\Omega}e^{-2\eta t}|u|^2\mathrm{d}y\mathrm{d}t\\
\lesssim& \int_0^{y_0}e^{-2\eta t}(\|u\|^2_{L^2(\Omega)}+\|\mathrm{D}u\|^2_{L^2(\Omega)})\mathrm{d}t.
\end{split}
\end{equation}

%The left hand side is the first order usual sobolev norm of $e^{-\eta t}u$ while the right hand side is  $\|u\|^2_{1,\eta}$.
%So this
This implies that  $\|u\|^2_{1,\eta,T}\lesssim \frac{1}{\eta}\|u\|^2_{\mathcal{H}^1_\eta(\Omega_T)}$.

Now for $k\in\mathbb{N}$, assume
\begin{equation}\label{4.31}
\|u\|^2_{k,\eta,T}\lesssim \frac{1}{\eta}\|u\|^2_{\mathcal{H}^k_\eta(\Omega_T)}.
\end{equation}

We are going to show
\begin{equation}\label{4.32}
\|u\|^2_{k+1,\eta,T}\lesssim \frac{1}{\eta}\|u\|^2_{\mathcal{H}^{k+1}_\eta(\Omega_T)}.
\end{equation}

%By using $\eqref{estimate of interior term}$ repeatedly(a $\eta^{-2}$ comes out on the right hand-side each time), we find that
Repeating the process for estimate \eqref{estimate of interior term} above $m$ times where $|u|^2$ in $A(y_0)$ is replaced by $|\md^nu|^2$, we have
\begin{equation}\label{recursive formula}
\int_0^{y_0}e^{-2\eta t}\|\mathrm{D}^{n}u\|^2_{L^2(\Omega)}\mathrm{d}t\le \eta^{-2m}\int_0^{y_0}e^{-2\eta t}\|\mathrm{D}^{m+n}u\|^2_{L^2(\Omega)}\mathrm{d}t
\end{equation}
provided that $\partial^l_tu|_{t=0}=0,\, l=0,1,2,\cdots, m+n-1.$

Note that
\begin{equation}\label{induction1}
\|u\|^2_{k+1,\eta,T}= \|u\|^2_{k,\eta,T}+ \sum_{|\alpha|=k+1}\|\mathrm{D}^{\alpha}(e^{-\eta t}u)\|_{L^2(\Omega_T)}^2
\end{equation}
and
\begin{align}\label{induction2}
%\begin{split}
\sum_{|\alpha|= k+1}\|\mathrm{D}^{\alpha}(e^{-\eta t}u)\|^2_{L^2(\Omega_T)}&=\sum_{l_1+l_2=k+1}\|(-\eta)^{l_1}e^{-\eta t}\mathrm{D}^{l_2}u\|^2_{L^2(\Omega_T)}\nonumber\\
&=\sum_{l_1+l_2=k+1}(\eta)^{2l_1}\|e^{-\eta t}\mathrm{D}^{l_2}u\|^2_{L^2(\Omega_T)}.
%\end{split}
\end{align}
So by $\eqref{recursive formula}$, we have
\begin{align}\label{induction3}
%\begin{split}
\sum_{l_1+l_2=k+1}(\eta)^{2l_1}\|e^{-\eta t}\mathrm{D}^{l_2}u\|^2_{L^2(\Omega_T)}&\le \sum_{l_1+l_2=k+1}\|e^{-\eta t}\mathrm{D}^{l_1+l_2}u\|^2_{L^2(\Omega_T)}\nonumber\\
&=\sum_{|\alpha|=k+1}\|e^{-\eta t}\mathrm{D}^{\alpha}u\|^2_{L^2(\Omega_T)}.
%\end{split}
\end{align}
From \eqref{4.31}, $\eqref{induction1}$--$\eqref{induction3}$, we obtain \eqref{4.32}.
Therefore, we derive the estimate \eqref{5.27} for any $s\in \mathbb{N}$ by the induction method.
%On the other hand, it is easy to see that $\|u\|_{\mathcal{H}^s_\eta(\Omega_T)}\lesssim\|e^{-\eta t}u\|_{H^s(\Omega_T)}$. It completes the proof of this lemma.
\end{proof}

\begin{lem}[Boundedness in the norm of high regularity]\label{lem-5.3}
 Under the assumption of theorem \ref{wellposedness of the NLP}, there exists a large $\eta_*\ge 1$ and a small $T_*>0$ and small $\epsilon_0>0$, such that for all $\eta\ge \eta_*$ and $T\le T_*$, the following estimate
\begin{equation}\label{high norm boundedness}
\|\tilde{u}_m\|^2_{s,\eta,T}+\sum_{|\alpha|\le s}\sup_{0\le t\le T}\|\md ^\alpha \tilde{u}_m(t,\cdot)\|^2_{L^2(\Omega)}+\sum_{|\alpha|\le s}\|e^{-\eta t} \md^\alpha\tilde{u}_m\|_{L^2(\omega^\ell_T)}^2\le \epsilon_0^2
 \end{equation}
 holds for all $m\ge 0$.
\end{lem}

\begin{proof}
 We prove this lemma by induction. Suppose \eqref{high norm boundedness} holds for all $m\le n$, we proceed to show it also holds true for $m=n+1$.
 In view of \eqref{eq:5.1}, in order to apply theorem \ref{well-posedness and energy estimate of the LP} to $\tilde{u}_{n+1}$, we need to estimate the source terms. By the definition of $F_n$, we know that
\begin{equation}\label{4-32}
 \|F_n\|^2_{s-1,\eta,T}\lesssim C'\epsilon_0^2\|\tilde{u}_n\|^2_{s,\eta,T}.
 \end{equation}
 Similarly, we have
 \begin{align}
    \|\tilde{a}_{ij}^n\partial_{ij}\psi\|^2_{s-1,\eta, T}\le C''\|\tilde{u}_n\|^2_{s,\eta,T}.\label{4-33}
\end{align}
For the boundary term, noticing that $G(u_b,\md u_b)=0$, we have
\begin{align}
    \mathcal{B}\tilde{u}_n-G_n&=\mathcal{B}\tilde{u}_n-(G(\tilde{u}_n+\psi,\mathrm{D}(\tilde{u}_n+\psi)-G(u_b,\mathrm{D}u_b))\nonumber\\
    &=\mathcal{B}\tilde{u}_n-\mathcal{B}(\tilde{u}_n+\psi-u_b)+A_n^\top\mathrm{D}^2 G|_{u=u_b+\theta (\tilde{u}_n+\psi-u_b)}A_n\nonumber\\
    &=\mathcal{B}(u_b-\psi)+A_n^\top\mathrm{D}^2 G|_{u=u_b+\theta (\tilde{u}_n+\psi-u_b)}A_n,
\end{align}
where $A_n=(\tilde{u}_n+\psi-u_b,\md(\tilde{u}_n+\psi-u_b))$.
Hence we deduce that
\begin{align}
    \sum_{|\alpha|\le s-1}\!\!\|e^{-\eta t}&\md^\alpha(\mathcal{B}\tilde{u}_n-G_n)\|^2_{L^2(\omega^\ell_T)}\nonumber\\
    &\lesssim \sum_{|\alpha|\le s}\|e^{-\eta t}\md^\alpha(\psi-u_b)\|^2_{L^2(\omega^\ell_T)}+\|e^{-\eta t}\tilde{u}_n|_{y_1=0}\|^4_{H^s(\omega^\ell_T)}.\label{4-35}
\end{align}

 By theorem \ref{well-posedness and energy estimate of the LP} and lemma \ref{lem:4.4} and in view of \eqref{4-32}, \eqref{4-33}, and \eqref{4-35}, we deduce that

 \begin{align}
    \|\tilde{u}_{n+1}\|^2_{s,\eta,T}&+\sup_{0\le t\le T}\|\md^\alpha \tilde{u}_{n+1}(t,\cdot)\|^2_{L^2(\Omega)}+\|e^{-\eta t}\tilde{u}_{n+1}|_{y_1=0}\|^2_{s,\eta,T}\nonumber\\
    \le&\frac{C}{\eta^2}e^{4\eta T}\|\tilde{u}_n+\psi\|^2_{s,\eta,T}\cdot(\|F_n\|^2_{3,\eta,T}+\|\tilde{a}^n_{ij}\partial_{ij}\psi\|^2_{3,\eta,T})\nonumber\\
    &+\frac{C}{\eta}e^{2\eta T}\|F_n\|_{s-1,\eta,T}^2+e^{2\eta T}\|(\mathcal{B}\tilde{u}_n-G_n)|_{y_1=0}\|^2_{s-1,\eta,T}\nonumber\\
    \le&\frac{C}{\eta}e^{4\eta T}(\epsilon_0^2+\|\psi\|^2_{s,\eta,T})\cdot(C^\prime \epsilon_0^4+C^{\prime\prime}\epsilon_0^2)\nonumber\\
    &+\frac{C}{\eta}e^{2\eta T}\epsilon_0^2+Ce^{2\eta T}(\epsilon^2+\epsilon_0^4).
\end{align}
Now let $\eta_*$ be properly large such that $\frac{C}{\eta_*}(C^\prime+C^{\prime\prime}+1)<\frac{1}{16}$. Then let $\epsilon_0$ be small such that $\epsilon_0^2<\min(\frac12,\sqrt{\frac{1}{8C}})$. Finally let $T_*$ and the $\epsilon$ in theorem 2.1 be properly small such that $e^{4\eta_* T_*}<2$ and $ \|\psi\|^2_{s,\eta_* ,T_*}<\epsilon_0^2$ and $\epsilon\le \min(\sqrt{\frac{1}{16C}}\epsilon_0,\epsilon_0)$. We obtain
\[
\|\tilde{u}_{n+1}\|^2_{s,\eta_*,T_*}+\sup_{0\le t\le T_*}\|\md^\alpha \tilde{u}_{n+1}(t,\cdot)\|^2_{L^2(\Omega)}+\|\tilde{u}_{n+1}|_{y_1=0}\|^2_{s,\eta_*,T_*}<\frac12\epsilon_0^2+2\times\frac18\epsilon_0^2+\frac{1}{4}\epsilon^2_0=\epsilon_0^2.
\]
This implies $\eqref{high norm boundedness}$ also holds for $m=n+1$.  It is clear that \eqref{high norm boundedness} holds for $m=0$. This completes the proof of this lemma.
 \end{proof}

 It is easy to check that $v_{m+1}\defs\tilde{u}_{m+1}-\tilde{u}_m$ satisfies following initial boundary value problem
 \begin{equation}
\left\{\!\!
\begin{array}{ll}
     \partial_{\bar{u}}\kappa\sum_{i,j=0}^3\limits\tilde{a}^m_{ij}\partial_{ij}v_{m+1}=F_{m}-F_{m-1}-\partial_{\bar{u}}\kappa(\tilde{a}_{ij}^m-\tilde{a}_{ij}^{m-1})\partial_{ij}(\tilde{u}_{m}+\psi)\quad&\mbox{in}\quad \Omega_T,\\
     \mathcal{B}v_{m+1}=\mathcal{B}v_{m}-(G_{m}-G_{m-1})\quad&\mbox{on} \quad\omega^\ell_T,\\
     \partial_{y_3}v_{m+1}=0\quad&\mbox{on} \quad\omega^r_T,\\
     (v_{m+1},\partial_{y_0}v_{m+1})|_{y_0=0}=(0,0)\quad&\mbox{on}\quad \Gamma_{in}.
\end{array}
\right.
\end{equation}
For the sequence $\{v_m\}_{m=1}^\infty$, we have following lemma:
\begin{lem}[Contraction in the norm of low regularity]\label{lem-5.4}
   Under the assumption of theorem \ref{wellposedness of the NLP} and suppose the $\epsilon$ in \ref{wellposedness of the NLP} is small enough, then there exist two constants $\eta_{**}\ge 1$ and $T_{**}>0$ such that
   \begin{align}
    &\|v_{m+1}\|^2_{1,\eta_{**},T_{**}}+\sup_{0\le t\le T_{**}}\|\md^\alpha v_{m+1}(t,\cdot)\|^2_{L^2(\Omega)}+\|v_{m+1}|_{y_1=0}\|^2_{1,\eta_{**}, T_{**}}\nonumber\\
    &\quad\le \sigma_0\cdot(\|v_{m}\|^2_{1,\eta_{**},T_*}+\sup_{0\le t\le T_{**}}\|\md^\alpha v_{m}(t,\cdot)\|^2_{L^2(\Omega)}+\|v_{m}|_{y_1=0}\|^2_{1,\eta_{**},T_{**}})\label{low norm contraction}
\end{align}
   hold for all $m\ge 0$, where $0<\sigma_0<1$ is a constant independent of $m$.
\end{lem}
\begin{proof}

In order to apply theorem \ref{well-posedness and energy estimate of the LP}, we need to estimate the source terms.
In fact, we have
\begin{align}
    F_m-F_{m-1}&=\partial_{\bar{u}}\kappa(\int^1_0F(u,\md u)|_{u=\tilde{u}_{m-1}+\theta v_m+\psi}\mathrm{d}\theta)\cdot v_m\nonumber\\
    &\quad+\sum_{i=0}^3\partial_{\bar{u}}\kappa(\int^1_0\partial_{u_i}F(u,\md u)|_{u=\tilde{u}_{m-1}+\theta v_m+\psi}\mathrm{d}\theta)\cdot\partial_iv_m.
\end{align}
Hence we deduce that
\begin{align}
   \| F_m-F_{m-1}\|_{0,\eta,T}\lesssim \epsilon_0\|v_{m}\|_{1,\eta,T}.
\end{align}
Similarly, one has
\begin{align}
    \|\partial_{\bar{u}}\kappa(\tilde{a}_{ij}^m-\tilde{a}_{ij}^{m-1})\partial_{ij}(\tilde{u}_{m}+\psi)\|_{0,\eta,T}\lesssim \|v_m\|_{1,\eta,T}.
\end{align}
%For the boundary term, we have
%\begin{align}
%    G_m-G_{m-1}&=(G_m-G(u_b,\md u_b)-(G_{m-1}-G(u_b,\md u_b))\nonumber\\
%    &=\mathcal{B}(\tilde{u}_m+\psi-u_b)+A_m^\top \md^2G|_{u=u_b+\theta_1(\tilde{u}_m+\psi-u_b))}A_m\nonumber\\
%    &\quad-\mathcal{B}(\tilde{u}_{m-1}+\psi-u_b)-A_{m-1}^\top \md^2G|_{u=u_b+\theta_2(\tilde{u}_{m-1}+\psi-u_b)}A_{m-1}\nonumber\\
%    &=\mathcal{B}v_m-(v_m,\md v_m)^\top\md^2G|_{u=u_b+\theta_1(\tilde{u}_m+\psi-u_b)}A_m\nonumber\\
%    &\quad+A_{m-1}^\top(\md^2G|_{u=u_b+\theta_1(\tilde{u}_m}+\psi-u_b)-\md^2G|_{u=u_b+\theta_2}(\tilde{u}_{m-1}+\psi-u_b))A_{m}\nonumber\\
%    &\quad+A^\top_{m-1}\md^2G|_{u=u_b+\theta_2}(\tilde{u}_{m-1}+\psi-u_b)(v_{m-1},\md v_{m-1}),
%\end{align}
%where $A_j^\top=(\tilde{u}_j+\psi-u_b,\md\tilde{u}_j+\psi-u_b)$ and $\md^2G$ is the Hessian matrix of $G$ with respect to $(u,\md u)$.
For the boundary term, we have
\begin{align}
    G_m-G_{m-1}&=\partial_{u}G|_{\tilde{u}=\tilde{u}_{m-1}}v_m+\sum_{i=0}^3\partial_{u_{y_i}}G|_{\tilde{u}=\tilde{u}_{m-1}}\partial_{y_i}v_m\nonumber\\
    &\quad+(v_m,\md v_m)^\top \md^2G|_{\tilde{u}=\tilde{u}_{m-1}+\theta v_m}(v_m,\md v_m),\label{eq4.42}
\end{align}
where $\theta\in (0,1)$ and $\md^2G$ is the Hessian matrix of $G$ with respect to $(u,\md u)$.
Hence one deduces that 
\begin{align}
    &\mathcal{B}v_m-(G_m-G_{m-1})\nonumber\\
    &\quad=(\partial_{u}G|_{\tilde{u}=\tilde{u}_b}-\partial_u G|_{\tilde{u}=\tilde{u}_{m-1}})v_m+\sum_{i=0}^3(\partial_{u_{y_i}}G|_{\tilde{u}=\tilde{u}_b}-\partial_{u_{y_i}}G|_{\tilde{u}=\tilde{u}_{m-1}})\partial_{y_i}v_m\nonumber\\
    &\qquad+(v_m,\md v_m)^\top \md^2G|_{\tilde{u}=\tilde{u}_{m-1}+\theta v_m}(v_m,\md v_m).\label{eq4.43}
    \end{align}
    By Taylor theorem, it is clear that
    \begin{align}
    &(\partial_{u}G|_{\tilde{u}=\tilde{u}_b}-\partial_u G|_{\tilde{u}=\tilde{u}_{m-1}})v_m\nonumber\\
    &\quad=-\left(\partial^2_{u}G|_{\tilde{u}=\tilde{u}_b}(\tilde{u}_{m-1}-\tilde{u}_b)+\sum_{i=0}^3(\partial_{uu_{y_i}}G)|_{\tilde{u}=\tilde{u}_b}\partial_{y_i}(\tilde{u}_{m-1}-\tilde{u}_b)\right)v_m\nonumber\\
    &\quad\qquad- \left(\mathbf{X}_{m-1}^\top(\md^2\partial_{u}G)|_{\tilde{u}=\tilde{u}_b+\theta_1 v_m}\mathbf{X}_{m-1}\right) v_m.
\end{align}
where $\theta_1\in(0,1)$, $\mathbf{X}_{m-1}= (\tilde{u}_{m-1}-\tilde{u}_b,\md(\tilde{u}_{m-1}-\tilde{u}_b))$, and $\md^2\partial_{u}G$ is the Hessian matrix of $\partial_{u}G$ with respect to $(u,\md u)$.

Hence by lemma \ref{lem-5.3}, one has
\begin{align}
    &\|e^{-\eta t}(\partial_{u}G|_{\tilde{u}=\tilde{u}_b}-\partial_u G|_{\tilde{u}=\tilde{u}_{m-1}})v_m\|_{L^2(\omega^\ell_T)}\lesssim \epsilon_0 \|e^{-\eta t}v_m\|_{L^2(\omega^\ell_T)}.\label{eq4.45}
\end{align}
Similarly one deduces that 
\begin{align}
    \|e^{-\eta t}(\partial_{u_{y_i}}G|_{\tilde{u}=\tilde{u}_b}-\partial_{u_{y_i}}G|_{\tilde{u}=\tilde{u}_{m-1}})\partial_{y_i}v_m\|_{L^2(\omega^\ell_T)}\lesssim \epsilon_0 \|e^{-\eta t}v_m\|_{H^1(\omega^\ell_T)},\quad 0\le i\le 3.\label{eq4.46}
\end{align}
It is easy to see that
\begin{align}
    \|e^{-\eta t}(v_m,\md v_m)^\top \md^2G|_{\tilde{u}=\tilde{u}_{m-1}+\theta v_m}(v_m,\md v_m)\|_{L^2(\omega^\ell_T)}\lesssim \epsilon_0 \|e^{-\eta t}v_m\|_{H^1(\omega^\ell_T)}.\label{eq4.47}
\end{align}
By \eqref{eq4.43} and \eqref{eq4.45}-\eqref{eq4.47}, we have
\begin{align}
\|e^{-\eta t}\mathcal{B}v_{m+1}\|_{L^{2}(\omega^\ell_T)}&\lesssim \epsilon_0\|e^{-\eta t}v_m\|_{H^{1}(\omega_T^\ell)}.\label{eq4.48}
\end{align}
%Utilizing the Gauss theorem and lemma \ref{lem-5.3}, one obtains
%\begin{align}
%    \|e^{-\eta t}v_m\|_{L^2(\omega_T^\ell)}^2&=-\int_{\Omega_T}\partial_{y_1}\left(e^{-2\eta t}|v_m|^2\right)\textrm{d}\mby \textrm{d}y_0\nonumber\\
%    & \le 2\int_{\Omega_T}e^{-2\eta t}|v_m|\cdot|\partial_{y_1}v_m|\textrm{d}\mby \textrm{d}y_0\nonumber\\
%    &\le 2 \|e^{-\eta t}v_m\|_{L^2(\Omega_T)} \cdot \|e^{-\eta t}\partial_{y_1}v_m\|_{L^2(\Omega_T)}\nonumber\\
%    &\le 2\epsilon_0\|e^{-\eta t} v_m\|_{L^2(\Omega_T)}. \label{eq4.49}
%\end{align}\
%Similarly, one can also deduce that 
%\begin{align}
%    \|e^{-\eta t}\partial_{y_i}v_m\|_{L^2(\omega_T^\ell)}^2 \lesssim \epsilon_0 \|e^{-\eta t}\partial_{y_i} v_m\|_{L^2(\Omega_T)}\label{eq4.50}
%\end{align}
%Combining \eqref{eq4.48}, \eqref{eq4.49}, and \eqref{eq4.50}, we have
%\begin{align}
%    \|e^{-\eta t}\mathcal{B}v_{m+1}\|_{L^{2}(\omega^\ell_T)}&\lesssim \epsilon_0^3 \|v_m\|_{1,\eta,T}.
%\end{align}
Then by 
%Clearly $\|e^{-\eta t}\mcb v_{m}\|_{H^{s-1}(\omega^\ell_T)}\lesssim \|v_{m}\|_{H^s(\omega^\ell_T)}\lesssim \epsilon$.
 theorem \ref{well-posedness and energy estimate of the LP}, one has 
\begin{align}
    &\| v_{m+1}\|^2_{1,\eta, T}+\sum_{|\alpha|\le 1}\sup_{0\le t\le T}\|\md^\alpha v_{m+1}(t,\cdot)\|^2_{L^2(\Omega)}+\|v_{m+1}|_{y_1=0}\|^2_{1,\eta,T}\nonumber\\
    &\quad\le C\left(\frac{1}{\eta}+\epsilon_0^2\right)(\|v_{m}\|^2_{1,\eta,T}+\|v_m|_{y_1=0}\|^2_{1,\eta,T}).
\end{align}
Above inequality holds for $\eta>\eta_*$ and $T<T_*$. From the proof of lemma \ref{lem-5.3}, we can further require $\epsilon_0$ small such that $C\epsilon_0<\frac12$. Then one selects $\eta_{**}>\eta_*\ge 1$ such that $\frac{C}{\eta_{**}}\le \frac12\epsilon_0$, then for properly small $T_{**}$ ($T_{**}<T_{*}$), we have
\begin{align}
    &\| v_{m+1}\|^2_{1,\eta_{**}, T_{**}}+\sum_{|\alpha|\le 1}\sup_{0\le t\le T_{**}}\|\md^\alpha v_{m+1}(t,\cdot)\|^2_{L^2(\Omega)}+\|e^{-\eta t} v_{m+1}|_{y_1=0}\|^2_{1,\eta_{**},T_{**}}\nonumber\\
    &\qquad\le \epsilon_0(\|v_m\|^2_{1,\eta_{**}, T_{**}}+\sum_{|\alpha|\le 1}\sup_{0\le t\le T_{**}}\|\md^\alpha v_{m}(t,\cdot)\|^2_{L^2(\Omega)}+\|v_m|_{y_1=0}\|^2_{1,\eta_{**},T_{**}}).\label{eq:5-31}
\end{align}
Since $\epsilon_0<1$, we finish the proof of this lemma by letting $
\sigma_0=\epsilon_0$.
\end{proof}

\textbf{Proof of theorem} \ref{wellposedness of the NLP}.
Armed with lemma \ref{lem-5.3} and lemma \ref{lem-5.4}, we are able to prove theorem \ref{wellposedness of the NLP}.
In fact, lemma \ref{lem-5.4} implies that $\{\tilde{u}_m\}_{m=1}^\infty$ is a Cauchy sequence in the norm of low regularity. Hence it converges strongly such that $\tilde{u}_m$ converges to some function  $\tilde{u}$, i.e.,
\begin{align}
&(\|\tilde{u}_m-\tilde{u}\|_{1,\eta_{**},T_{**}}+\sum_{|\alpha|\le 1}\sup_{0\le t\le T_{**}}\|\md^\alpha (\tilde{u}_m-\tilde{u})(t,\cdot)\|^2_{L^2(\Omega)}\nonumber\\
&\qquad+\|(\tilde{u}_m-\tilde{u})|_{y_1=0}\|_{1,\eta_{**},T_{**}})\longrightarrow 0\mbox{\ as } m \mbox{\  goes\ to\ infinity}.\label{4.45}
\end{align}

Limit \eqref{4.45} also means the coefficients in the equation and boundary conditions in \eqref{eq:5.1}, $\tilde{a}^{m}_{ij}$, $F_m$, $G_m$ and $\mathcal{B}\tilde{u}_m$,  converge to the corresponding quantities with $\tilde{u}_{m}$ being replaced by $\tilde{u}$.

On the other hand, it follows from Lemma \ref{lem-5.3} that $\tilde{u}_m$ converges to $\tilde{u}$ weakly in the norm of high regularity such that $\tilde{u}$ satisfies estimate \eqref{high norm boundedness}. 

%Then by the standard interpolation inequality:
%\[
%\|\tilde{u}_m-\tilde{u}\|_{s-\lambda,\eta_{**},T_{**}}\le C \|\tilde{u}_m-\tilde{u}\|_{1,\eta_{**},T_{**}}^{\frac{s-\lambda-1}{s-1}}\cdot\|\tilde{u}_m-\tilde{u}\|_{s,\eta_{**},T_{**}}^{\frac{\lambda}{s-1}},
%\]
%lemma \ref{lem-5.3}, lemma \ref{lem-5.4}, and \eqref{4.45}
%one can deduce that $\tilde{u}_m$ actually converges to $\tilde{u}$ in $H^{s-\lambda}(\Omega_{T_{**}})$ for any $0<\lambda\le s-1$. This means that $\tilde{u}$ belong to $H^{s-\lambda}(\Omega_{T_{**}})$. Hence for $0<\lambda<s-\frac{7}{2}$ such that $s-\lambda>\frac{7}{2}$, by the sobolev embedding theorem, 
Hence, by passing the limit in \eqref{eq:5.1}, it is easy to see that $\tilde{u}+\psi$ is the the smooth solution of the non-linear problem \eqref{NLP} with estimate \eqref{high norm boundedness}. By \eqref{high norm boundedness} and the assumption of theorem \ref{wellposedness of the NLP} one has
\begin{align}
    \|\tilde{u}+\psi-u_b\|_{s,\eta_{**},T_{**}}\le \|\tilde{u}\|_{s,\eta_{**},T_{**}}+\|\psi-u_b\|_{s,\eta_{**},T_{**}}\le C\epsilon_0.
\end{align}

This completes the proof of theorem \ref{wellposedness of the NLP}.
%By \eqref{eq:5-31} and the interpolation inequality, we can also deduce that there exists a positive constant $C$ such that $\|\tilde{u}-\tilde{u}_0\|_{s-\lambda,\eta_{**},T_{**}}\le C\epsilon_0$. 

\section{Appendix}
\subsection{Interior coefficients}
By direct computation, we can determine other coefficients.
\begin{align}
\tilde{a}_{00}&=(\partial_{y_1}u)^2.\label{A.1}\\
\tilde{a}_{01}&=\tilde{a}_{10}=\partial_{y_1}u(-\partial_{y_0}u+\partial_{x_1}\Phi(\partial_{x_1}\mcw\partial_{y_3}u-\partial_{x_1}p\partial_{y_2}u+1))\nonumber\\
 &\quad+\partial_{y_1}u(\partial_{x_2}\Phi(\partial_{x_3}p\partial_{x_2}\mcw\partial_{y_2}u+\partial_{x_2}\mcw\partial_{y_2}u-\partial_{y_2}u)
-\partial_{x_3}\Phi(\partial_{x_3}p\partial_{y_2}u+\partial_{y_3}u)).\\
\tilde{a}_{02}&=\tilde{a}_{20}=(\partial_{y_1}u)^2(\partial_{x_1}\Phi\partial_{x_1}p+\partial_{x_2}\Phi(\partial_{x_2}p+1)+\partial_{x_3}\Phi\partial_{x_3}p).\\
\tilde{a}_{12}&=\tilde{a}_{21}=-\partial_{y_0}u(\partial_{x_1}\Phi\partial_{x_1}p\partial_{y_1}u+\partial_{x_2}\Phi(\partial_{x_2}p+1)\partial_{y_1}u)\nonumber\\
&\quad+\partial_{y_1}u(\partial_{x_1}\mcw\partial_{y_3}u-\partial_{x_1}p\partial_{y_2}u +1)\nonumber\\
&\qquad\qquad\times (\partial_{x_1}p(-c^2+|\partial_{x_1}\Phi|^2)+\partial_{x_1}\Phi\partial_{x_2}\Phi(\partial_{x_2}p+1)+\partial_{x_1}\Phi\partial_{x_3}\Phi\partial_{x_3}p).\\
\tilde{a}_{22}&=\partial_{x_1}p\partial_{y_1}u((-c^2+|\partial_{x_1}\Phi|^2)\partial_{x_1}p\partial_{y_1}u+\partial_{x_1}\Phi\partial_{x_2}\Phi(\partial_{x_2}p+1)\partial_{y_1}u)\nonumber\\
&+(\partial_{y_1}u)^2(\partial_{x_2}p+1)(\partial_{x_2}\Phi\partial_{x_1}\Phi\partial_{x_1}p+(-c^2+|\partial_{x_2}\Phi|^2)(\partial_{x_2}p+1)+\partial_{x_2}\Phi\partial_{x_3}\Phi\partial_{x_3}p)\nonumber\\
&+\partial_{x_3}p(\partial_{y_1}u)^2(\partial_{x_3}\Phi\partial_{x_1}\Phi\partial_{x_1}p+\partial_{x_3}\Phi\partial_{x_2}\Phi(\partial_{x_2}p+1)+(-c^2+|\partial_{x_3}\Phi|^2)\partial_{x_3}p).\\
\tilde{a}_{11}&=-\partial_{y_0}u(-\partial_{y_0}u+\partial_{x_1}\Phi(\partial_{x_1}\mcw\partial_{y_3}u-\partial_{x_1}p\partial_{y_2}u+1)\nonumber\\
&+\partial_{x_2}\Phi((\partial_{x_3}p\partial_{x_2}\mcw-1)\partial_{y_2}u)+\partial_{x_2}\mcw\partial_{y_3}u)\nonumber\\
&+(\partial_{x_1}\mcw\partial_{y_3}u-\partial_{x_1}p\partial_{y_2}u+1)\nonumber\\
&\qquad\qquad\qquad\qquad\times(-\partial_{x_1}\Phi\partial_{y_0}u+(-c^2+|\partial_{x_1}\Phi|^2)(\partial_{x_1}\mcw\partial_{y_3}u-\partial_{x_1}p\partial_{y_2}u+1))\nonumber\\
&+\partial_{x_3}\Phi\partial_{y_0}u(\partial_{x_3}p\partial_{y_2}u+\partial_{y_3}u)+\partial_{x_1}\Phi\partial_{x_2}\Phi(\partial_{x_1}\mcw\partial_{y_3}u-\partial_{x_1}p\partial_{y_2}u+1)\nonumber\\
&\qquad\qquad\qquad\qquad\times((\partial_{x_3}p\partial_{x_2}\mcw-1)\partial_{y_2}u+\partial_{x_2}\mcw\partial_{y_3}u)\nonumber\\
&-(\partial_{x_1}\mcw\partial_{y_3}u-\partial_{x_1}p\partial_{y_2}u+1)\partial_{x_1}\Phi\partial_{x_3}\Phi(\partial_{x_3}p\partial_{y_2}u+\partial_{y_3}u)\nonumber\\
&+((\partial_{x_3}p\partial_{x_2}\mcw-1)\partial_{y_2}u+\partial_{x_2}\mcw\partial_{y_3}u)\nonumber\\
&\qquad\qquad\qquad\qquad\times(-\partial_{x_2}\Phi\partial_{y_0}u+\partial_{x_1}\Phi\partial_{x_2}\Phi(\partial_{x_1}\mcw\partial_{y_3}u-\partial_{x_1}p\partial_{y_2}u+1))\nonumber\\
&+(-c^2+|\partial_{x_2}\Phi|^2)((\partial_{x_3}p\partial_{x_2}\mcw-1)\partial_{y_2}u+1)((\partial_{x_3}p\partial_{x_2}\mcw-1)\partial_{y_2}u+\partial_{x_2}\mcw\partial_{y_3}u)\nonumber\\
&-((\partial_{x_3}p\partial_{x_2}\mcw-1)\partial_{y_2}u+\partial_{x_2}\mcw\partial_{y_3}u)\partial_{x_2}\Phi\partial_{x_3}\Phi(\partial_{x_3}p\partial_{y_2}u+\partial_{y_3}u)\nonumber\\
&-(\partial_{x_3}p\partial_{y_2}u+\partial_{y_3}u)(-\partial_{x_3}\Phi\partial_{y_0}u+\partial_{x_3}\Phi\partial_{x_1}\Phi(\partial_{x_1}\mcw\partial_{y_3}u-\partial_{x_1}p\partial_{y_2}u+1))\nonumber\\
&-(\partial_{x_3}p\partial_{y_2}u+\partial_{y_3}u)\partial_{x_3}\Phi\partial_{x_2}\Phi((\partial_{x_3}p\partial_{x_2}\mcw-1)\partial_{y_2}u+\partial_{x_2}\mcw\partial_{y_3}u)\nonumber\\
&+(\partial_{x_3}p\partial_{y_2}u+\partial_{y_3}u)^2(-c^2+|\partial_{x_3}\Phi|^2).\\
\tilde{a}_{33}&=\partial_{x_1}p(\partial_{y_1}u)^2((-c^2+|\partial_{x_1}\Phi|^2)\partial_{x_1}p+\partial_{x_1}\Phi\partial_{x_2}\Phi(\partial_{x_2}p+1)+\partial_{x_1}\Phi\partial_{x_3}\Phi\partial_{x_3}p)\nonumber\\
&+(\partial_{x_2}p+1)(\partial_{y_1}u)^2(\partial_{x_2}\Phi\partial_{x_1}\Phi\partial_{x_1}p+(-c^2+|\partial_{x_2}\Phi|^2)(\partial_{x_2}p+1)+\partial_{x_2}\Phi\partial_{x_3}\Phi\partial_{x_3}p)\nonumber\\
&+\partial_{x_3}p(\partial_{y_1}u)^2(\partial_{x_3}\Phi\partial_{x_1}\Phi\partial_{x_1}p+\partial_{x_3}\Phi\partial_{x_2}\Phi(\partial_{x_2}p+1)+(-c^2+|\partial_{x_3}\Phi|^2)\partial_{x_3}p).\label{A.7}
\end{align}

%\section{Linearisation of $G$}\label{appendix B}
%We only need exact formula of $b_3$. In fact we have
%\begin{align}
%    b_3=
%\end{align}
\subsection{Second order derivatives of $\kappa$}
The second order derivatives of $\kappa$ can be computed via chain rule on the basis of the first order derivatives.
\begin{align}
    \partial_{\bar{u}\bar{u}}\kappa&=\partial_{x_1}(\partial_{\bar{u}}\kappa)\frac{\partial x_1}{\partial u}\partial_{\bar{u}}\kappa+\partial_{x_2}(\partial_{\bar{u}}\kappa)\frac{\partial x_2}{\partial u}\partial_{\bar{u}}\kappa\\
    \kappa_{\bar{u}y_2}&=\partial_{x_1}(\partial_{\bar{u}}\kappa)\frac{\partial x_1}{\partial y_2}+\partial_{x_2}(\partial_{\bar{u}}\kappa)\frac{\partial x_2}{\partial y_2}\\
    \kappa_{\bar{u}y_3}&=\frac{\partial_{x_2}N(1+y_3(\partial_{x_1}N)+\partial_{x_2}N)+(\partial_{x_1}N+\partial_{x_2}N)(1+y_3\partial_{x_2}N)}{(1+y_3(\partial_{x_1}N+\partial_{x_2}N))^2}\nonumber\\
    &\quad+\partial_{x_1}(\partial_{\bar{u}}\kappa)\frac{\partial x_1}{\partial u}+\partial_{x_2}(\partial_{\bar{u}}\kappa)\frac{\partial x_2}{\partial u}\\
    \partial_{y_2y_2}\kappa&=\partial_{x_1}(\partial_{y_2}\kappa)\frac{\partial x_1}{\partial y_2}+\partial_{x_2}(\partial_{y_2}\kappa)\frac{\partial x_2}{\partial y_2}\\
    \partial_{y_2y_3}\kappa&=\partial_{x_1}\partial_{y_3}\kappa\frac{\partial x_1}{\partial y_2}+\partial_{x_2}\partial_{y_3}\kappa\frac{\partial x_2}{\partial y_2}\\
    \partial_{y_3y_3}\kappa&=\partial_{x_1}\partial_{y_3}\kappa\frac{\partial x_1}{\partial y_3}+\partial_{x_2}\partial_{y_3}\kappa\frac{\partial x_2}{\partial y_3}+\frac{N(\partial_{x_1}N+\partial_{x_2}N)}{(1+y_3(\partial_{x_1}N+\partial_{x_2}N))^2}
\end{align}
where 
\begin{align}
    \frac{\partial x_1}{\partial u}&=1\\
    \frac{\partial x_2}{\partial u}&=-\frac{y_3\partial_{x_1}N}{1+y_3\partial_{x_2}N}
    \end{align}
%    \frac{\partial x_1}{\partial y_2}&=\\
%    \frac{\partial x_1}{\partial y_3}&=\\
By calculating the inverse of $\mathbf{J}$, one can derive $\frac{\partial x_1}{\partial y_i}$ $(i=2,3)$ and $\frac{\partial x_2}{\partial y_i}$ $(i=2,3)$. For example, one has
\begin{align}
    \frac{\partial x_2}{\partial y_2}&=\frac{1}{1+y_3\partial_{x_2}N},\\
    \frac{\partial x_2}{\partial y_3}&=-\frac{N}{1+y_3\partial_{x_2}N}.
\end{align}
And $\partial_{\bar{u}}\kappa$, $\partial_{y_2}\kappa$ and $\partial_{y_3}\kappa$ are given in \eqref{4-2}.

\medskip

\noindent
{\bf Acknowledgements}.
The research of Beixiang Fang was supported in part by NSFC Grant Nos. 11971308 and 11631008. 
The research of Feimin Huang was supported in part by NSFC Grant No. 11688101.
The research of Wei Xiang was supported in part by the Research Grants Council of the HKSAR, China
(Project No.CityU 11303518, Project CityU 11304820 and Project CityU 11300021). The research of Feng Xiao was supported by the National Center for Mathematics and Interdisciplinary Sciences, CAS.

\end{document}

%% file: Denotations_Math_Fonts.tex
\newcommand{\mcb}{\mathcal{B}}

\newcommand{\mcd}{\mathcal{D}}

\newcommand{\mch}{\mathcal{H}}

\newcommand{\mcn}{\mathcal{N}}

\newcommand{\mcp}{\mathcal{P}}
\newcommand{\mcq}{\mathcal{Q}}

\newcommand{\mcw}{\mathcal{W}}

\newcommand{\mfp}{\mathscr{P}}

\newcommand{\mbx}{\mathbf x}
\newcommand{\mby}{\mathbf y}